\documentclass[twoside,11pt]{article}

\usepackage{jmlr2e}

\usepackage{amsmath,amssymb}

\usepackage{tikz,ifthen} 

\usepackage{bbold}
\usepackage{enumerate}

\usepackage{placeins}

\usepackage[bbgreekl]{mathbbol}
\usepackage{amsfonts}
\DeclareSymbolFontAlphabet{\mathbbl}{AMSb}
\DeclareSymbolFontAlphabet{\mathbb}{bbold}

\newcommand{\R}{\mathbb{R}}
\newcommand{\V}{\mathcal{V}}
\newcommand{\F}{\mathcal{F}}
\newcommand{\E}{\mathcal{E}}

\newcommand{\VV}{\mathbb{V}}
\newcommand{\EE}{\mathbb{E}}
\newcommand{\TT}{\mathbb{T}}
\newcommand{\LL}{\mathbb{L}}

\renewcommand{\AA}{\mathbb{A}}
\newcommand{\PP}{\mathbb{P}}
\newcommand{\dd}{\mathbb{d}}
\newcommand{\WW}{\mathbb{W}}
\newcommand{\CC}{\mathbb{C}}

\newcommand{\RR}{\mathbb{R}}
\newcommand{\DD}{\mathbb{D}}
\newcommand{\bb}{\mathbb{b}}
\newcommand{\pp}{\mathbb{p}}

\newcommand{\xx}{\mathbb{x}}
\newcommand{\yy}{\mathbb{y}}
\newcommand{\vv}{\mathbb{v}}

\newcommand{\ww}{\mathbb{w}}
\newcommand{\zz}{\mathbb{z}}
\newcommand{\ee}{\mathbb{e}}
\newcommand{\ff}{\mathbb{f}}

\usepackage[ruled,vlined]{algorithm2e}

\usepackage{tikz}
\usetikzlibrary{positioning}

\usepackage{subcaption}

\ShortHeadings{A new approach for Laplacian solvers and flow problems}{Rebeschini and Tatikonda}
\firstpageno{1}

\usepackage{lastpage}
\jmlrheading{20}{2019}{1-\pageref{LastPage}}{5/17; Revised
12/18}{2/19}{17-286}{Patrick Rebeschini and Sekhar Tatikonda}
\ShortHeadings{A New Approach to Laplacian Solvers and Flow Problems}{Rebeschini and Tatikonda}

\begin{document}

\title{A New Approach to Laplacian Solvers and Flow Problems}

\author{\name Patrick Rebeschini \email patrick.rebeschini@stats.ox.ac.uk\\
       \addr Department of Statistics\\
       University of Oxford\\
       Oxford, OX1 3LB, UK
       \AND
       \name Sekhar Tatikonda \email sekhar.tatikonda@yale.edu \\
       \addr Department of Statistics and Data Science\\
       Yale University\\
       New Haven, CT 06511, USA}

\editor{}

\maketitle

\thispagestyle{plain}

\begin{abstract}%
This paper investigates the behavior of the Min-Sum message passing scheme to solve systems of linear equations in the Laplacian matrices of graphs and to compute electric flows. Voltage and flow problems involve the minimization of quadratic functions and are fundamental primitives that arise in several domains.
Algorithms that have been proposed are typically centralized and involve multiple graph-theoretic constructions or sampling mechanisms that make them difficult to implement and analyze. On the other hand, message passing routines are distributed, simple, and easy to implement. In this paper we establish a framework to analyze Min-Sum to solve voltage and flow problems. We characterize the error committed by the algorithm on general weighted graphs in terms of hitting times of random walks defined on the computation trees that support the operations of the algorithms with time. For $d$-regular graphs with equal weights, we show that the convergence of the algorithms is controlled by the total variation distance between the distributions of non-backtracking random walks defined on the original graph that start from neighboring nodes.
The framework that we introduce extends the analysis of Min-Sum to settings where the contraction arguments previously considered in the literature (based on the assumption of walk summability or scaled diagonal dominance) can not be used, possibly in the presence of constraints.
\end{abstract}


\section{Introduction}\label{sec:introduction}

\underline{Voltage and flow problems.} Laplacian matrices of graphs are fundamental. Solving linear systems of equations in these matrices constitutes the back-bone of many algorithms with applications to a wide variety of fields.
See \cite{S11,V13}. 
When a weighted graph is viewed as an electric circuit, where weighted edges represent resistors and vertices represent nodes to which some external current is applied, solving linear systems of equations in the Laplacian matrix of the graph corresponds to computing the voltages induced to each node by the external currents. For this reason, we refer to the Laplacian problem as the ``voltage problem." Along with the problem of computing the voltages is the problem of computing the current induced across each resistor, which we refer to as the ``flow problem." The two problems are related by Ohm's law---which says that electric flows are induced by voltage differences across the nodes---and they both can be cast in the form of quadratic optimization problems. The voltage problem is an unconstrained problem to find the voltage assignment that maximizes the quadratic Laplacian form. The flow problem is a constrained problem to find the flow assignment that minimizes the circuit energy and satisfies Kirchhoff's law. The voltage problem is the dual of the flow problem. 

\underline{Previous literature on Laplacian solvers.} The extensive literature on approximately solving the voltage and flow problems typically involves algorithms that rely on multiple graph-theoretic constructions or sampling mechanisms which make them difficult to implement and analyze. The first quasi-linear time solver was given in \cite{Spielman:2004,ST14}, and it achieves a running time of $\tilde O(m\log^c n \log 1/\varepsilon)$ in the Laplacian-modified $\ell_2$ norm. Here, $n$ and $m$ are, respectively, the number of vertices and edges in the graph, $c$ is a constant close to $100$, and the $\tilde O$ notation hides $\log\log n$ factors.
Since this result, many algorithms have been proposed and showed to yield smaller values of $c$. The work in \cite{KMP11} has reduced the constant $c$ to $1$. Currently, the best running time is in \cite{Cohen:2014:SSL:2591796.2591833} where $c=1/2$. These algorithms all rely on the same general architecture, which combines the Chebyshev method or the conjugate gradient algorithm with the use of recursive preconditioning, spectral sparsifiers, low-stretch spanning trees, and expander graphs.
Despite their theoretical guarantees, these algorithms present noticeable difficulties to practitioners due to the involved machinery they rely upon.
Lately, new algorithms have been proposed that depart from this framework and use fewer graph-theoretic constructions. The method in \cite{KOSZ13}, for instance, which operates in the flow domain and yields a fast solver with $c=2$, using very little of the previous machinery---although it still requires low-stretch spanning trees.
More recently, the work in \cite{KS16}, which relies on sparse Gaussian elimination for Laplacian matrices to yield a fast solver with $c=3$, without the use of any graph-theoretic construction. While more appealing, these new algorithms still rely on random sampling so the convergence analysis is only performed in expectation, or with high probability.
Moreover, all these algorithms are centralized. There has also been some work in parallelizing these solvers, with the first parallel solver with near linear work and poly-logarithmic depth presented in \cite{PS14parallel}. These solvers work in the shared memory model and are not completely distributed.

\underline{Min-Sum algorithm.} The aforementioned limitations prompted us to investigate the behavior of the Min-Sum algorithm to solve both the voltage and the flow problem. Min-Sum is a distributed, deterministic (i.e., not randomized), general-purpose message passing algorithm that is used to optimize objective functions that are sums of component functions supported on a given graph.
Min-Sum has emerged as a canonical procedure to address large scale problems in a variety of domains.
Its success comes from the fact that the algorithm is simple and easy to implement, require minimal data structures and low communication per iteration, and in many applications it has been shown to converge to the problem solution. Interest in Min-Sum has been triggered by its success in solving certain classes of challenging combinatorial optimization problems, such as decoding low-density parity-check codes, see \cite{BGT93,910577} for instance, and solving certain type of satisfiability problems, see \cite{MPZ02,BMZ05}, for instance.

Despite its successes, the convergence analysis of the Min-Sum algorithm remains limited.
For quadratic optimization problems and, more generally, convex optimization problems, the Min-Sum algorithm has been shown to converge to the problem solution for \emph{any} graph topology under the assumption that the objective function is walk-summable \citep{MJW06}, or, equivalently, scaled diagonal dominant \citep{MVR10}.
These conditions allow the use of contraction arguments to establish exponential convergence rates. These conditions do not apply to quadratic objective functions in Laplacian matrices, hence to the voltage problem,
and they do not apply to constrained optimization problems, hence to the flow problem.
At the same time, on the one hand quadratic approximations of the Min-Sum algorithm leading to the Approximate Message Passing (AMP) algorithm are becoming increasingly popular in high-dimensional statistics and machine learning \citep{Donoho18914,montanari_2012, OptAMP18}. The AMP algorithm is typically analysed from an asymptotical point of view and this analysis is \emph{not} based on contraction arguments. The recent success of the AMP algorithm prompts for improving the understanding of the convergence and correctness of the Min-Sum algorithm directly applied to quadratic models, where previously-considered contraction-based assumptions do not hold. On the other hand, while some work has been done to study the behavior of Min-Sum in constrained settings, this work seems to have focused on linear objective functions and a general theory is missing. See \cite{GSW12} and references therein. 

\underline{Results.} This paper investigates the convergence behavior of the Min-Sum algorithm applied to the voltage and flow problems and establishes a general framework for the analysis of message-passing algorithms that can accommodate for non contraction-based assumptions and for constrained-based optimization problems.
The algorithms that we analyze---Algorithm \ref{alg:Min-Sum voltages, quadratic messages} for the voltage problem and Algorithm \ref{alg:Min-Sum, quadratic messages, no leaves} for the flow problem---are easy to implement, with a pseudo-code that fits two lines.
We investigate the behavior of these algorithms on general weighted graphs, and we show that the error they incur can be characterized in terms of hitting times of ordinary diffusion random walks defined on the computation trees that are obtained by unraveling the operations of the algorithms with time \citep{WF01}. We specialize the analysis to the  case of $d$-regular graphs. When $d=2$, we show that the flow algorithm converges to the true solution in the $\ell_\infty$ norm with rate $O(1/t)$, where $t$ is the iteration time, independently of the graph dimensions $n$ and $m$, yielding a $\varepsilon$-approximate solution with a running time $O(m/\varepsilon)$, linear in $m$ (Corollary \ref{cor:cycle conv}). On the other hand, when $d=2$, we show that the voltage algorithm does \emph{not} converge, as it keeps oscillating (Corollary \ref{cor:not convergence}).
For graphs with equal weights, when $d\ge 3$, we show that the error both algorithms commit is characterized by quantities involving the difference of non-backtracking random walks on the original graph that originate from neighbor nodes (Lemma \ref{lem:regulargraph}).
We show that the convergence rate of the algorithms (with respect to various norms) can be controlled by uniform bounds on the total variation distance between the distributions of these walks. For graphs where these bounds decay exponentially $O(e^{-t})$ (such as complete graphs, as we show numerically) or polynomially $O(1/t^\beta)$ with $\beta>0$ (such as cycles and tori, where $\beta=1/2$ as we show numerically), possibly under additional assumptions on the problem instance (depending on the norm chosen for the analysis), the Min-Sum algorithm yield $\varepsilon$-approximate solutions for both voltage and flow problem with time $O(m \log 1/\varepsilon)$ or $O(m /\varepsilon^{1/\beta})$, respectively, linear in $m$ (Theorem \ref{thm:runningtime} for voltages, and Theorem \ref{thm:runningtime flow} for flows).

\underline{Novelty in the analysis of message passing.} Our analysis of the Min-Sum algorithm presents two key contributions that go beyond the specifics of the problems here addressed.
\begin{enumerate}
\item \underline{Beyond contraction arguments.} Typically, the convergence analysis of the Min-Sum scheme exploits a decay of correlation property in the computation tree established upon imposing \emph{local} assumptions on the original problem, such as the assumption of scaled diagonal dominance. 
These conditions essentially yield a bound on the local interactions in the computation tree by a quantity strictly less than one. In this case, a contraction argument can be performed at each level of the tree to establish an \emph{exponential} decay of correlation from root to leaves, which translates into an exponential convergence for the algorithm.
On the other hand, the analysis of message passing that we present in this paper is based on a \emph{global} analysis of the problem supported on the computation tree, via the analogy with electric circuits and the characterization with hitting times of random walks, 
which is not captured by contraction arguments based on worst-case bounds for local interactions. In fact, as mentioned above, for some graphs this analysis yields a convergence \emph{polynomial} in time, not exponential.
\vspace{-.1cm}
\item \underline{Beyond unconstrained optimization.} Typically, the convergence analysis of the Min-Sum algorithm relies on finding the fix point(s) of the message updates. For unconstrained optimization problems, the key insight to find fix point(s) is that the first order optimality conditions for the original problem are respected by the problem supported on the computation tree, for every node that is not a leaf \citep{WF01,MVR10}. The convergence analysis that we present for the flow problem---a constrained problem---is based on the observation that the Karush-Kuhn-Tucker (KKT) optimality conditions for constrained problems are respected on the computation tree, for every non-leaf node. This insight allows to extend the analysis of the Min-Sum algorithm to more general constrained optimization problems.
\end{enumerate}

\underline{Structure of the paper.} The structure of the paper is as follows. In Section \ref{sec:Voltage and flow problem} we define the voltage and flow problems as optimization routines. In Section \ref{sec:Min-Sum algorithm} we introduce the general-purpose Min-Sum algorithm, and we specialize it to the voltage and flow problem, respectively. In Section \ref{sec:results} we present the error characterization for the Min-Sum algorithms for $d$-regular graphs, and we present the main results on their convergence. In Section \ref{sec:proofs} we develop the general framework to investigate the Min-Sum algorithm, yielding the characterization of the error in terms of hitting times of random walks on the computation trees, and we present the proofs of the results given in the previous section. In Appendix \ref{sec:Laplacians and random walks} we establish a general connection between the inverse of restricted Laplacian matrices and hitting times of ordinary diffusion random walks on graphs---which is the main computational tool used in our analysis. Technical results are in Appendix \ref{app:technical}.
\begin{remark}[Notation]
\label{rem:notation}
Given a matrix $M$, we denote by $M^T$ its transpose, by $M^{-1}$ its inverse, and by $M^{+}$ its Moore-Penrose pseudoinverse. 
Given a set $\mathcal{I}$ and $i\in \mathcal{I}$, we denote by $i$ both the element $i\in \mathcal{I}$ and the subset $\{i\}\subseteq \mathcal{I}$ (e.g., we write $\mathcal{I}\setminus i$ to mean $\mathcal{I}\setminus \{i\}$). Given a function $f:\R^{\mathcal{I}}\rightarrow \R$, $i\in\mathcal{I}$, we write $f(x_{\mathcal{I}\setminus i}\,\cdot)$ to mean the function $z\in\R \rightarrow f(x_{\mathcal{I}\setminus i}z)$. We adopt the following taxonomy of notation:
\vspace{-.1cm}
\begin{itemize}
\setlength\itemsep{-.17em}
\item the arrow notation $\vec {\color{white}\cdot}$ is used to either refer to quantities related to directed graphs or to refer to probability distributions related to non-backtracking random walks;
\item the hat notation $\hat{\color{white}\cdot}$ is used to refer to the output of algorithms;
\item the bar notation $\bar{\color{white}\cdot}$ is for quantities related to graphs where a node has been removed;
\item the double-struck notation is for quantities related to computation trees, cf.\ Section \ref{sec:proofs}.
\end{itemize}

\end{remark}

\section{Voltage and Flow Problems}\label{sec:Voltage and flow problem}

Throughout, let us consider a simple (i.e., no self-loops, and no multiple edges), undirected, weighted graph $G$ given by the triple $(V,E,W)$, where $V$ is a set of $n$ vertices, $E$ is a set of $m$ edges, and $W\in\R^{V\times V}$ is a symmetric matrix that assigns a positive weight to every edge, namely, $W_{vw}>0$ if $\{v,w\}\in E$, and $W_{vw}=0$ if $\{v,w\}\not\in E$. The matrix $W$ is called the weighted \emph{adjacency matrix} of the weighed graph $G$. Define the weighted degree of a vertex $v$ by $d_v:=\sum_{w\in V} W_{vw}$, and let $D\in\R^{V\times V}$ be the diagonal matrix defined by $D_{vv}:=d_v$. The Laplacian $L$ of the graph $G$ is the matrix defined as $L:=D-W$. Without loss of generality, in what follows we assume that the graph $G$ is connected, otherwise we can treat each connected component on its own. It can be verified that in this case the null space of the Laplacian matrix $L$ is spanned by the all-ones vector $1\in\R^V$, so that the range of $L$ is spanned by the vectors orthogonal to $1$.

\underline{Voltage problem.} We are interested in solving linear equations in the Laplacian matrix $L$. That is, given $b$ in the range of $L$, i.e., $1^Tb=0$, our goal is to find $\nu$ that satisfies
\begin{align}
	L \nu = b.
	\label{system}
\end{align}
As it can immediately be checked from first order optimality conditions, the set of solutions to \eqref{system} coincides with the set of solutions to the ``voltage problem," namely:
\begin{align}
\begin{aligned}
	\text{maximize }\quad   & - \frac{1}{2} \nu^T L \nu + b^T \nu.
\end{aligned}
\label{dual}
\end{align}
We are interested in the unique solution that is orthogonal to the all one vector, namely, $\nu^\star = L^+ b$, where $L^{+}$ is the Moore-Penrose pseudoinverse of $L$. 
Given an accuracy parameter $\varepsilon>0$, the fast solvers mentioned in the introduction achieve a quasi-linear running time $\tilde O(m\log^c n \log 1/\varepsilon)$ to return an estimate $\hat\nu$ that satisfies
$
	\| \nu^\star - \hat\nu \|_L \le \varepsilon \| \nu^\star \|_L,
$
where $\|\nu\|_L := \sqrt{\nu^T L \nu}$ is the Laplacian-modified $\ell_2$ norm, also known as $L$-norm. It can be seen that $\|\,\cdot\,\|_L$ is a pseudo-norm, as $\|\nu\|_L=0$ for any vector $\nu$ orthogonal to the all one vector. In many applications where Laplacian solvers are needed, the $L$-norm is the natural norm to run the analysis. See \cite{V13}, for instance. The $L$-norm is related to the Euclidean norm as follow: $\| \nu \|_L \le \kappa \| \nu \|_2$ and $\| \nu \|_2 \le \kappa \| \nu \|_L$, where $\kappa$ is the ratio between the the maximum and minimum non-zero eigenvalues of $L$. 

\underline{Flow problem.} 
To introduce the flow problem, let us fix an arbitrary orientation of the edges in $E$. Denote this set of directed edges by $\vec E$, and let $\vec G= (V, \vec E)$ be the corresponding directed graph. Let $A\in\R^{V\times \vec E}$ be the \emph{signed vertex-edge incidence} matrix of the graph $\vec G$, defined as $A_{ve}:=1$ if edge $e$ leaves node $v$, $A_{ve}:=-1$ if edge $e$ enters node $v$, $A_{ve}:=0$ otherwise.
To each directed edge $e=(v,w)$ let $R\in\R^{\vec E\times \vec E}$ be the diagonal matrix defined by $R_{ee}:=1/W_{vw}$. It is easy to verify that the Laplacian $L$ of the undirected graph $G$ can be expressed as $L=AR^{-1} A^T$. Note that while $A$ depends on the choice of the direction of the edges in $\vec E$, $L$ does not. The ``flow problem" is defined as follows:
\begin{align}
\begin{aligned}
	\text{minimize }\quad   & \frac{1}{2} x^T R x\\
	\text{subject to }\quad & Ax = b.
\end{aligned}
\label{primal}
\end{align}
This problem can be interpreted as a minimal-energy electrical network problem. To each directed edge $e\in \vec E$ is associated a resistor of value $R_{ee}$ (analogously, a conductor of value $1/R_{ee}$) through which a current $x_e$ flows, with the convention that $x_e>0$ if the current is in the direction of the edge, $x_e<0$ if the current is in the direction opposite the edge. The energy of having current $x_e$ on edge $e$ is given by the $\frac{1}{2}R_{ee} x_e^2$. For each node $v\in V$, $b_v$ represents a given external current: $b_v>0$ represents a source where the current enters node $v$, whereas $b_v<0$ represents a sink where the current exits node $v$. The constraint equations $Ax=b$ embody Kirchhoff's conservation law, by which at each vertex the net sum of the incoming and outgoing internal flows equals the external flow.

\underline{Connection between voltage and flow problems.} From the KKT optimality conditions it can be shown (see Proposition \ref{prop:quadratic} in Appendix \ref{app:technical}) that the unique solution of problem \eqref{primal} reads
$
	x^\star = R^{-1}A^T L^+ b = R^{-1}A^T \nu^\star.
$
That is, if $e=(v,w)\in\vec E$, we have $x^\star_e = W_{vw} (\nu^\star_v - \nu^\star_w)$ which represents Ohm's law.
Finally, notice that problem \eqref{dual} is the dual of \eqref{primal}. Recall that the Lagrangian of \eqref{primal} is the function $\mathcal{L}$ from $\mathbb{R}^{\vec E}\times\mathbb{R}^V$ to $\mathbb{R}$ defined as
\begin{align}
	\mathcal{L}(x,\nu) := \frac{1}{2} x^T R x + \nu^T (b - A x).
	\label{lagrangian-primal}
\end{align}
The dual function is
$
	q(\nu) 
	= \min_{x\in\mathbb{R}^{\vec E}} \mathcal{L}(x,\nu)
	= 
	- \frac{1}{2} \nu^T L \nu + b^T \nu,
$
and the dual problem reads $\max_{\nu\in\R^V} q(\nu)$, which is problem \eqref{dual}.

\section{Min-Sum Algorithm}\label{sec:Min-Sum algorithm}

The Min-Sum algorithm is a distributed routine to optimize a cost function that has a graph structure. 
Let $\V$ and $\F$ be two finite sets, whose elements are respectively referred to as \emph{variables} and \emph{factors}, to be interpreted as two groups of vertices in a certain undirected bipartite graph with edge set $\E$, where each edge in $\E$ connects an element of $\V$ with an element of $\F$. For any variable node $i\in\V$, we use the notation $\partial i$ to denote the set of neighbors of vertex $i$ in the factor graph. Analogously, for any factor node $a\in\F$, we use the notation $\partial a$ to denote the set of neighbors of vertex $a$ in the factor graph. Consider the following optimization problem over $x\in\R^\V$:
\begin{align}
	\text{minimize }\quad   & g(x) := 
	\sum_{i\in \V} g_{i}(x_{i}) +
	\sum_{a\in \F} g_{a}(x_{\partial a}),
	\label{cost function Min-Sum}
\end{align}
for a given set of functions $g_{i}: \R \rightarrow \R \cup \{\infty\}$, for each $i\in\V$, and $g_{a}: \R^{\partial a} \rightarrow \R \cup \{\infty\}$, for each $a\in\F$.
The Min-Sum algorithm is an iterative algorithm that proceeds by updating a set of \emph{messages} associated to each edge in the underlying factor graph. Messages are functions of the optimization variables $x_i$, for $i\in\V$, and each edge in the factor graph is associated with two messages at any given iteration of the algorithm. For any time $s\ge 0$, and $i\in\V,a\in\F$ so that $\{i,a\}\in\E$, $\mu^{s}_{i\rightarrow a}:\R\rightarrow\R \cup \{\infty\}$ represents a message from variable to factor, and $\mu^{s}_{a\rightarrow i}:\R\rightarrow\R \cup \{\infty\}$ represents a message from factor to variable.

In this work we consider the \emph{synchronous} implementation of the Min-Sum algorithm, where at each iteration all messages are updated.
For a given iteration step $t\ge 1$, the algorithm reads as in Algorithm \ref{alg:Min-Sum}.
\begin{algorithm}
	\DontPrintSemicolon
	\KwIn{Initial messages $\{\mu^{0}_{a\rightarrow i}\}$, for each $\{i,a\}\in\E$;}
	\For{$s \in \{1,\ldots,t\}$}{
		Compute, for each $\{i,a\}\in\E$,\\
		$\mu^{s}_{i\rightarrow a} = 
		g_i + \sum_{b\in\partial i \setminus a} \mu^{s-1}_{b\rightarrow i},
		\qquad
		\mu^{s}_{a\rightarrow i} =
		\min_{x\in\R^{\V}}
		\{
		g_a(x_{\partial a\setminus i}\, \cdot) 
		+ \sum_{j\in\partial a \setminus i} \mu^{s}_{j\rightarrow a}(x_j)
		\}$;
	}
	Compute $\mu^{t}_{i} = g_i + \sum_{a\in\partial i} \mu^{t}_{a\rightarrow i}$, for each $i\in\V$;\\
	\KwOut{$\hat x_i^t := \arg\min_{z\in\R} \mu^{t}_{i}(z)$, for each $i\in\V$.}
	\caption{Min-Sum}
	\label{alg:Min-Sum}
\end{algorithm}

The iterations of the Min-Sum algorithm correspond to the dynamic programming iterations to solve \eqref{cost function Min-Sum} when the underlying graph is a tree. In general, the Min-Sum algorithm is not guaranteed to converge, and even when it converges it is not guaranteed that it converges to an optimal point of problem \eqref{cost function Min-Sum}.
The Min-Sum algorithm is suited for cooperative distributed multi-agents optimization. In this framework, to each variable $i\in \V$ is associated an agent who has access only to the local objective functions $g_i$ and $g_a$, for each $a\in\partial i$. At each time step, each agent can only exchange information with his near-neighbors in the factor graph, i.e., the agents who share at least one factor with him. Agents' goal is to cooperate over time to solve problem \eqref{cost function Min-Sum}. The Min-Sum algorithm is a mean to achieve this goal: iteratively, each agent $i$ updates a set of messages $\{\mu^{s}_{i\rightarrow a}\}, a\in\partial i$, using the messages that he has received from his neighbors at the previous time step, i.e., $\{\mu^{s-1}_{j\rightarrow a}\}, a\in\partial j$, for each neighbor $j$ of $i$, and using the locally known functions $g_i$ and $g_a$, for each $a\in\partial i$.

\begin{remark}[Belief Propagation]
The Min-Sum algorithm alternates minimizations and summations to compute the minimum of a sum of functions of the form in \eqref{cost function Min-Sum}. The main property that is used to derive the local updates of this algorithm (on a tree, where the algorithm computes the correct solution) is the fact that the minimum distributes over the summation. Closely related to the Min-Sum algorithm is the Sum-Product algorithm, where the goal is to compute the sum (or integral) of a product of functions of the following form:
\begin{align}
	\prod_{i\in \V} g_{i}(x_{i})
	\prod_{a\in \F} g_{a}(x_{\partial a}).
	\label{BP factorization}
\end{align}
As the sum (or integral) distributes over the product, the Sum-Product algorithm takes the exact same form as Algorithm \ref{alg:Min-Sum} with minima replaced by sums and sums replaced by products. In the machine learning community, the Sum-Product algorithm is usually referred to as Belief Propagation, and a classical application involves the computation of the normalization function of graphical models, where random variables $\{X_v\}_{v\in V}$ have a joint density that factorizes as in \eqref{BP factorization}. While Min-Sum and Sum-Product are formally related, they are known to produce results that are qualitatively different. In particular, while there are sufficient conditions that guarantee the convergence and correctness of Min-Sum in graphs of arbitrary topology, the Sum-Product algorithm is known to yield correct results only on locally tree-like graphs \citep{WF01}.
\end{remark}

We now specialize the Min-Sum algorithm to the voltage and flow problem, respectively.

\subsection{Min-Sum for the Voltage Problem}
\label{sec:Min-Sum for the voltage problem}
In order to apply the Min-Sum algorithm to solve the voltage problem \eqref{dual}, throughout the rest of this paper we consider the equivalent minimization form:
$
	\text{minimize }\  \frac{1}{2} \nu^T L \nu - b^T \nu.
$
We rewrite this problem in the form \eqref{cost function Min-Sum}.\footnote{It is well known that different ways of mapping the original optimization problem to the form \eqref{cost function Min-Sum} yield different convergence results. See \cite{MVR09}, for instance.} To this end, let us identify $\V \leftrightarrow V$ and $\F \leftrightarrow E$. In this context, the notation $\partial v$ indicates the set of edges in $E$ are connected to $v\in V$; analogously, the notation $\partial e$ indicates the pair of vertices that are connected by edge $e$. For each $v\in V$, define the function $g_v(\nu_v) = -b_v \nu_v$, and for each
$e\in E$ define the function
$
	g_{e}(\nu_{\partial e}) :=
	\frac{1}{2} W_{vw} (\nu_v-\nu_w)^2,
$
where $\partial e = \{v,w\}$.
Recalling that $\nu^T L \nu = \sum_{\{v,w\}\in E} W_{vw} (\nu_v-\nu_w)^2$, the problem above is equivalent to the problem of minimizing the function
$
	\sum_{v\in V} g_{v}(\nu_v)
	+
	\sum_{e\in E} g_{e}(\nu_{\partial e}),
$
which is of the form \eqref{cost function Min-Sum}. As in the factor graph interpretation each factor node (i.e., edge in $E$) is adjacent to exactly two variable nodes (i.e., vertices in $V$), it is convenient to rewrite the Min-Sum algorithm in terms of updates of messages from factors to variables only. At each time step $t\ge 1$, to each edge $e=\{v,w\}\in E$ are associated two messages: $\mu^t_{e\rightarrow v}$ and $\mu^t_{e\rightarrow w}$.
It is easy to verify that in the present setting Algorithm \ref{alg:Min-Sum} becomes Algorithm \ref{alg:Min-Sum voltages}.

\begin{algorithm}
	\DontPrintSemicolon
	\KwIn{Initial messages $\{\mu^{0}_{e\rightarrow v}\}$, for $e\in E, v\in\partial e$;}
	\For{$s \in \{1,\ldots,t\}$}{
		Compute, for each $e\in E$, $v\in\partial e$, with $w=\partial e \setminus v$
		$\mu^{s}_{e\rightarrow v} =
		\min_{\nu_w\in\R}
		\{
		g_w(\nu_w)
		+ g_e(\nu_{w}\, \cdot) 
	 	+ \sum_{f\in\partial w \setminus e} \mu^{s-1}_{f\rightarrow w}(\nu_w)
		\}$;
	}
	Compute $\mu^{t}_{v} = g_v + \sum_{e\in\partial v} \mu^{t}_{e\rightarrow v}$, for each $v\in V$;\\
	\KwOut{$\hat \nu_v^t := \arg\min_{z\in\R} \mu^{t}_{v}(z)$, for each $v\in V$.}
	\caption{Min-Sum, voltage problem}
	\label{alg:Min-Sum voltages}
\end{algorithm}

Algorithm \ref{alg:Min-Sum voltages} applies for any choice of the initial messages $\{\mu^{0}_{e\rightarrow v}\}$, for $e\in E, v\in\partial e$. If we choose initial messages that are quadratic functions, then the messages computed by the algorithm remain quadratic. The following Proposition presents the quadratic updates, without stating formulas for the constant terms as these terms do not influence the output of the Min-Sum algorithm, as we will show next.

\begin{proposition}\label{prop:Min-Sum updates voltages}
For a given $s\ge 1$, for each $e\in E$, $v\in \partial e$ let
$
	\mu^{s-1}_{e\rightarrow v}(z) = 
	\frac{1}{2} 
	W^{s-1}_{e\rightarrow v}z^2 + w^{s-1}_{e\rightarrow v}z
	+ c^{s-1}_{e\rightarrow v}.
$
Then, for each $e\in E$, $v\in\partial e$, we have
$
	\mu^{s}_{e\rightarrow v}(z) =
	\frac{1}{2} 
	W^{s}_{e\rightarrow v}z^2 + w^{s}_{e\rightarrow v}z
	+ c^{s}_{e\rightarrow v},
$
where $\{c^{s}_{e\rightarrow v}\}$ is a certain set of constants, and (here $w=\partial e \setminus v$)
\begin{align*}
		W^{s}_{e\rightarrow v}
		=
		W_{vw}
		\frac{
		\sum_{f\in\partial w \setminus e} 
		W^{s-1}_{f\rightarrow w}}
		{W_{vw} +
		\sum_{f\in\partial w \setminus e} 
		W^{s-1}_{f\rightarrow w}},
		\qquad
		w^{s}_{e\rightarrow v}
		=
		W_{vw}
		\frac{\sum_{f\in\partial w \setminus e} 
		w^{s-1}_{f\rightarrow w}
		- b_w}{W_{vw} +
		\sum_{f\in\partial w \setminus e} 
		W^{s-1}_{f\rightarrow w}}.
\end{align*}
\end{proposition}

\begin{proof}
Fix $s\ge 1$, $e\in E$, $v\in\partial e$, and let $w=\partial e \setminus v$. The Min-Sum update rule reads
$
	\mu^{s}_{e\rightarrow v}(\nu_v) =
	\min_{z\in\R}
	\{
	- b_wz
	+ \frac{1}{2} W_{vw} (\nu_v-z)^2
	+ \frac{1}{2} z^2 \sum_{f\in\partial w \setminus e} 
	W^{s-1}_{f\rightarrow w} + 
	z\sum_{f\in\partial w \setminus e} w^{s-1}_{f\rightarrow w}
	+ \sum_{f\in\partial w \setminus e} c^{s-1}_{f\rightarrow w}
	\}.
$
The solution is
$
	z^\star = 
	(W_{vw}\nu_v + b_w - \sum_{f\in\partial w \setminus e}\! w^{s-1}_{f\rightarrow w})/
	(W_{vw} +
		\sum_{f\in\partial w \setminus e}\!
		W^{s-1}_{f\rightarrow w}),
$
and the proof is immediately concluded by substitution.
\end{proof}

From Proposition \ref{prop:Min-Sum updates voltages} it follows that if we initialize Algorithm \ref{alg:Min-Sum voltages} with quadratic messages of the form
$
	\mu^{0}_{e\rightarrow v}(z) = \frac{1}{2} 
	W^{0}_{e\rightarrow v}z^2 + w^{0}_{e\rightarrow v}z,
$
then, modulo the constant terms, we can track the evolution of the quadratic messages computed by the algorithm by evaluating recursively the set of parameters $\{W^s_{e\rightarrow v}\},\{w^s_{e\rightarrow v}\}$ as prescribed by Proposition \ref{prop:Min-Sum updates voltages}. The belief function associated to vertex $v$ at time $t$ then reads
$
	\mu^{t}_{v}(z) = -b_v z + \sum_{e\in\partial v} 
	(\frac{1}{2} 
	W^{t}_{e\rightarrow v}z^2 + w^{t}_{e\rightarrow v}z
	+ c^{t}_{e\rightarrow v} ),
$
so that
$
	\hat \nu_v^t = \arg\min_{z\in\R} \mu^{t}_{v}(z)
	= (b_v - \sum_{e\in\partial v} 
		w^{t}_{e\rightarrow v})
		/(\sum_{e\in\partial v} 
		W^{t}_{e\rightarrow v}),
$
independent of $\{c^t_{e\rightarrow v}\}$. The final procedure, in the case of quadratic initial messages, yields Algorithm \ref{alg:Min-Sum voltages, quadratic messages}.

\begin{algorithm}
	\DontPrintSemicolon
	\KwIn{Initial messages $\{W^{0}_{e\rightarrow v}\}, \{w^{0}_{e\rightarrow v}\}$, $e\in E,\!v\in\partial e$;}
	\For{$s \in \{1,\ldots,t\}$}{
		Compute, for each $e\in E$, $v\in\partial e$, with $w=\partial e \setminus v$
		$
		W^{s}_{e\rightarrow v}
		=
		W_{vw}
		\frac{
		\sum_{f\in\partial w \setminus e} 
		W^{s-1}_{f\rightarrow w}}
		{W_{vw} +
		\sum_{f\in\partial w \setminus e} 
		W^{s-1}_{f\rightarrow w}},
		\qquad
		w^{s}_{e\rightarrow v}
		=
		-W_{vw}
		\frac{b_w-\sum_{f\in\partial w \setminus e} 
		w^{s-1}_{f\rightarrow w}}{W_{vw} +
		\sum_{f\in\partial w \setminus e} 
		W^{s-1}_{f\rightarrow w}};
		$
		\vskip-0.13cm
		}
	\KwOut{
	$
	\hat \nu_v^t = \frac{b_v - \sum_{e\in\partial v} 
		w^{t}_{e\rightarrow v}}
		{\sum_{e\in\partial v} 
		W^{t}_{e\rightarrow v}}
	$, for each $v\in V$.}
\caption{Min-Sum, voltage problem, quadratic messages}
\label{alg:Min-Sum voltages, quadratic messages}
\end{algorithm}

Algorithm \ref{alg:Min-Sum voltages, quadratic messages}, with the choice of initial conditions $W^0_{e\rightarrow v} = W_{vw}$ and $w^0_{e\rightarrow v} = 0$, for each $e=\{v,w\}\in E, v\in\partial e$, will be the focus of our investigation for the voltage problem. The convergence analysis will rely on the general formulation of Algorithm \ref{alg:Min-Sum voltages}.

\begin{remark}[Gaussian Belief Propagation]\label{rem:gaussianBP}
When applied to quadratic objective functions, the Min-Sum algorithm (e.g., Algorithm \ref{alg:Min-Sum voltages, quadratic messages}) corresponds to the computation of the mean performed by the Belief Propagation algorithm applied to Gaussian graphical models \citep{WF01}. 
This connection is due to the fact that in the Gaussian case the mode of a distribution coincides with its mean. Min-Sum can be seen as computing the maximum a posteriori (MAP) estimate of the graphical model (i.e., the mode), which corresponds to the mean computed by Belief Propagation. Typically, the convergence and correctness of the mean of Gaussian Belief Propagation is investigated within the framework of walk-summability \citep{MJW06}, or, equivalently, scaled diagonal dominance \citep{MVR10}. As the Laplacian $L$ is not positive definite, the existing framework can not be applied to investigate the convergence behavior of Algorithm \ref{alg:Min-Sum voltages, quadratic messages}. In this paper, we overcome this limitation by using a different approach based on the global analysis of the computation tree that supports the operations of the algorithm with time.
\end{remark}

\subsection{Min-Sum for the Flow Problem}\label{ref:Min-Sum flows}

To apply the Min-Sum algorithm to solve the flow problem \eqref{primal}, we first rewrite this problem in the unconstrained form \eqref{cost function Min-Sum}. To this end, let us identify $\V \leftrightarrow \vec E$ and $\F \leftrightarrow V$. In this context, the notation $\partial v$ indicates the set of edges in $\vec E$ that either enter or leave node $v\in V$; analogously, the notation $\partial e$ indicates the pair of vertices that are connected by edge $e$. For each $e\in\vec E$, define the function $g_e(x_e) = \frac{1}{2}R_{ee} x_e^2$, and for each
$v\in V$ define the function $g_{v}(x_{\partial v}) := 0$ if $\sum_{e\in \vec E} A_{ve} x_e = b_v$, and $g_{v}(x_{\partial v}) := \infty$ otherwise.
Clearly, problem \eqref{primal} is equivalent to the problem of minimizing the function
$
	\sum_{e\in \vec E} g_{e}(x_e)
	+
	\sum_{v\in V} g_{v}(x_{\partial v}),
$
which is of the form \eqref{cost function Min-Sum}. As in the factor graph interpretation each variable node (i.e., edge in $\vec E$) is adjacent to exactly two factor nodes (i.e., vertices in $V$), it is convenient to rewrite the Min-Sum algorithm in terms of updates of messages from variables to factors only. At each time step $t\ge 1$, to each edge $e\in\vec E$ such that $\partial e = \{v,w\}$ (that is, either $e=(v,w)$ or $e=(w,v)$) are associated two messages: $\mu^t_{e\rightarrow v}$ and $\mu^t_{e\rightarrow w}$.
It can be verified that in the present setting Algorithm \ref{alg:Min-Sum} becomes Algorithm \ref{alg:Min-Sum for min-cost network flow}.\footnote{For convenience of notation, we present Algorithm \ref{alg:Min-Sum for min-cost network flow} with initial messages indexed by time $0$. Messages are indexed by a time parameter that is shifted by one unit compared to Algorithm \ref{alg:Min-Sum}.}

\begin{algorithm}
	\DontPrintSemicolon
	\KwIn{Initial messages $\{\mu^{0}_{e\rightarrow v}\}$, for $e\in \vec E, v\in\partial e$;}
	\For{$s \in \{1,\ldots,t\}$}{
		Compute, for each $e\in \vec E$, $v\in\partial e$, with $w=\partial e \setminus v$
		$
		\mu^{s}_{e\rightarrow v} =
		g_e + 
		\min_{x\in\R^{\vec E}}
		\bigg\{
		g_w(x_{\partial w\setminus e}\,\cdot) + \sum_{f\in\partial w \setminus e} \mu^{s-1}_{f\rightarrow w}(x_{f})
		\bigg\};
		$
		\vskip-0.13cm
		}
	Compute $\mu^{t}_{e} = \mu^{t}_{e\rightarrow v} + \mu^{t}_{e\rightarrow w} - g_e$, for $e=(v,w)\in\vec E$;\\
	\KwOut{$\hat x_e^t := \arg\min_{z\in\R} \mu^{t}_{e}(z)$, for each $e\in\vec E$.}
	\caption{Min-Sum, flow problem}
	\label{alg:Min-Sum for min-cost network flow}
\end{algorithm}

Note that if $\partial w \setminus e = \varnothing$, then the message $\mu^{s}_{e\rightarrow v}$ in Algorithm \ref{alg:Min-Sum for min-cost network flow} reads $\mu^{s}_{e\rightarrow v}(z)=\frac{1}{2} R_{ee}b_w^2$ if $z = A_{we}b_w$, and $\mu^{s}_{e\rightarrow v}(z)=\infty$ otherwise,
for any $s\ge 1$, irrespective of the choice of the initial messages. This is a consequence of the fact that if node $w\in V$ is a \emph{leaf} of $G$ (i.e., $w$ has degree $1$), then by Kirchhoff's law the value $x_e$ of the flow on the unique edge $e$ that is associated to $w$ is $x_e=A_{we} b_w$. Without loss of generality, in what follows we assume that the graph $G$ has no leaves. Otherwise, by removing one leaf at a time as described above, we can reduce the original problem to a new problem supported on a graph with no leaves (if $G$ is a tree we can exactly compute all the components of the solution $x^\star$ in this way).

Algorithm \ref{alg:Min-Sum for min-cost network flow} applies for any choice of the initial messages $\{\mu^{0}_{e\rightarrow v}\}$, for $e\in \vec E, v\in\partial e$. If we choose initial messages that are quadratic functions, then also the following messages computed by the algorithm remain quadratic. The next Lemma presents the quadratic updates, without stating formulas for the constant terms as these terms do not influence the output of the Min-Sum algorithm, as we will show next.

\begin{proposition}\label{prop:Min-Sum updates}
For a given $s\ge 1$, for each $e\in\vec E$, $v\in \partial e$ let
$
	\mu^{s-1}_{e\rightarrow v}(z) = \frac{1}{2} 
	R^{s-1}_{e\rightarrow v}z^2 + r^{s-1}_{e\rightarrow v}z
	+ c^{s-1}_{e\rightarrow v}.
$
Then, if the graph $G$ has no leaves, for each $e\in \vec E$, $v\in\partial e$, we have
$
	\mu^{s}_{e\rightarrow v}(z) =
	\frac{1}{2} 
	R^{s}_{e\rightarrow v}z^2 + r^{s}_{e\rightarrow v}z
	+ c^{s}_{e\rightarrow v},
$
where $\{c^{s}_{e\rightarrow v}\}$ is a set of constants, and (here $w=\partial e \setminus v$)
$$
		R^{s}_{e\rightarrow v}
		= R_{ee} + 
		\frac{1}
		{\sum_{f\in\partial w\setminus e} 1/R^{s-1}_{f\rightarrow w}},
		\qquad
		r^{s}_{e\rightarrow v}
		=
		- A_{we} 
		\frac{
		\sum_{f\in\partial w\setminus e} 
		A_{wf} r^{s-1}_{f\rightarrow w}/R^{s-1}_{f\rightarrow w}
		+ b_w
		}
		{\sum_{f\in\partial w\setminus e} 1/R^{s-1}_{f\rightarrow w}}.
$$
\end{proposition}

\begin{proof}
Fix $s\ge 1$, $e\in\vec E$, $v\in\partial e$, and let $w\in\partial e \setminus v$. As we assume that the graph $G$ has no leaves, then $\partial w \setminus e \neq \varnothing$ and the Min-Sum update rule reads
$
	\mu^{s}_{e\rightarrow v}(z)
	= g_e(z) + 
	\min_{x\in\R^{\partial w\setminus e} : \tilde A x=b_w-A_{we}z}
	\sum_{f\in\partial w \setminus e} \mu^{s-1}_{f\rightarrow w}(x_{f}),
$
where $\tilde A := A_{w,\partial w\setminus e}$ is the submatrix of $A$ given by the $w$-th row and the columns indexed by the elements in $\partial w\setminus e$.
The minimization in the definition of $\mu^{s}_{e\rightarrow v}$ can be written as
\begin{align}
\begin{aligned}
	\text{minimize }\quad   & \frac{1}{2} x^T \tilde R x + \tilde r^T x + 1^T\tilde c\\
	\text{subject to }\quad & \tilde A x = b_w - A_{we}z
\end{aligned}
\label{intermediate opt}
\end{align}
over $x\in \R^{\partial w\setminus e}$, with
	$\tilde R \in \R^{(\partial w\setminus e)\times (\partial w\setminus e)}$ diagonal with $\tilde R_{ff} := R^{s-1}_{f\rightarrow w}$,
	$\tilde r \in \R^{\partial w\setminus e}$ with $\tilde r_f := r^{s-1}_{f\rightarrow w}$, and
	$\tilde c \in \R^{\partial w\setminus e}$ with $\tilde c_f := c^{s-1}_{f\rightarrow w}$.
By Proposition \ref{prop:quadratic} in Appendix \ref{app:technical} the solution of \eqref{intermediate opt} reads
$
	x^\star = h - A_{we}z y,
$
with $h:= b_w y - (I-y \tilde A) \tilde R^{-1} \tilde r$, $y := \tilde R^{-1} \tilde A^T \tilde L^{-1}$, and $\tilde L:= \tilde A \tilde R^{-1} \tilde A^T\in\R$. Substituting $x^\star$ in \eqref{intermediate opt} we get
$
	\mu^{s}_{e\rightarrow v}(z) = \frac{1}{2} 
	R^{s}_{e\rightarrow v}z^2 + r^{s}_{e\rightarrow v}z
	+ c^{s}_{e\rightarrow v},
$
with
$
	R^{s}_{e\rightarrow v} := R_{ee} + y^T\tilde R y
$,
and
$
	r^{s}_{e\rightarrow v} := -A_{we} (y^T \tilde R h +\tilde r^T y).
$
The proof follows as
$
	\tilde L
	= \sum_{f\in \partial w\setminus e} 1/R^{s-1}_{f\rightarrow w},
$
$
	y^T\tilde R y = \tilde L^{-1},
$
$
	y^T \tilde R h
	= \tilde L^{-1} b_w,
$
$
	\tilde r^T y
	= \tilde L^{-1} 
	\sum_{f\in \partial w\setminus e} 
	A_{wf} r^{s-1}_{f\rightarrow w}/R^{s-1}_{f\rightarrow w}.
$
\end{proof}

From Proposition \ref{prop:Min-Sum updates} it follows that if we initialize Algorithm \ref{alg:Min-Sum for min-cost network flow} with quadratic messages
$
	\mu^{0}_{e\rightarrow v}(z) = \frac{1}{2} 
	R^{0}_{e\rightarrow v}z^2 + r^{0}_{e\rightarrow v}z,
$
then, modulo the constant terms, we can track the evolution of the quadratic messages computed by the algorithm by evaluating recursively the set of parameters $\{R^s_{e\rightarrow v}\},\{r^s_{e\rightarrow v}\}$ as prescribed by Proposition \ref{prop:Min-Sum updates}. The belief function associated to edge $e$ at time $t$ then reads
$
	\mu^{t}_{e}(z) = \frac{1}{2} 
	(R^{t}_{e\rightarrow v} + R^{t}_{e\rightarrow w} - R_{ee})z^2
	+ (r^{t}_{e\rightarrow v} + r^{t}_{e\rightarrow w}) z + c^t_e,
$
for a certain constant $c^t_e$, so that
$
	\hat x_e^t = \arg\min_{z\in\R} \mu^{t}_{e}(z)
	= - (r^{t}_{e\rightarrow v} + r^{t}_{e\rightarrow w}) /
	(R^{t}_{e\rightarrow v} + R^{t}_{e\rightarrow w} - R_{ee}),
$
independent of $c^t_e$. The final procedure so obtained, in the case of quadratic initial messages, for a graph $G$ with no leaves, is given in Algorithm \ref{alg:Min-Sum, quadratic messages, no leaves}.

\begin{algorithm}[h!]
	\DontPrintSemicolon
	\KwIn{Initial messages $\{R^{0}_{e\rightarrow v}\}, \{r^{0}_{e\rightarrow v}\}$, $e\in \vec E,\!v\in\partial e$;}
	\For{$s \in \{1,\ldots,t\}$}{
		Compute, for each $e\in \vec E$, $v\in\partial e$, with $w=\partial e \setminus v$
		$
		R^{s}_{e\rightarrow v}
		= R_{ee} + 
		\frac{1}
		{\sum_{f\in\partial w\setminus e} 1/R^{s-1}_{f\rightarrow w}},
		\qquad
		r^{s}_{e\rightarrow v}
		=
		- A_{we} 
		\frac{
		\sum_{f\in\partial w\setminus e} 
		A_{wf} r^{s-1}_{f\rightarrow w}/R^{s-1}_{f\rightarrow w}
		+ b_w
		}
		{\sum_{f\in\partial w\setminus e} 1/R^{s-1}_{f\rightarrow w}};
		$
		\vskip-0.13cm
		}
	\KwOut{$\hat x_e^t
	= - \frac{r^{t}_{e\rightarrow v} + r^{t}_{e\rightarrow w}}
	{R^{t}_{e\rightarrow v} + R^{t}_{e\rightarrow w} - R_{ee}}$, for $e=(v,w)\in\vec E$.}
\caption{Min-Sum, flow problem, quadratic messages, no leaves}
\label{alg:Min-Sum, quadratic messages, no leaves}
\end{algorithm}

Algorithm \ref{alg:Min-Sum, quadratic messages, no leaves}, with the choice of initial conditions $R^0_{e\rightarrow v} = R_{ee}$, $r^0_{e\rightarrow v} = 0$, for each $e\in\vec E, v\in\partial e$, will be the focus of our investigation for the flow problem. The convergence analysis will rely on the general formulation of Algorithm \ref{alg:Min-Sum for min-cost network flow}.

\section{Results for $d$-Regular Graphs}\label{sec:results}

A $d$-regular graph is a graph where each node has $d$ neighbors. In this section, we first characterize the error committed by the Min-Sum algorithm to solve the voltage and flow problems on connected $d$-regular graphs in terms of the solution $\nu^\star=L^+ b$ of the voltage problem. Then, we develop convergence results with respect to various norms. 

The proofs of the main results here presented---Lemma \ref{lem:ring} and Lemma \ref{lem:regulargraph}---are given in Section \ref{sec:proofs}, where we develop the general machinery to analyze the behavior of the Min-Sum scheme to solve constrained optimization problems, and we present the characterization of the error committed by the voltage and flow algorithms on general weighted graphs in terms of hitting times of random walks on the computation tree (for the sake of simplicity, we postpone the results on general graphs to Section \ref{sec:proofs}, as these results require the notion of computation tree). The proofs of the technical results used here are given in Appendix \ref{app:technical}.

Henceforth, let $\hat\nu^t$ be the outcome of Algorithm \ref{alg:Min-Sum voltages, quadratic messages} at time $t$ with the choice of initial conditions $W^0_{e\rightarrow v} = W_{vw}$ and $w^0_{e\rightarrow v} = 0$, for each $e=\{v,w\}\in E, v\in\partial e$. Analogously, let $\hat x^t$ be the outcome of Algorithm \ref{alg:Min-Sum, quadratic messages, no leaves} at time $t$ with the choice of initial conditions $R^0_{e\rightarrow v} = R_{ee}$, $r^0_{e\rightarrow v} = 0$, for each $e\in\vec E, v\in\partial e$.
Throughout, given a weighted directed graph $\vec G=(V,\vec E, W)$ with no bi-directed edges (i.e., either $(v,w)\in \vec E$ or $(w,v)\in \vec E$), let $G=(V,E,W)$ be the weighted undirected graph naturally associated to $\vec G$, i.e., edges in $E$ are obtained by removing the orientation of the edges in $\vec E$. Let $n:=|V|$ and $m:=|E|$ be the number of nodes and edges, respectively.

\subsection{Error Characterization}
We characterize the error committed by the Min-Sum algorithm as a function of time and as a function of the voltage solution $\nu^\star=L^+ b$.

The case $d=2$, i.e., a cycle graph, allows a characterization for general weighted graphs.

\begin{lemma}[Regular graphs, $d=2$]\label{lem:ring}
Let $\vec G=(V=\{0,\ldots,n-1\},\vec E,W)$ be a weighted directed cycle, and let $G=(V,E,W)$ be the corresponding undirected graph.
Define the function $i\in \mathbb{N} \rightarrow \rho(i) := i \mod n$. For $v\in V$, $e=\{\rho(v),\rho(v+1)\}\in E$, $t\ge 2$, we have
\begin{align*}
	\nu^\star_{v} - \hat \nu^t_{v}
	&=
	\alpha_{v,t} \, \nu^\star_{\rho( v-t-1)} 
	+ (1-\alpha_{v,t}) \, \nu^\star_{\rho( v+t+1)},
	&\alpha_{v,t}:=\frac{\sum_{k= v}^{ v + t} 1/W_{\rho(k)\rho(k+1)}}{\sum_{k= v-t-1}^{ v+t} 1/W_{\rho(k)\rho(k+1)}},\\
	x^\star_{e} - \hat x^t_{e}
	&= A_{ve} \, \beta_{e,t} \, (\nu^\star_{\rho(v-t)} - \nu^\star_{\rho(v+t+1)}),
	&\beta_{e,t} := \frac{1}{\sum_{k=v-t}^{v+t} 1/W_{\rho(k)\rho(k+1)}}.
\end{align*}
If each edge has the same weight $\omega$, then $\alpha_{v,t}= 1/2$ and $\beta_{e,t}=\omega/(2t+1)$.
\end{lemma}

In the case of constant weights, this result shows that Min-Sum displays a non-convergent (oscillatory) behavior in the voltage setting (as $\alpha_{v,t}$ does not decay with time), while Min-Sum displays a convergent behavior in the flow setting (as $\beta_{v,t}$ decays with time). We will discuss this behavior more extensively in the following sections.

\begin{remark}
The fact that the Min-Sum algorithm does not converge is not surprising. Gaussian BP, see Remark \ref{rem:gaussianBP}, is known to suffer from divergent behaviors even in the absence of degeneracy, i.e., even when the quadratic form of interest is defined by a positive definite matrix, as observed numerically in \cite{RvR01}.
\end{remark}

For $d\ge 3$ we can provide a characterization of the error in the case of graphs with equal weights. We begin by defining and bounding some quantities of interest that we will need to state our results. The proof of the following proposition is given in Appendix \ref{app:technical}.

\begin{proposition}\label{prop:constant d t}
Fix $d\ge 3$. For $s\ge 1$ let $h_s:=\frac{1}{(d-1)^{s}-1}$, and let $\delta_0 := \frac{1}{d}$, $\delta_1 := \frac{1}{d-1}+ d-1$, and
$
	\delta_s := 
	\frac{1}{d-1}
	(2+\frac{(d-2)^2}{d-1}(1+h_{s+2})) \delta_{s-1} -\frac{1}{(d-1)^2}\delta_{s-2}
$
for $s\ge 2$. For any $t\ge 3$ define
\begin{align*}
	b_{d,t} &:= 
	\frac{1}{(d-1)^2}\bigg(1+(d-2)(1+h_{t+1})\bigg) \delta_{t-2}
	- \frac{(d-2)(1+h_{t+1})}{(d-1)^3} \delta_{t-3},\\
	c_{d,t} &:= 
	\frac{1}{d-1}\bigg(1+\frac{1}{(d-2)(1+h_{t})}\bigg) \delta_{t-2}
	- \frac{1}{(d-1)^2} \delta_{t-3}.
\end{align*}
Then we have $\frac{1}{2} \le \frac{d-2}{d-1} \le b_{d,t} \le c_{d,t} \le 1+\varepsilon_d < 4$ and $c_{d,t}\ge 1$, where $\varepsilon_d$ is a positive decreasing function of $d\ge 3$ so that $\varepsilon_d\rightarrow 0$ as $d\rightarrow \infty$ and $\varepsilon_3<3$.
\end{proposition}

For each $v\in V$, let $\vec{\mathbf{P}}_{v}$ be the law of a time-homogeneous \emph{non-backtracking} random walk $Y_{0},Y_{1},Y_2,\ldots$ on the nodes of $G$ with initial state $Y_0=v$, which is defined by $\vec{\mathbf{P}}_{v}(Y_0=w) := 1$ if $w=v$ and equal to $0$ otherwise, by $\vec{\mathbf{P}}_{v}(Y_1=w) := 1/ d$ if $\{w,v\}\in E$ and equal to $0$, and, for any $\{z,z'\}\in E$, by $\vec{\mathbf{P}}_{v}(Y_{t+1} = w | Y_{t}=z,Y_{t-1}=z') := \frac{1}{d-1}$ if $\{w,z\}\in E, w\neq z'$ and equal to $0$ otherwise.
For ease of notation, for $t\ge 0$, $w\in V$, define the row-stochastic matrices $P^{(t)}$ and $P^{(t,w)}$ as
\begin{align}
	P^{(t)}_{vz} := \vec{\mathbf{P}}_{v}(Y_t=z),
	\qquad
	P^{(t,w)}_{vz} := \vec{\mathbf{P}}_{v}(Y_t=z|Y_1\neq w),
	\qquad v,z\in V.
	\label{prob distrib voltage flow}
\end{align}

We can now state the main result for $d$-regular graphs with $d\ge 3$ and equal weights, namely, $W_{vw}=\omega$ for any $\{v,w\}\in E$, for some $\omega > 0$.

\begin{lemma}[Regular graphs, $d\ge 3$, equal weights]\label{lem:regulargraph}
Let $\vec G=(V,\vec E,W)$ be a connected directed $d$-regular graph, with $d\ge 3$, where each edge has the same weight $\omega>0$, and let $G=(V,E,W)$ be the corresponding undirected graph. For $v\in V$, $e=(v, w)\in\vec E$, $t\ge 3$,
\begin{align*}
	\nu^\star_{v} - \hat \nu^t_{v} =
	\frac{1}{b_{d,t}}\,
	(P^{(t+1)} \nu^\star)_v,
	\qquad
	x^\star_{e} - \hat x^t_{e} =
	\frac{\omega}{c_{d,t}} \frac{d-1}{d}
	(
	(P^{(t,w)} \nu^\star)_v
	-
	(P^{(t,v)} \nu^\star)_w
	),
\end{align*}
where $b_{d,t}$ and $c_{d,t}$ are defined and bounded as in Proposition \ref{prop:constant d t}.
\end{lemma}

This result characterizes the error committed by the Min-Sum algorithm in terms of the distribution of properly-defined non-backtracking random walks. Several papers have previously related the convergence properties of message-passing algorithms to non-backtracking random walks. For instance, we refer to \cite{ConsensusProp} for an application to consensus optimization (this paper also analyzes $d$-regular graphs), and to \cite{CIT-067} for a more recent application to the Stochastic Block Model.

Given the error characterizations provided by Lemma \ref{lem:ring} and Lemma \ref{lem:regulargraph}, we now investigate and discuss the convergence properties of the Min-Sum algorithm.

\subsection{Convergence Results for the Voltage Problem}\label{sec:Convergence results for the voltage problem}
We first analyze the convergence behavior of the Min-Sum algorithm when applied to the voltage problem: Algorithm \ref{alg:Min-Sum voltages, quadratic messages}. Following the literature on Laplacian solvers, as already noticed in Section \ref{sec:Voltage and flow problem}, we perform the analysis in the $L$-norm, which is defined as $\| \nu \|_L := \sqrt{\nu^TL\nu}$ for any $\nu\in\R^V$, where $\nu^TL\nu = \sum_{\{v,w\}\in E} W_{vw} (\nu_v-\nu_w)^2$.

In the case $d=2$, as previously discussed, the results of Lemma \ref{lem:ring} attest that the output $\hat \nu^t$ of Algorithm \ref{alg:Min-Sum voltages, quadratic messages} does \emph{not} converge, as it keeps oscillating as a function of $t$. This is particularly evident in the case of equal weights, as the following corollary attests.

\begin{corollary}[Regular graphs, $d= 2$]\label{cor:not convergence}
Consider the setting of Lemma \ref{lem:ring} with equal weights $W_{vw}=\omega$. For $t\ge 2$,
$
	\| \nu^\star - \hat \nu^t \|_L^2
	= \frac{1}{2} \| \nu^\star \|_L^2
	+ \frac{\omega}{2} \sum_{v=0}^{n-1}
	\nu^\star_{v}
	(
	2\nu^\star_{\rho(v+2t+2)} 
	- \nu^\star_{\rho(v+2t+3)}
	- \nu^\star_{\rho(v+2t+1)}
	).
$
\end{corollary}

\begin{proof}
The proof follows immediately from Lemma \ref{lem:ring} noticing that $\| \nu^\star - \hat \nu^t \|_L^2$ equals
$
	\omega \sum_{v=0}^{n-1} (\nu^\star_{v+1} - \hat \nu^t_{v+1}
	- (\nu^\star_{v} - \hat \nu^t_{v}))^2
	= \frac{\omega}{4} \sum_{v=0}^{n-1}
	(
	\nu^\star_{\rho(v-t)} - \nu^\star_{\rho(v-t-1)}
	+ \nu^\star_{\rho( v+t+2)}
	- \nu^\star_{\rho( v+t+1)}
	)^2,
$
which reads
$
	\frac{1}{2} \| \nu^\star \|_L^2
	+ \frac{\omega}{2} \sum_{v=0}^{n-1}
	(
	\nu^\star_{\rho(v-t)} - \nu^\star_{\rho(v-t-1)}
	)
	(
	\nu^\star_{\rho( v+t+2)} - \nu^\star_{\rho( v+t+1)}
	).
$
The result follows by grouping and relabeling the terms in the sum.
\end{proof}

The oscillatory behavior of the Min-Sum algorithm on cycle graphs in the voltage problem marks a clear difference with what happens in the flow problem where, as we will see in Corollary \ref{cor:cycle conv} below, the algorithm \emph{does} converge to the problem solution on cycle graphs.

Let us now consider the case $d\ge 3$. In this case, as Corollary \ref{cor:convergence 2} below attests, the convergence of the Min-Sum algorithm in the $L$-norm can be controlled by the total variation (TV) distance between the probability distributions of non-backtracking random walks that originate from neighbor nodes. In this framework, it is natural to also consider a modification of the $L$-norm normalized by the probability mass that is left out by the overlap of the above-mentioned random walks, which coincides with the TV distance between the walks. We give a general definition of this norm, and then present some of its properties.

\begin{definition}[TV-normalized $L$-norm]
Given a row-stochastic matrix $M\in\R^{V\times V}$, define the pseudo-norm $\| \,\cdot\, \|_{L,M}$ as
$
	\| \nu \|_{L,M}^2 := \sum_{\{v,w\}\in E} W_{vw} 
	\big( \frac{(M\nu)_v - (M\nu)_w}{\| M_{v} - M_{w} \|_{TV}} \big)^2,
$
for any $\nu\in\R^V$, where $\| M_{v} - M_{w} \|_{TV} := \frac{1}{2}\sum_{z\in V}|M_{vz} - M_{wz}|$ is the TV distance between the probability mass functions defined by the $v$-th and $w$-th rows of the matrix $M$, respectively.
\end{definition}

The fact that $\| \,\cdot\, \|_{L,M}$ is a pseudo-norm follows from the fact that $\| \,\cdot\, \|_{L}$ itself is a pseudo-norm. Indeed, note that $\| \nu \|_{L,M} = \| M\nu \|_{L'}$, where $L'$ is the Laplacian of the weighted graph $(V,E,W')$, with $W'_{vw} := W_{vw}/\| M_{v} - M_{w} \|_{TV}^2$. By a well-known result (see Proposition 4.3 in \cite{LPW09}, for instance), we have $\| M_{v} - M_{w} \|_{TV} = \sum_{z\in V: M_{vz} > M_{wz}} (M_{vz} - M_{wz}) = \sum_{z\in V: M_{vz} < M_{wz}} (M_{vz} - M_{wz})$, which shows how the TV-normalized $L$-norm takes into account the overlap of the distributions defined by the rows of the matrix $M$.

\begin{proposition}\label{property TV distance}
For any row-stochastic matrix $M\in\R^{V\times V}$ and $\nu\in\R^V$ we have:
\begin{enumerate}[(i)]
\item 
$\| \nu \|_{L,I} = \| \nu \|_{L}$, where $I$ is the identity matrix;
\item
$
	\| M\nu \|_L \le \max_{\{v,w\}\in E} \| M_{v} - M_{w} \|_{TV}
	\,\| \nu \|_{L,M};
$
\item $\| \nu \|_{L} \le (\sum_{\{v,w\}\in E} W_{vw})^{1/2} \max_{\{v,w\}\in E} |\nu_v - \nu_w| \le (\sum_{\{v,w\}\in E} W_{vw})^{1/2} \operatorname{osc}(\nu)$;
\item $\| \nu \|_{L,M} \le (\sum_{\{v,w\}\in E} W_{vw})^{1/2} \operatorname{osc}(\nu)$;
\end{enumerate}
where the oscillation of $\nu$ is defined as $\operatorname{osc}(\nu) := \max_{v,w\in V} |\nu_v - \nu_w|$.
\end{proposition}

\begin{proof}
Properties $(i),(ii)$ and $(iii)$ are straightforward to verify. Property $(iv)$ follows from $|(M\nu)_v - (M\nu)_w| \le \| M_{v} - M_{w} \|_{TV} \operatorname{osc}(\nu)$. See \cite{Geo11}, Chapter 8, for instance.
\end{proof}

We now state a convergence result for the Min-Sum algorithm $\hat\nu^t$ for the voltage problem, in the case of $d$-regular graphs, $d\ge 3$, with the same edge weights.
As the $L$-norm involves a sum over neighbour nodes, and as Lemma \ref{lem:regulargraph} shows that the error of Min-Sum is characterized by the distribution of a (non-overlapping) random walk, we immediately see that Min-Sum for the voltage problem does not converge on bipartite graphs as in this case random walks that start at two neighbor locations never overlap. For this reason, we also consider an averaged version of the Min-Sum algorithm where the output of two consecutive iterations is properly averaged: $
	\hat \nu_{\emph{ave}}^t := 
	\frac{b_{d,t-1}}{b_{d,t-1} + b_{d,t}} \hat \nu^{t-1} 
	+ 
	\frac{b_{d,t}}{b_{d,t-1} + b_{d,t}} \hat \nu^{t}.
$
Let $Q^{(t)}:=(P^{(t-1)} + P^{(t)})/2$ and define
\begin{align*}
	\gamma(t) := \max_{\{v,w\}\in E} \| P^{(t)}_{v} - P^{(t)}_{w} \|_{TV},
	\qquad\gamma_{\text{ave}}(t) := \max_{\{v,w\}\in E} \| Q^{(t)}_{v} - Q^{(t)}_{w} \|_{TV}.
\end{align*}

\begin{corollary}[Regular graphs, $d\ge 3$, equal weights]\label{cor:convergence 2}
Setting of Lemma \ref{lem:regulargraph}. For $t\ge 4$,
\begin{align*}
	\|\nu^\star - \hat \nu^t \|_L
	&\le 
	2\, \gamma(t+1) \,
	\| \nu^\star \|_{L, P^{(t+1)}}
	\le 
	2\, \gamma(t+1) \, \sqrt{m} \, 
	\operatorname{osc}(\nu^\star),\\
	\|\nu^\star - \hat \nu_{\emph{ave}}^t \|_L
	&\le 
	2\, \gamma_{\emph{ave}}(t+1) \,
	\| \nu^\star \|_{L, Q^{(t+1)}}
	\le 
	2\, \gamma_{\emph{ave}}(t+1) \, \sqrt{m} \,
	\operatorname{osc}(\nu^\star).
\end{align*}
\end{corollary}

\begin{proof}
From Lemma \ref{lem:regulargraph},
$
	\nu^\star - \hat \nu^t =
	P^{(t+1)} \nu^\star / b_{d,t}
$. Using property $(ii)$ in Proposition \ref{property TV distance} with $M=P^{(t+1)}$ we get $
	\| \nu^\star - \hat \nu^t \|_L^2
	=
	\| P^{(t+1)} \nu^\star \|_L^2 / b_{d,t}^2
	\le 
	4 \, \gamma(t+1)^2\,
	\| \nu^\star \|_{L,P^{(t+1)}}^2,
$
as $b_{d,t} \ge 1/2$ by Proposition \ref{prop:constant d t}.
As for $t\ge 4$ Lemma \ref{lem:regulargraph} yields
$
	(b_{d,t-1} + b_{d,t}) (\nu^\star - \hat \nu_{\text{ave}}^t)
	=
	b_{d,t-1}(\nu^\star - \hat \nu^{t-1})
	+ b_{d,t}(\nu^\star - \hat \nu^t)
	= 2 Q^{(t+1)} \nu^\star,
$
we get
$
	\| \nu^\star - \hat \nu_{\text{ave}}^t \|_L^2
	= 4
	\| Q^{(t+1)} \nu^\star \|_L^2 / (b_{d,t-1} + b_{d,t})^2
	\le 
	4\,\gamma_{\text{ave}}(t+1)^2\,
	\| \nu^\star \|_{L,Q^{(t+1)}}^2.
$
The proof is concluded using Property $(iv)$ in Proposition \ref{property TV distance}.
\end{proof}

Corollary \ref{cor:convergence 2} shows that the convergence in the $L$-norm of the Min-Sum estimate $\hat \nu^t$ and its averaged version $\hat \nu_{\text{ave}}^t$ is controlled by the behavior with time of the quantities $\gamma(t+1)$ and $\gamma_{\text{ave}}(t+1)$, respectively, which represent uniform bounds on the TV distance between probability distributions of non-backtracking random walks that make $t+1$ steps originating from neighbor nodes.
Proposition \ref{proof:Non-backtracking random walks} in Appendix \ref{app:technical} yields a recursive formula to compute the matrices $P^{(t)}$'s in terms of the adjacency matrix of the undirected unweighted graph $(V,E)$. This formula can then be used to compute the quantities $\gamma(t)$ and $\gamma_{\text{ave}}(t)$. Here we present numerical results for a few classes of $d$-regular graphs (for convenience of presentation, we include also the numerical results related to the quantities $\widetilde\gamma(t)$ and $\widetilde\gamma_{\text{ave}}(t)$ that control the convergence behavior of the Min-Sum algorithm in the flow problem and that will be defined later on in Section \ref{sec:flow result}).

\begin{example}[Complete graphs]
A $n$-complete graph is a graph with $n$ nodes where every pair of nodes is connected by an edge. A $n$-complete graph is $d$-regular with degree $d=n-1$. Figure \ref{fig:complete decay} (left).
The quantity $\gamma(t)$ is seen to decay (albeit not monotonically) like an exponential function $\alpha(d) \exp(-\beta(d) t)$, where $\alpha(d)>0$ is a decreasing function of $d$ and $\beta(d)>0$ is an increasing function of $d$. Figure \ref{fig:complete decay} (center).
\begin{figure}[h!]
   \centering      
   \hspace*{-.35in}
   \raisebox{-0.5\height}{\includegraphics[height=4.2cm]{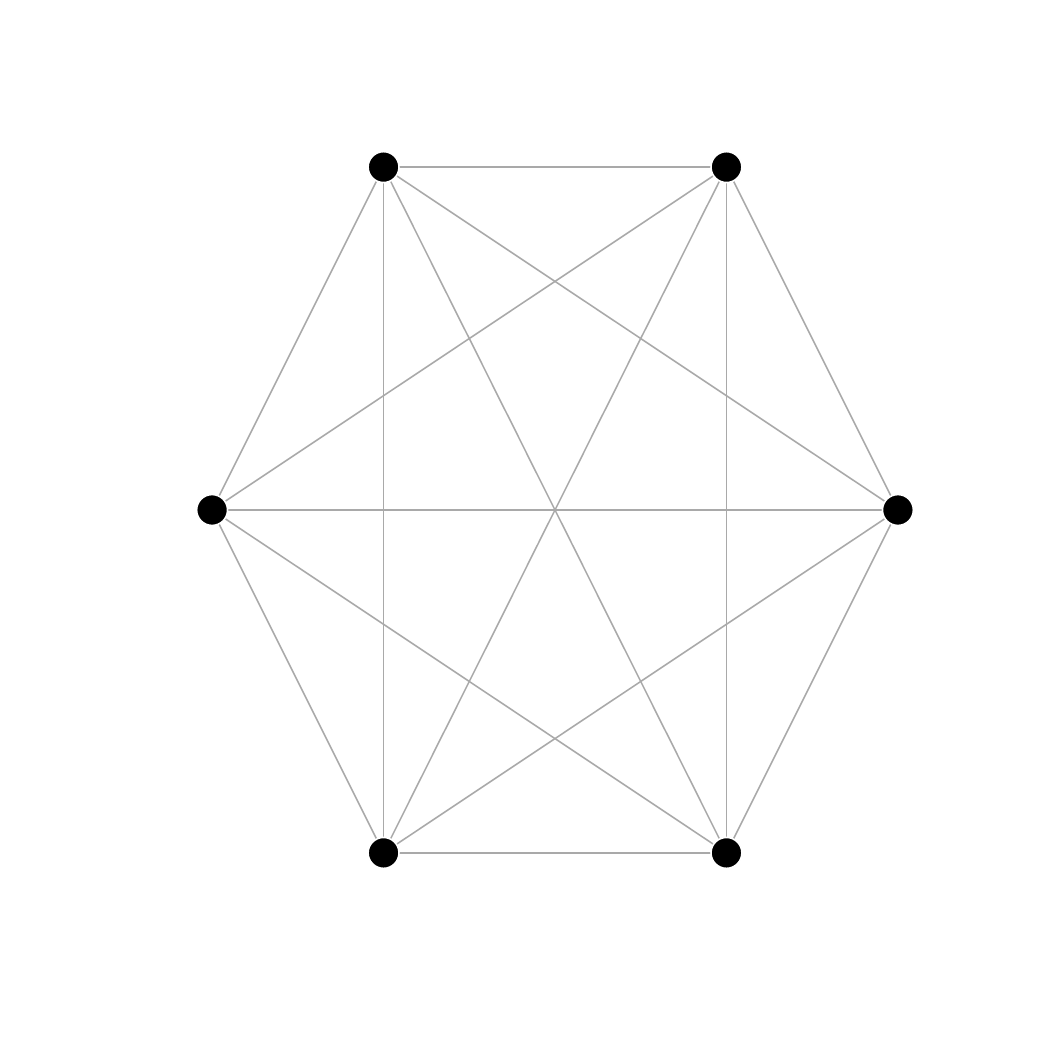}}
   \hspace*{-.05in}
   \raisebox{-0.5\height}{\includegraphics[height=4cm]{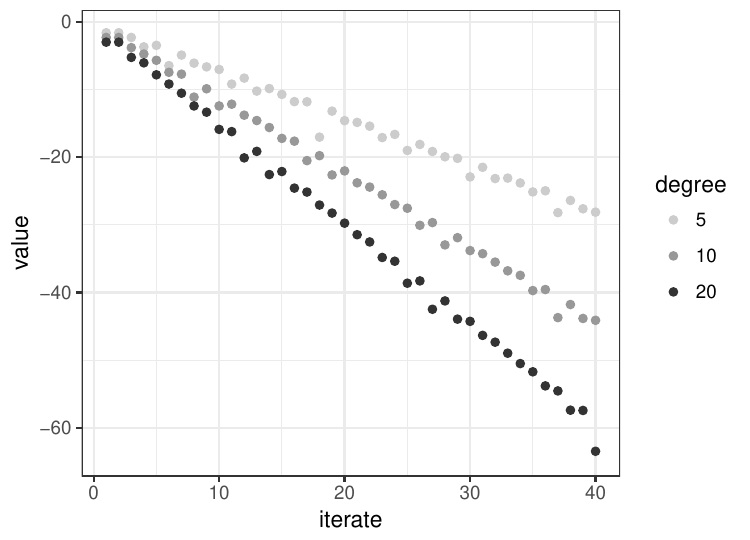}}
   \hspace*{0.1in}
   \raisebox{-0.5\height}{\includegraphics[height=4cm]{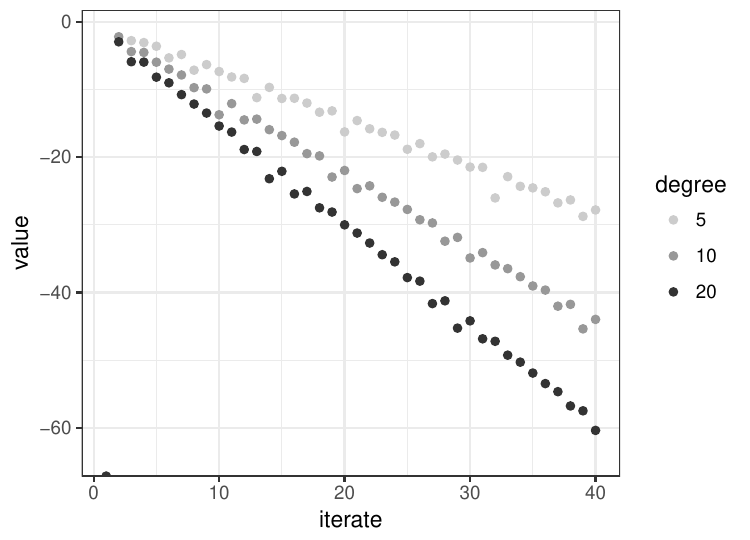}}
 \caption{A $6$-complete graph (left). Plot of $\log(\gamma(t))$ (center) and of $\log(\widetilde\gamma(t))$ (right) versus time $t$ for various $(d+1)$-complete graphs.}
 \label{fig:complete decay}
\end{figure}
\end{example}

\begin{example}[Connected cycles]\label{ex:connected cycle}
A $d/2$-connected cycle, for $d \ge 4$ even, is a cycle graph where each node is connected to its $d$ nearest nodes. Figure \ref{fig:ring decay} (left). The quantity $\gamma(t)$ is seen to decay like a polynomial function $\alpha(d)/\sqrt{t}$, where $\alpha(d)>0$ is an increasing function of $d$. The decay of $\gamma(t)$ is independent\footnote{It is implied that the decay with time is independent of $n$ and $m$ as long as time $t$ satisfies $t \lesssim D/d$, where $D$ is the diameter of the graph. Once the random walks reach the periodic boundary conditions the decay is faster (exponential instead of polynomial) due to the extra overlaps of the walks.} of $n$ and $m$. Figure \ref{fig:ring decay} (center).
\begin{figure}[h!]
   \centering      
   \hspace*{-.35in}
   \raisebox{-0.5\height}{\includegraphics[height=4.2cm]{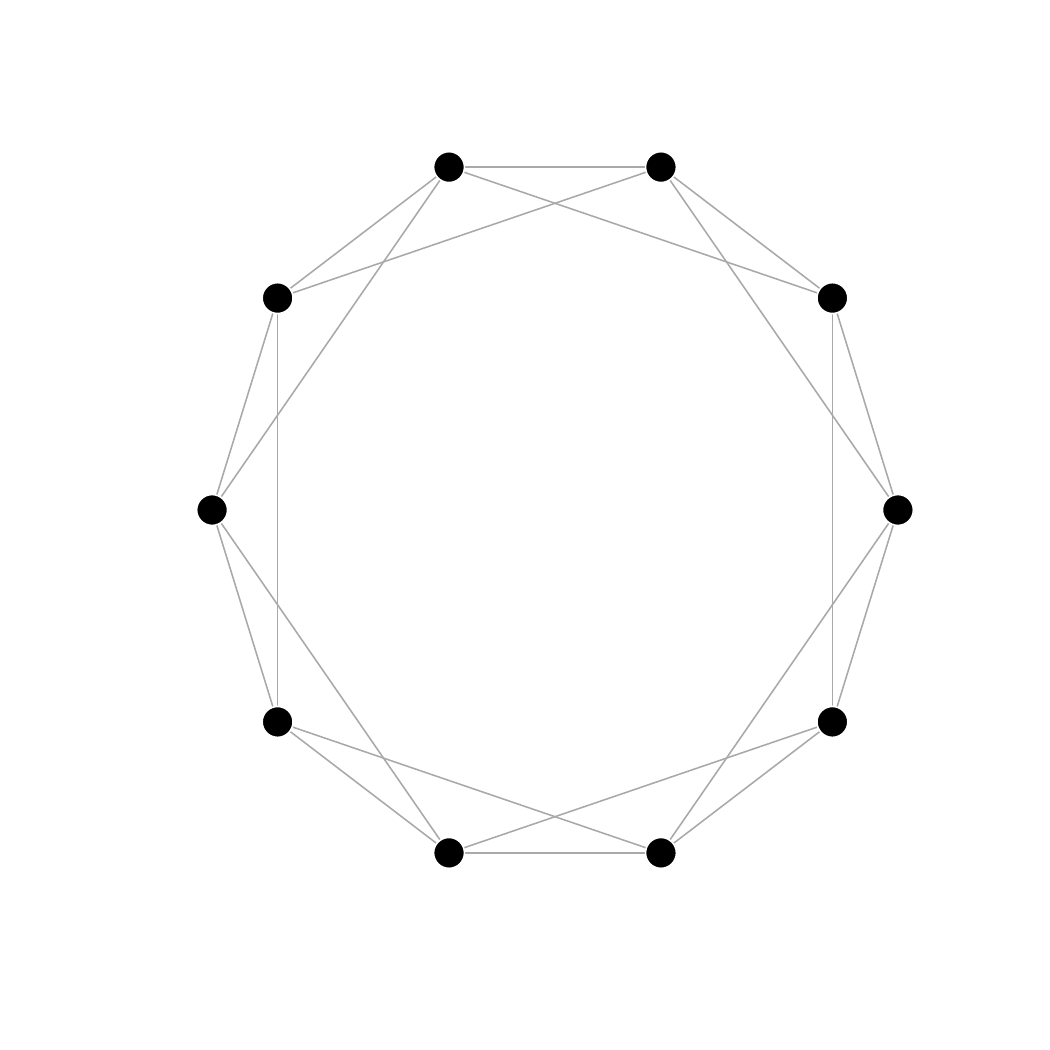}}
   \hspace*{-.05in}
   \raisebox{-0.5\height}{\includegraphics[height=4cm]{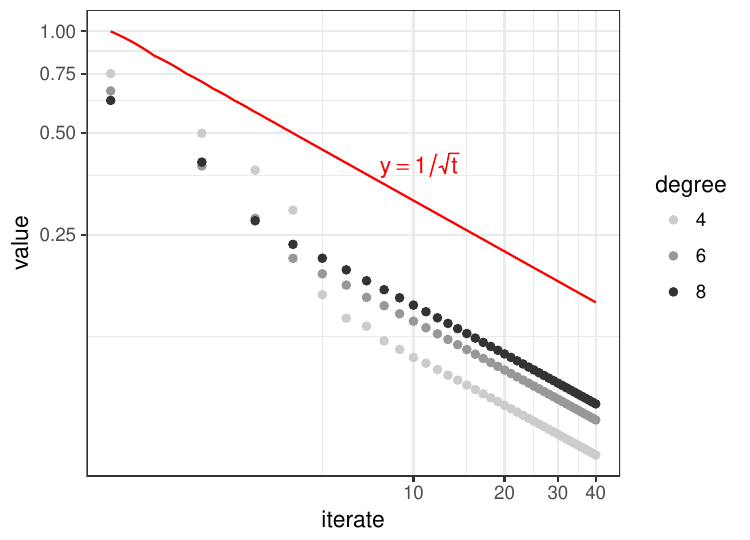}}
   \hspace*{0.1in}
   \raisebox{-0.5\height}{\includegraphics[height=4cm]{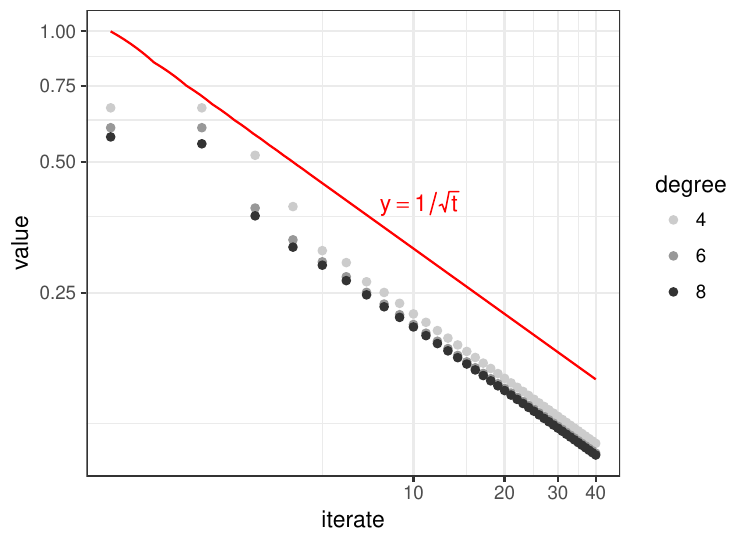}}
 \caption{A $2$-connected cycle (left). Plot in logarithmic scale of $\gamma(t)$ (center) and of $\widetilde\gamma(t)$ (right) versus time $t$ for various $d/2$-connected cycles.}
 \label{fig:ring decay}
\end{figure}
\end{example}

\begin{example}[Tori]
A $d/2$-dimensional torus, for $d \ge 4$ even, is a $d/2$-dimensional grid with periodic boundary conditions. Figure \ref{fig:torus decay} (left). The quantity $\gamma_{\emph{ave}}(t)$ is seen to decay like a polynomial function $\alpha(d)/\sqrt{t}$, where $\alpha(d)>0$ is an increasing function of $d$. The decay of $\gamma_{\emph{ave}}(t)$ is independent of $n$ and $m$ (the same considerations as in Example \ref{ex:connected cycle} hold). Figure \ref{fig:torus decay} (center).
\begin{figure}[h!]
   \centering
   \hspace*{0.005in}
   \raisebox{-0.5\height}{\includegraphics[height=3.3cm]{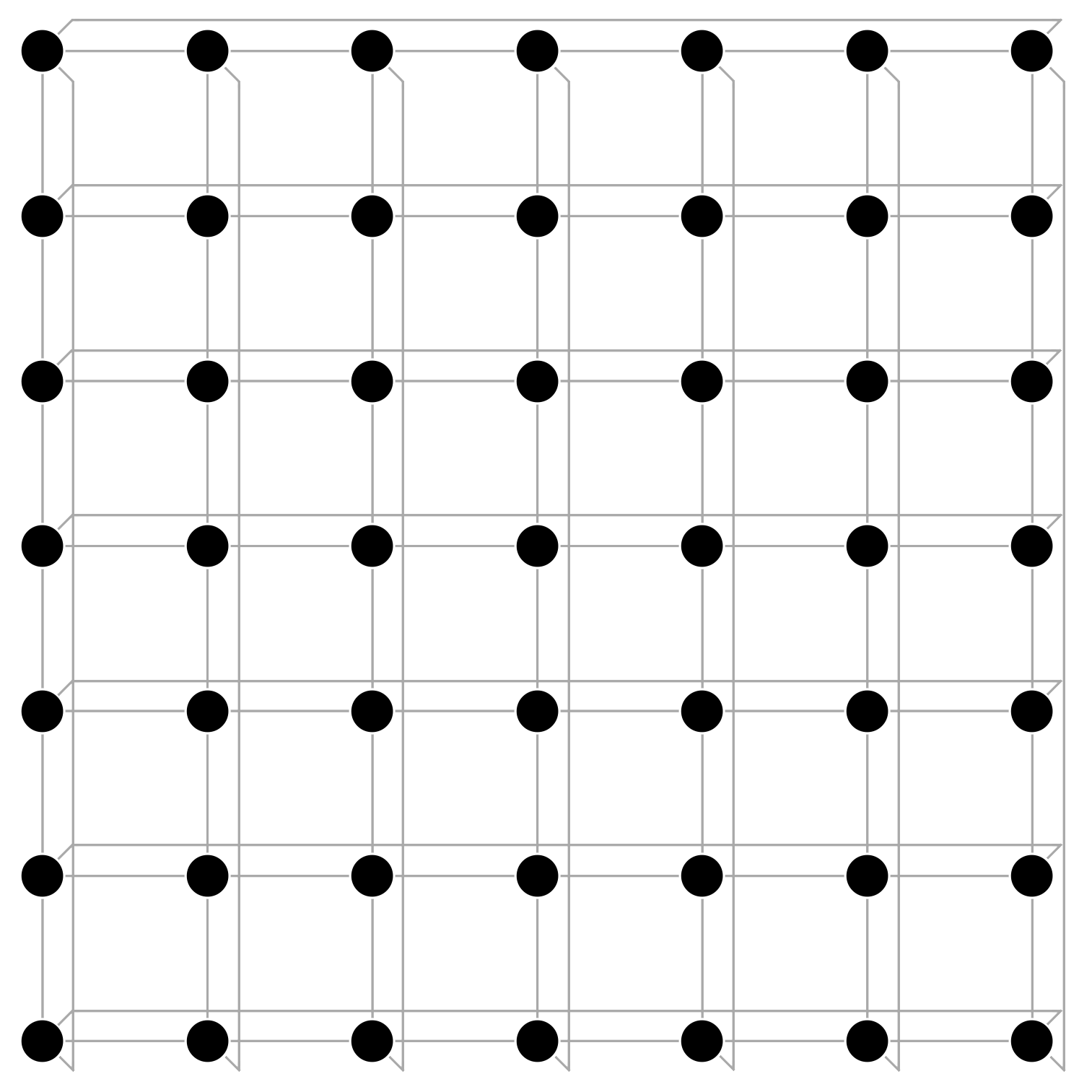}}
   \hspace*{-.05in}
   \raisebox{-0.5\height}{\includegraphics[height=4cm]{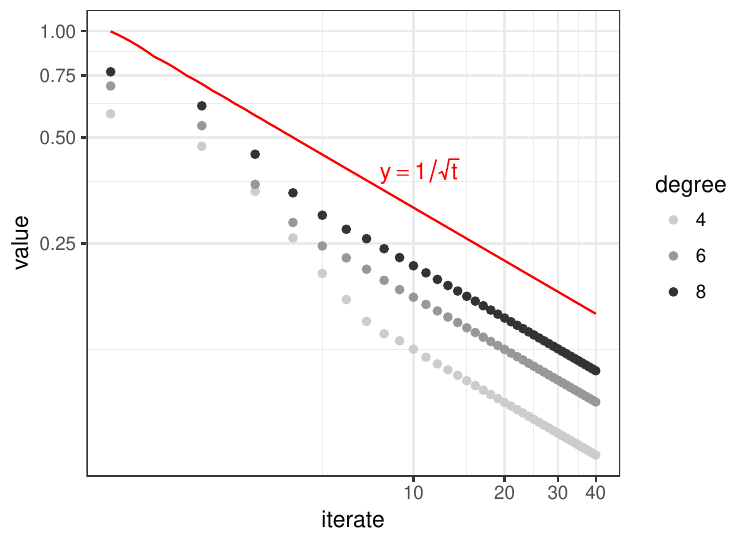}}
   \hspace*{0.1in}
   \raisebox{-0.5\height}{\includegraphics[height=4cm]{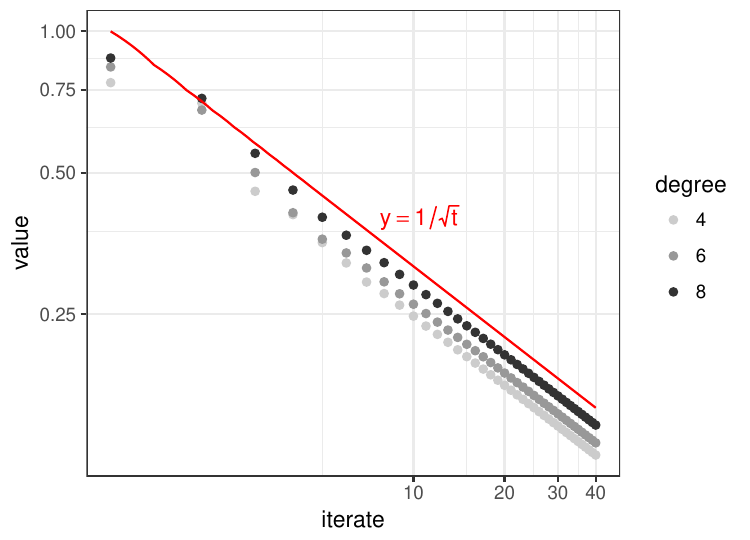}}
 \caption{A $2$-dimensional torus (left). Plot in logarithmic scale of $\gamma_{\text{ave}}(t)$ (center) and of $\widetilde\gamma_{\text{ave}}(t)$ (right) versus time $t$ for various $d/2$-dimensional tori.
}
\label{fig:torus decay}
\end{figure}
\end{example}

\begin{remark}[Dobrushin's ergodicity coefficient]
For Markov chains with transition kernel $P$, there is a large literature investigating the contraction coefficient $\max_{v,w\in V} \| P^{t}_{v} - P^{t}_{w} \|_{TV}$ (a.k.a. Dobrushin's ergodicity coefficient), where $t$ here represents the minimal time so that the coefficient is strictly less than $1$. The Dobrushin's coefficient is fundamentally different from $\gamma(t)$. The former involves a worst-case bound over all nodes in the graph, while the latter involves a worst-case bound only over nearest-neighbors nodes. As a result, the Dobrushin's ergodicity coefficient typically depends on $n$ and $m$, even when $\gamma(t)$ does not, such as in the examples above.
\end{remark}

The fast solvers analyzed in the literature (see Introduction, Section \ref{sec:introduction}, and Section \ref{sec:Voltage and flow problem}) achieve a quasi-linear running time $\tilde O(m\log^c n \log 1/\varepsilon)$ to return an estimate $\hat\nu$ that satisfies
$
	\| \nu^\star - \hat\nu \|_L \le \varepsilon \| \nu^\star \|_L.
$
As per the results in Corollary \ref{cor:convergence 2}, to compare the performance of the Min-Sum algorithm to these solvers we need to relate the $TV$-normalized $L$-norm to the ordinary $L$-norm. 
It is easy to construct examples where $\| \nu^\star \|_L < \| \nu^\star \|_{L,M}$ or where $\| \nu^\star \|_L > \| \nu^\star \|_{L,M}$, for a certain row-stochastic matrix $M$, while both norms satisfy $\| \nu^\star \|_L \le \sqrt{m} \operatorname{osc}(\nu)$ (property (iii) in Proposition \ref{property TV distance}) and $\| \nu^\star \|_{L,M} \le \sqrt{m} \operatorname{osc}(\nu)$ (property (iv) in Proposition \ref{property TV distance}). 
The following theorem summarizes the algorithmic complexity of the Min-Sum algorithm $\nu^t$ in the $L$-norm, under the assumption that the two norms of interest can be compared, and under the assumption that $\gamma(t)$ decays exponentially or polynomially, as attested numerically in the examples above (an analogous statement holds for the averaged Min-Sum algorithm $\nu_{\text{ave}}^t$, replacing $P^{(t)}$ with $Q^{(t)}$, and $\gamma(t)$ with $\gamma_{\text{ave}}(t)$).

\begin{theorem}[Min-Sum voltage, running time, $L$-norm]\label{thm:runningtime}
Setting of Lemma \ref{lem:regulargraph}.\\ Let $\alpha>0,\beta>0,\delta>1$ independent of $n$, $m$. Assume that $\| \nu^\star \|_{L,P^{(t)}} \le \delta \| \nu^\star \|_L$ for $t\ge 0$.
\begin{enumerate}
\item If $\gamma(t) \le \alpha \exp(-\beta t)$ for $t\ge 0$, then the Min-Sum algorithm yields an estimate $\hat\nu$ such that $\| \nu^\star - \hat\nu \|_L \le \varepsilon \| \nu^\star \|_L$ with running time $O(m \log 1/\varepsilon)$.
\item If $\gamma(t) \le \alpha / t^{\beta}$ for $t\ge 0$, then the Min-Sum algorithm yields an estimate $\hat\nu$ such that $\| \nu^\star - \hat\nu \|_L \le \varepsilon \| \nu^\star \|_L$ with running time $O(m /\varepsilon^{1/\beta})$.
\end{enumerate}
\end{theorem}
\begin{proof}
The proof follows immediately from Corollary \ref{cor:convergence 2}, recalling that each iteration of the Min-Sum algorithm requires a complexity that is linear in the number of edges, i.e., $O(m)$. If \emph{1.}\ holds, then $\| \nu^\star - \hat\nu^t \|_L \le \varepsilon \| \nu^\star \|_L$ is guaranteed by choosing $t$ such that $2 \alpha \delta \exp(-\beta (t+1)) \le \varepsilon$; we can choose the smallest value $t=\lceil \log(2\alpha\delta/\varepsilon)/\beta - 1 \rceil \sim O(\log 1/\varepsilon)$. If \emph{2.}\ holds, then $\| \nu^\star - \hat\nu^t \|_L \le \varepsilon \| \nu^\star \|_L$ is guaranteed by choosing $t$ such that $2 \alpha \delta / (t+1)^\beta \le \varepsilon$; we can choose the smallest value $t=\lceil (2\alpha\delta/\varepsilon)^{1/\beta} - 1 \rceil \sim O(1/\varepsilon^{1/\beta})$.
\end{proof}

Theorem \ref{thm:runningtime} shows that for certain classes of $d$-regular graphs, as characterized by the decaying properties of $\gamma(t)$, under the assumption that the TV-normalized $L$-norm of $\nu^\star$ is bounded by the $L$-norm of $\nu^\star$, the Min-Sum algorithm yields linear time solvers for the voltage problem, with respect to the $L$-norm.

\subsection{Convergence Results for the Flow Problem}\label{sec:flow result}

We now analyze the convergence behavior of the Min-Sum algorithm when applied to the flow problem: Algorithm \ref{alg:Min-Sum, quadratic messages, no leaves}. 
We present the results for the $\ell_\infty$ norm, defined as $\| x \|_\infty := \max_{e\in \vec E} x_e$ for any $x\in\R^{\vec E}$. The results for the $\ell_2$ norm, defined as $\| x \|_2 := \sqrt{x^Tx}$, are analogous to the results in the $L$-norm for the voltage case, as we describe in Remark \ref{rem:analysis ell2}. 

For the case $d=2$, Lemma \ref{lem:ring} immediately yields the following result.
\begin{corollary}[Regular graphs, $d=2$]\label{cor:cycle conv}
Consider the setting of Lemma \ref{lem:ring}. Let $\omega_M:=\max_{\{v,w\}\in E} W_{vw}$. Then,
$
	\| x^\star - \hat x^t \|_\infty
	\le \frac{2\omega_M}{2t+1} \| \nu^\star \|_\infty.
$
If $\omega_M$ is independent of $n$ and $m$, then the Min-Sum algorithm yields $\hat x$ such that $\| x^\star - \hat x \|_\infty \le \varepsilon \| \nu^\star \|_\infty$ with running time $O(m/\varepsilon)$.
\end{corollary}

\begin{proof}
From Lemma \ref{lem:ring}, by the triangle inequality, we have
$
	|x^\star_{e} - \hat x^t_{e} |
	\le \frac{\omega_M}{2t+1}
	(
	|\nu^\star_{\rho(v-t)}| + |\nu^\star_{\rho(v+t+1)}|
	),
$
and the first bound follows by the triangle inequality for norms. The running time is immediate, as each iteration of the algorithm requires a complexity $O(m)$.
\end{proof}

Corollary \ref{cor:cycle conv} shows that the flow Min-Sum algorithm on cycles converges in linear time to the problem solution, as opposed to the divergent behavior observed in the voltage case.

Let us now consider the case $d\ge 3$. As in the voltage problem, also in this case the convergence of the Min-Sum algorithm can be controlled by the TV distance between the probability distributions of properly-defined non-backtracking random walks that originate from neighbor nodes. As before, we also consider an averaged version of the Min-Sum algorithm where the output of two consecutive iterations is properly averaged:
$
	\hat x_{\text{ave}}^t := 
	\frac{c_{d,t-1}}{c_{d,t-1} + c_{d,t}} \hat x^{t-1} 
	+ 
	\frac{c_{d,t}}{c_{d,t-1} + c_{d,t}} \hat x^{t}.
$
Analogously to the voltage case, let $Q^{(t,w)}:=(P^{(t-1,w)} + P^{(t,w)})/2$ and
\begin{align*}
	\widetilde\gamma(t) := \max_{\{v,w\}\in E} \| P^{(t,w)}_{v} - P^{(t,v)}_{w} \|_{TV},
	\qquad\widetilde\gamma_{\text{ave}}(t) := \max_{\{v,w\}\in E} \| Q^{(t,w)}_{v} - Q^{(t,v)}_{w} \|_{TV}.
\end{align*}

\begin{corollary}[Regular graphs, $d\ge 3$, equal weights]\label{cor:convergence}
Setting of Lemma \ref{lem:regulargraph}. For $t\ge4$,
\begin{align*}
	\|x^\star - \hat x^t \|_\infty
	\le 2 \, \omega\, \frac{d-1}{d}
	\,\widetilde\gamma(t)\, \|\nu^\star\|_\infty,
	\qquad
	\|x^\star - \hat x_{\emph{ave}}^t \|_\infty
	\le 2 \, \omega\, \frac{d-1}{d}
	\,\widetilde\gamma_{\emph{ave}}(t)\, \|\nu^\star\|_\infty.
\end{align*}
\end{corollary}

\begin{proof}
For any $t\ge 1$, define the matrix $\widetilde\Delta^{(t)}\in\mathbb{R}^{\vec E\times V}$ as 
$\widetilde\Delta^{(t)}_{ez} := P^{(t,w)}_{vz} - P^{(t,v)}_{wz}, e=(v,w)\in \vec E$,
$z\in V.$
From Lemma \ref{lem:regulargraph} we have
$
	x^\star - \hat x^t =
	\frac{\omega}{c_{d,t}} \frac{d-1}{d}
	\widetilde\Delta^{(t)}\nu^\star,
$
and the first bound follows from the sub-multiplicative property of the induced $\ell_\infty$ matrix norm, noticing that $\widetilde\gamma(t) = \| \widetilde\Delta^{(t)} \|_\infty/2$, and using that $c_{d,t} \ge 1$ from Proposition \ref{prop:constant d t}. The second bound follows analogously, as for any $t\ge 4$ we have
$
	(c_{d,t-1} + c_{d,t}) (x^\star - \hat x_{\text{ave}}^t)
	=
	c_{d,t-1}(x^\star - \hat x^{t-1})
	+ c_{d,t}(x^\star - \hat x^t)
	=
	\omega\,
	\frac{d-1}{d}
	(\widetilde\Delta^{(t-1)} + \widetilde\Delta^{(t)})\nu^\star,
$
and
$\widetilde\gamma_{\text{ave}}(t) = \| ( \widetilde\Delta^{(t-1)} + \widetilde\Delta^{(t)})/2 \|_\infty/2$.
\end{proof}

Corollary \ref{cor:convergence} shows that the convergence in the $\ell_\infty$ norm of the Min-Sum estimate $\hat x^t$ and its averaged version $\hat x_{\text{ave}}^t$ is controlled by the behavior with time of the quantities $\widetilde\gamma(t)$ and $\widetilde\gamma_{\text{ave}}(t)$, respectively, which represent uniform bounds on the TV distance between probability distributions of non-backtracking random walks that make $t$ steps originating from neighbor nodes, conditioned on the event that the two walks avoid each others at the first step.
Proposition \ref{proof:Non-backtracking random walks} in Appendix \ref{app:technical} yields a recursive formula to compute the matrices $P^{(t,w)}$'s in terms of the adjacency matrix of the undirected unweighted graph $(V,E)$. This formula can then be used to compute the quantities $\widetilde\gamma(t)$ and $\widetilde\gamma_{\text{ave}}(t)$. Numerical results for a few classes of $d$-regular graphs have been presented in the previous section. See Figure \ref{fig:complete decay} (right)\footnote{It is easy to see that $\widetilde\gamma(1)=0$ for complete graphs, which explains the points at $t=1$ in Figure \ref{fig:complete decay} (right).}, Figure \ref{fig:ring decay} (right), and Figure \ref{fig:torus decay} (right). The following theorem summarizes the algorithmic complexity of the Min-Sum algorithm $x^t$ in the $\ell_\infty$ norm, under the assumption that $\widetilde\gamma(t)$ decays exponentially or polynomially, as attested numerically in the examples above (an analogous statement holds for the output $x_{\text{ave}}^t$, replacing $\widetilde\gamma(t)$ with $\widetilde\gamma_{\text{ave}}(t)$).

\begin{theorem}[Min-Sum flow, running time, $\ell_\infty$ norm]\label{thm:runningtime flow}
Setting of Lemma \ref{lem:regulargraph}. Let $\alpha>0,\beta>0$, and $\omega$ be independent of $n$ and $m$.
\begin{enumerate}
\item If $\widetilde\gamma(t) \le \alpha \exp(-\beta t)$ for $t\ge 0$, then the Min-Sum algorithm yields an estimate $\hat x$ such that $\| x^\star - \hat x \|_\infty \le \varepsilon \| \nu^\star \|_\infty$ with running time $O(m \log 1/\varepsilon)$.
\item If $\widetilde\gamma(t) \le \alpha / t^{\beta}$ for $t\ge 0$, then the Min-Sum algorithm yields an estimate $\hat x$ such that $\| x^\star - \hat x \|_\infty \le \varepsilon \| \nu^\star \|_\infty$ with running time $O(m /\varepsilon^{1/\beta})$.
\end{enumerate}
\end{theorem}
\begin{proof}
It follows from Corollary \ref{cor:convergence}, analogously to the proof of Theorem \ref{thm:runningtime}.
\end{proof}

Theorem \ref{thm:runningtime flow} shows that for certain classes of $d$-regular graphs, as characterized by the decaying properties of $\widetilde\gamma(t)$, the Min-Sum algorithm yields linear time solvers for the flow problem, with respect to the $\ell_\infty$ norm. The following remark shows that an analogous result to Theorem \ref{thm:runningtime flow} can be derived with respect to the $\ell_2$ norm, with an additional assumption in the spirit of the one used in Theorem \ref{thm:runningtime} for the voltage case.

\begin{remark}[Analysis in $\ell_2$ norm]\label{rem:analysis ell2}
For $d$-regular graphs with $d\ge 3$, the analysis in the $\ell_2$ norm for the flow problem follows the same steps as the analysis in the $L$-norm for the voltage problem. In fact, if one defines the matrices $\Delta^{(t)}$ and $\widetilde\Delta^{(t)}\in\mathbb{R}^{\vec E\times V}$ respectively as
$
	\Delta^{(t)}_{ez} := P^{(t)}_{vz} - P^{(t)}_{wz}
$
and
$\widetilde\Delta^{(t)}_{ez} := P^{(t,w)}_{vz} - P^{(t,v)}_{wz}$,
for any $e=(v,w)\in \vec E, z\in V$, then
$
	\| x^\star - \hat x^t \|_2 =
	\frac{\omega}{c_{d,t}} \frac{d-1}{d}
	\| \widetilde\Delta^{(t)}\nu^\star \|_2,
$
while
$
	\| \nu^\star - \hat \nu^t \|_L =
	\frac{1}{b_{d,t}} 
	\| \Delta^{(t)}\nu^\star \|_2.
$
In analogy with Theorem \ref{thm:runningtime}, the results of Theorem \ref{thm:runningtime flow} hold when replacing $\| x^\star - \hat x \|_\infty \le \varepsilon \| \nu^\star \|_\infty$ with $\| x^\star - \hat x \|_2 \le \varepsilon \| \nu^\star \|_L$, upon assuming that there exists $\delta>1$ independent of $n$ and $m$ such that 
$
\sum_{\{v,w\}\in E} W_{vw} 
	\big( \frac{(P^{(t,w)}\nu)_v - (P^{(t,v)}\nu)_w}{\| P^{(t,w)}_{v} - P^{(t,v)}_{w} \|_{TV}} \big)^2 \le \delta^2 \| \nu^\star \|_L^2,
$
for $t\ge 0$,
which corresponds to the assumption $\| \nu^\star \|_{L,P^{(t)}} \le \delta \| \nu^\star \|_L$ in Theorem \ref{thm:runningtime}.
\end{remark}

\section{Results for General Weighted Graphs}\label{sec:proofs}

This section develops the characterization of the error committed by the voltage and flow algorithms on general weighted graphs in terms of hitting times of ordinary diffusion random walks defined on the computation trees that are obtained by unraveling the operations of the algorithms with time (Theorem \ref{thm:errorcharacterization} and Lemma \ref{lem:computationdevice} for the flow case, and Theorem \ref{thm:errorcharacterization voltage} for the voltage case). As applications, we specialize the general results here developed to prove the error characterization lemmas given in Section \ref{sec:results}, i.e., Lemma \ref{lem:ring} and Lemma \ref{lem:regulargraph}.

While the main architecture of the results for the voltage problem (Algorithm \ref{alg:Min-Sum voltages, quadratic messages}) and the flow problem (Algorithm \ref{alg:Min-Sum, quadratic messages, no leaves}) is similar, there are some key differences that need to be taken into account. First, as the voltage algorithm updates functions (i.e., messages) on the vertices while the flow algorithm updates functions on the edges, the two algorithms give rise to different topologies for the computation tree. Second, as the voltage problem is an unconstrained optimization problem while the flow problem is an optimization problem with constraints, a different mechanism to find the fix point(s) of the algorithm is needed. As in the literature, to the best of our knowledge, previous analyses of the Min-Sum algorithm for quadratic and, more generally, convex problems only focused on the unconstrained case, we first present the full details of the results for the flow problem. Then, we give an overview of the results for the voltage problem, outlining in details only the parts where the analysis differs significantly from the one done for the flow problem.

\subsection{Flow Problem, Algorithm \ref{alg:Min-Sum voltages, quadratic messages}}\label{proofs:flow}

Our framework builds on the general scheme presented in \cite{MVR10} for unconstrained optimization problems, as we investigate the evolution of the Min-Sum algorithm under a set of linear perturbations. Henceforth, we consider Algorithm \ref{alg:Min-Sum, quadratic messages, no leaves} when the initial messages are parametrized by $\{R^{0}_{e\rightarrow v}=R_{ee}\}, \{r^{0}_{e\rightarrow v}=p_{e\rightarrow v}\}$, for a certain set of real numbers $p=\{p_{e\rightarrow v}\}$.
As far as the analysis is concerned, it is convenient (and more amenable to generalizations) to consider the form of Algorithm \ref{alg:Min-Sum for min-cost network flow}. For a given choice of perturbation parameters $p=\{p_{e\rightarrow v}\}$, define the initial messages as
$
	\mu^{0}_{e\rightarrow v}(\,\cdot\,,p)
	: z\in\R \longrightarrow
	\mu^{0}_{e\rightarrow v}(z,p)
	= \frac{1}{2} 
	R_{ee}z^2 + p_{e\rightarrow v}z,
$
and denote by $\{\mu^{s}_{e\rightarrow v}(\,\cdot\,,p)\}$, $s\ge 1$, the corresponding sequence of messages generated by the Min-Sum algorithm. Denote the estimates at time $t$ as $\hat x^t(p)_e := \arg\min_{z\in\R} \mu^{t}_{e}(z,p)$, where $\mu^{t}_{e}(\,\cdot\,,p)$ is the corresponding belief function. We use the notation $\mathcal{N}(v) := \{w\in V: \{v,w\}\in E\}$ to denote the neighborhood of vertex $v$ in $G$.

\subsubsection{Computation Tree (Flow Problem)}\label{sec:comp tree}
To investigate the convergence behavior of the Min-Sum algorithm, we recall the concept of the computation tree, which is the graph obtained by unraveling the operations of the algorithm with time \citep{gallager1963,WF01,TJ02}.
Given an edge $\tilde e=(\tilde v,\tilde w)\in \vec E$, the computation tree rooted at $\tilde e$ of depth $t$ supports the optimization problem that is obtained by unfolding the computations involved in the Min-Sum estimate $\hat x^t(p)_{\tilde e}$.
Formally, the computation tree is a directed tree $\vec{\TT}=(\VV,\vec\EE)$ (throughout, we use the double-struck notation to refer to quantities related to the computation tree) where each vertex in $\VV$ is mapped to a vertex in $V$ through a map $\sigma:\VV \rightarrow V$ that preserves the edge structure, namely, if $\ee=(\vv,\ww)\in \vec \EE$ then $(\sigma(\vv),\sigma(\ww))\in \vec E$. Henceforth, we use the notation $\sigma(\ee):=(\sigma(\vv),\sigma(\ww))$. The tree $\vec{\TT}$ is defined iteratively. Initially, at level $0$, the tree corresponds to a single root edge $\tilde \ee=(\tilde \vv,\tilde \ww)\in\vec\EE$ corresponding to $\tilde e$, i.e., $\sigma(\tilde \ee)=\tilde e$. At this stage, $\tilde\vv$ and $\tilde\ww$ are the leaves of the tree. For the remaining $t$ levels, the following procedure is repeated. The leaves in the tree are examined. Given a leaf $\vv$ with $\sigma(\vv)=v$ that is connected to a vertex $\ww$ with $\sigma(\ww)=w$, for any $z\in\mathcal{N}(v)\setminus w$, a vertex $\zz$ with $\sigma(\zz)=z$ and a directed edge $(\vv,\zz)$ (resp.\ $(\zz,\vv)$) are added to the next level of the tree if $(v,z)\in\vec E$ (resp.\ $(z,v)\in\vec E$).
Figure \ref{fig:computationtree} gives an example.
We denote the set of vertices and edges in the $k$-th level of the tree respectively by $\VV^{k}\subset\VV$ and $\EE^{k}\subset\EE$.
In what follows we also extend the usual neighborhood notation to vertices and edges in the graph $\TT$, namely, $\partial\vv$ is the set of edges in $\TT$ that are connected to node $\vv$, and $\partial\ee$ is the set of vertices in $\TT$ that are connected to edge $\ee$.

\begin{figure}[h!]
\begin{subfigure}{.45\textwidth}
\centering
\begin{tikzpicture}
[scale=0.75,every node/.style={draw,circle,scale=0.75,solid,minimum size=0.72cm,inner sep=0pt},
->,>=latex,shorten >=1pt,shorten <=1pt]

\def \x {2}
\def \y {-2}

\node (1) at (0,0) {1};
\node (2) at (\x,0) {2};
\node (3) at (\x,\y) {3};
\node (4) at (0,\y) {4};

\draw (1) -- (2);
\draw (2) -- (3);
\draw (4) -- (3);
\draw (4) -- (1);
\draw (3) -- (1);

\end{tikzpicture}
\end{subfigure}
\begin{subfigure}{.45\textwidth}
\centering
\begin{tikzpicture}
[scale=0.75,every node/.style={draw,circle,scale=0.75,solid,minimum size=0.72cm,inner sep=0pt},
,>=latex,shorten >=1.3pt,shorten <=1.3pt,
level distance=1.2cm, 
level 1/.style={sibling distance=2cm},
level 2/.style={sibling distance=1cm}]

\node (Root1) {$1$}
child {
	node {$3$} edge from parent [<-]
	child { 
		node {$2$}  edge from parent [<-]
		child { node{$1$} edge from parent [<-] }
	}
	child {
		node {$4$}  edge from parent [<-]
		child { node{$1$} edge from parent [->] }
	}
	}
child {
	node {$4$} edge from parent [<-]
	child {
		node {$3$} edge from parent [->]
			child { node{$1$} edge from parent [->]}
			child { node{$2$} edge from parent [<-]}
		}
	};
\node[right=2.7cm of Root1] (Root2) {$2$}
	child {
		node {$3$} edge from parent [->]
			child { node {$1$} edge from parent [->]
				child { node{$2$} edge from parent [->]}
				child { node{$4$} edge from parent [<-]}
			}
			child { node {$4$}[right] edge from parent [<-]
				child { node{$1$} edge from parent [->]}
			}
	};
\draw[->] (Root1) -- (Root2);
\end{tikzpicture}
\end{subfigure}
\caption{Graph $\vec G=(V=\{1, 2, 3, 4\},\vec E)$ (left) and computation tree $\vec\TT=(\VV,\vec\EE)$ of depth $t=3$ with root edge corresponding to $(\text{1, 2})\in \vec E$ (right). Vertices in $\vec\TT$ are labeled by the corresponding vertices in $\vec G$.}
\label{fig:computationtree}
\end{figure}
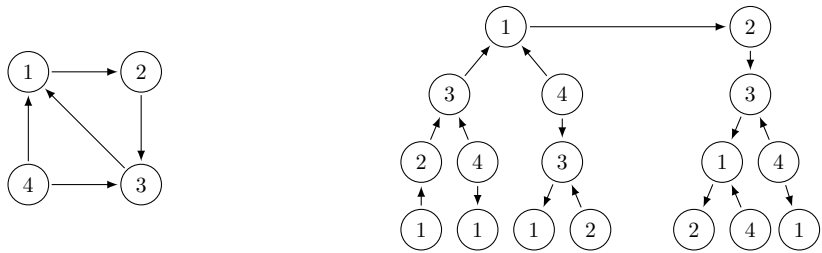

The $t$-th (last) level of the tree $\vec{\TT}$ is of particular relevance, as it supports the initial conditions of the Min-Sum algorithm, as we are now about to describe. Let $\R\in\R^{\vec \EE\times \vec \EE}$ be the diagonal matrix defined by $\R_{\ee\ee}:=R_{\sigma(\ee)\sigma(\ee)}$. For a given choice of the perturbation parameters $p=\{p_{e\rightarrow v}\}$, define $\pp(p)\in\R^{\vec \EE}$ by $\pp(p)_\ee := p_{\sigma(\ee)\rightarrow \sigma(\vv)}$ if $\ee\in\vec\EE^t, \vv=\partial \ee \cap \VV^{t-1}$, and equals to $0$ otherwise.
Let $\bar \VV := \VV\setminus \VV^t$ (throughout, we use the bar notation to denote quantities related to a graph in which a node has been removed). Define the matrix $\bar\AA\in\R^{\bar\VV \times \vec \EE}$ as $\bar\AA_{\vv\ee} := A_{\sigma(\vv)\sigma(\ee)}$ if $\ee\in\partial \vv$ and $\bar\AA_{\vv\ee} :=0$ otherwise.
Define the vector $\bar\bb\in\R^{\bar\VV}$ as $\bar\bb_\vv := b_{\sigma(\vv)}$.
Consider the following problem supported on the tree $\vec{\TT}$, over $\xx\in\R^{\vec\EE}$:
\begin{align}
\begin{aligned}
	\text{minimize }\quad    & \frac{1}{2} \xx^T \R \xx + \pp(p)^T \xx\\
	\text{subject to }\quad & \bar\AA\xx =\bar\bb.
\end{aligned}
\label{network flow computation tree}
\end{align}

Let $\xx^\star(p)$ denote the unique optimal solution of this problem, as a function of the perturbation $p$. The relationship between the Min-Sum algorithm and the computation tree is made explicit by the following lemma, which can be easily established by inductively examining the operations of Algorithm \ref{alg:Min-Sum for min-cost network flow}.

\begin{lemma}\label{lem:computation tree}
Given the setting of this section, $\hat x^t(p)_{\tilde e}=\xx^\star(p)_{\tilde \ee}$.
\end{lemma}

We now introduce a few quantities related to the computation tree $\vec\TT$ and the problem \eqref{network flow computation tree} that will play a key role in our analysis.
The signed vertex-edge incidence matrix $\AA\in\R^{\VV\times\vec \EE}$ associated to the directed graph $\vec{\TT}$ is defined for each $\vv\in\VV$, $\ee\in\vec \EE$ as $\AA_{\vv\ee} := 1$ if edge $\ee$ leaves node $\vv$, $\AA_{\vv\ee} := -1$ if edge $\ee$ enters node $\vv$, and $\AA_{\vv\ee} := 0$ otherwise.
It is easy to verify that $\bar\AA$ is the submatrix of $\AA$ that corresponds to the rows associated to $\bar\VV$. 
Let $\WW\in\R^{\VV\times \VV}$ be a symmetric matrix that assigns a positive weight to every edge in $\vec \EE$, defined as $\WW_{\vv\ww}:=W_{\sigma(\vv)\sigma(\ww)}$ if either $(\vv,\ww)\in\vec\EE$ or $(\ww,\vv)\in\vec\EE$, and $\WW_{\vv\ww}:=0$ otherwise. Henceforth, we think of $\vec\TT=(\VV,\vec\EE,\WW)$ as the corresponding weighted graph, and $\TT=(\VV,\EE,\WW)$ as the undirected graph naturally associated to $\vec\TT$, where the edge set $\EE$ is obtained by removing the orientation from the edges in $\vec\EE$. As $\TT$ is connected by assumption, it is easy to show that the range of $\AA$ is spanned by the vectors orthogonal to the all-ones vector. It follows that the matrix $\bar \AA$ has full row rank.
Define the weighted degree of a vertex $\vv$ by $\dd_\vv:=\sum_{\ww\in \VV} \WW_{\vv\ww}$, and let $\DD\in\R^{\VV\times \VV}$ be the diagonal matrix defined by $\DD_{\vv\vv}:=\dd_\vv$. The Laplacian $\LL$ of the graph $\vec\TT$ is the matrix defined as $\LL:=\DD-\WW$. Let $\bar\LL\in\R^{\bar\VV\times\bar\VV}$ denote the submatrix of $\LL$ that corresponds to the rows and columns associated to $\bar\VV$.
Note that $\R_{\ee\ee}:=1/\WW_{\vv\ww}$.
It is easy to verify that the following equalities holds: $\LL=\AA\R^{-1}\AA^T$ and $\bar\LL=\bar\AA\R^{-1}\bar\AA^T$. We also define the transition matrix of the ordinary diffusion random walk on $\TT$ as the matrix $\PP:=\DD^{-1}\WW$.

\subsubsection{Fix Point (Flow Problem)}
Clearly, the choice $p=0$ for the linear perturbation parameters yields the original (unperturbed) algorithm. The next theorem below shows that for a particular choice of $p$ the Min-Sum algorithm yields the optimal solution at any time $t\ge 1$. The optimal choice of $p$ is a function of the optimal Lagrangian multipliers for problem \eqref{primal}. Recall from \eqref{lagrangian-primal} that the Lagrangian of this optimization problem is given by
$
	\mathcal{L}(x,\nu) := \frac{1}{2} x^T R x + \nu^T (b - A x),
$
where $\nu = (\nu_v)_{v\in V}$ is the vector formed by the Lagrangian multipliers.
Define the function $\Psi$ from $\mathbb{R}^{\vec E}\times\mathbb{R}^V$ to $\mathbb{R}^{\vec E}\times\mathbb{R}^V$ as
$
	\Psi
	( 
	x,
	\nu 
	)
	:=
	(
	\nabla_x \mathcal{L}(x,\nu),
	A x - b
	)
	=
	(
	R x - A^T\nu,
	A x - b 
	).
$
Recall that $x^\star$ denotes the unique optimal solution of problem \eqref{primal}. As the constraints are linear, by the Lagrange multiplier theorem there exists $\nu^\star\in\mathbb{R}^V$ so that
$
	\Psi
	( 
	x^\star,
	\nu^\star
	)
	=
	0,
$
i.e.,
\begin{align}
	\begin{cases}
	R_{ee}x^\star_e - (A^T \nu^\star)_e = 0
	\text{ for } e\in\vec E,\\
	\sum_{e\in\vec E} A_{ve} x^\star_e = b_v
	\text{ for } v\in V.
	\end{cases}
	\label{KKT}
\end{align}

Note that the first line in \eqref{KKT} represents Ohm's law, and the second line represents Kirchhoff's law. Given the pair $(x^\star,\nu^\star)$, we are now ready to define the set of parameters $p^\star=\{p^\star_{e\rightarrow v}\}$ so that the Min-Sum algorithm yields the optimal solution at any time step.

\begin{theorem}\label{thm: opt choice p}
For each $e\in \vec E$, $v\in\partial e$, let $w=\partial e\setminus v$ and define
$
	p^\star_{e\rightarrow v}
	:= 
	- A_{we} \nu^\star_w.
$
Then, for any $\tilde e\in\vec E$, $t\ge 1$, we have $\hat x^t(p^\star)_{\tilde e} = x^\star_{\tilde e}$.
\end{theorem}

\begin{proof}
Fix $\tilde e\in \vec E$, $t\ge 1$. Consider the setting in Section \ref{sec:comp tree}, and let $\vec{\TT}$ be the computation tree of depth $t$ rooted at $\tilde e$.
The Lagrangian of the problem \eqref{network flow computation tree} corresponding to the choice $p=p^\star$ is the function $\LL$ from $\mathbb{R}^{\vec \EE}\times\mathbb{R}^{\bar\VV}$ to $\mathbb{R}$ defined by
$
	\mathbb{L}(\xx,\bbnu) := \frac{1}{2} \xx^T \R \xx + \pp(p^\star)^T \xx
	 + \bbnu^T (\bar\bb - \bar\AA \xx),
$
where $\bbnu = (\bbnu_\vv)_{\vv\in \bar\VV}$.
Let $\mathbb{\Psi}$ from $\mathbb{R}^{\vec \EE}\times\mathbb{R}^{\bar\VV}$ to $\mathbb{R}^{\vec \EE}\times\mathbb{R}^{\bar\VV}$ be defined as
$
	\mathbb{\Psi}
	( 
	\xx,
	\bbnu
	)
	:=
	(
	\nabla_\xx \LL(\xx,\bbnu),
	\bar\AA \xx - \bar\bb
	)
	=
	(
	\R \xx + \pp(p^\star) - \bar\AA^T\bbnu,
	\bar\AA \xx - \bar\bb 
	).
$
The Lagrangian multiplier theorem says for the unique minimizer $\xx^\star\equiv \xx^\star(p^\star)\in\mathbb{R}^{\vec \EE}$ of \eqref{network flow computation tree} there exists $\bbnu^\star \in\mathbb{R}^{\bar\VV}$ so that
$
	\mathbb{\Psi}
	( 
	\xx^\star,
	\bbnu^\star
	)
	=
	0,
$
namely,
\begin{align*}
	\begin{cases}
	\R_{\ee\ee}\xx^\star_\ee
	- (\bar\AA^T \bbnu^\star)_\ee = 0,
	\ee\in\vec \EE \setminus \vec \EE^t,\\
	\R_{\ee\ee}\xx^\star_\ee + p^\star_{\sigma(\ee)\rightarrow \sigma(\vv)}
	\!-\! (\bar\AA^T \bbnu^\star)_\ee = 0, \ee\in\vec\EE^t, \vv=\partial \ee \cap \VV^{t-1},\\
	\sum_{\ee\in\vec\EE} \bar\AA_{\vv\ee} \xx^\star_\ee = \bar\bb_\vv, \vv\in\bar\VV.
	\end{cases}
\end{align*}
Using \eqref{KKT} it is easy to check that the choice $\xx^\star_\ee = x^\star_{\sigma(\ee)}$ and $\bbnu^\star_{\vv}=\nu^\star_{\sigma(\vv)}$ satisfies the above system of equations. The proof is concluded by Lemma \ref{lem:computation tree}.
\end{proof}

\subsubsection{Error Characterization (Flow Problem)}
The next lemma characterizes the sensitivity of the $\tilde \ee$-th component (root edge of $\vec{\TT}$) of the optimal solution of the optimization problem \eqref{network flow computation tree} with respect to perturbations of the parameters $p_{e\rightarrow v}$'s, which are supported on $\vec{\EE}^t$.

\begin{lemma}\label{lem:corr tree}
Consider the setting in Section \ref{sec:comp tree}. Fix $\tilde e=(\tilde v,\tilde w)\in \vec E$, $t\ge 1$, and let $\vec{\TT}$ be the computation tree of depth $t$ rooted at $\tilde e$, with root edge $\tilde \ee=(\tilde \vv,\tilde \ww)\in\vec\EE$. For any $p=\{p_{e\rightarrow v}\}$, $e\in \vec E$ and $v\in \partial e$, we have
\begin{align*}
	\frac{\partial}{\partial p_{e\rightarrow v}} \xx^\star(p)_{\tilde \ee}
	&= \frac{1}{R_{\tilde e\tilde e}} \sum_{\vv\in\VV^{t-1}} 
	(\bar\LL^{-1}_{\tilde \vv\vv} - \bar\LL^{-1}_{\tilde \ww\vv})
	(\yy_{e\rightarrow v})_\vv,
\end{align*}
where $\yy_{e\rightarrow v}\in\R^{\VV^{t-1}}$ is defined, for each $\vv\in\VV^{t-1}$, as
$
	(\yy_{e\rightarrow v})_\vv := 
	\frac{A_{ve}}{R_{ee}}\mathbf{1}_{\sigma(\vv)=v}\mathbf{1}_{\sigma(\partial \vv\cap \vec \EE^t)=e},
$
where $\mathbf{1}_S$ is the indicator function defined as $\mathbf{1}_S = 1$ if statement $S$ is true, $\mathbf{1}_S = 0$ otherwise.
\end{lemma}

\begin{proof}
By Proposition \ref{prop:quadratic} in Appendix \ref{app:technical}, the solution of problem \eqref{network flow computation tree} reads
$
	\xx^\star(p) 
	= \R^{-1} \bar\AA^T \bar\LL^{-1} \bar\bb + (\R^{-1}\bar\AA^T\bar\LL^{-1} \bar\AA - I) \RR^{-1} \pp(p),
$
where $I\in\R^{\vec \EE \times \vec \EE}$ is the identity matrix. As $\pp(p)$ is supported on the edges of $\TT$ that are connected to its leaves and $\tilde \ee$ is the root edge, we have
$
	\frac{\partial}{\partial p_{e\rightarrow v}} \xx^\star(p)_{\tilde\ee}
	= \big(\R^{-1}\bar\AA^T \bar\LL^{-1} \bar\AA\R^{-1}\frac{\partial}{\partial p_{e\rightarrow v}} \pp(p)\big)_{\tilde\ee}
	= \frac{1}{R_{\tilde e \tilde e}} \sum_{\vv\in\VV^{t-1}} 
	(\bar\LL^{-1}_{\tilde\vv\vv} - \bar\LL^{-1}_{\tilde\ww\vv})
	(\yy_{e\rightarrow v})_\vv,
$
where 
$
	(\yy_{e\rightarrow v})_\vv = 
	(\bar\AA\R^{-1}\frac{\partial}{\partial p_{e\rightarrow v}} \pp(p) )_{\vv}
	=
	\frac{A_{ve}}{R_{ee}}\mathbf{1}_{\sigma(\vv)=v}\mathbf{1}_{\sigma(\partial \vv\cap \vec \EE^t)=e}
$
for each $\vv\in\VV^{t-1}$.
\end{proof}

As a consequence of Lemma \ref{lem:corr tree}, we immediately have the following characterization of the error committed by Algorithm \ref{alg:Min-Sum, quadratic messages, no leaves} with initial conditions $\{R^0_{e\rightarrow v} = R_{ee}\}$, $\{r^0_{e\rightarrow v} = 0\}$.

\begin{theorem}\label{thm:errorcharacterization}
Consider the setting in Section \ref{sec:comp tree}. Fix $\tilde e=(\tilde v,\tilde w)\in \vec E$, $t\ge 1$, and let $\vec{\TT}$ be the computation tree of depth $t$ rooted at $\tilde e$, with root edge $\tilde\ee=(\tilde\vv,\tilde\ww)\in\vec\EE$. We have
\begin{align*}
	x^\star_{\tilde e} - \hat x^t_{\tilde e}
	&= \WW_{\tilde\vv\tilde\ww}
	\sum_{\vv\in\VV^{t-1}} 
	(\bar\LL^{-1}_{\tilde\vv\vv} - \bar\LL^{-1}_{\tilde\ww\vv})
	\sum_{\ww\in\VV^t} 
	\WW_{\vv\ww} \nu^\star_{\sigma(\ww)}.
\end{align*}
\end{theorem}

\begin{proof}
From Theorem \ref{thm: opt choice p} and Lemma \ref{lem:corr tree}, by the fundamental theorem of calculus and the chain rule of differentiation, we get
$
	x^\star_{\tilde e} - \hat x^t_{\tilde e}
	= \hat x^t(p^\star)_{\tilde e} - \hat x^t(0)_{\tilde e}
	= \int_0^1 d\theta\, \frac{d}{d \theta} \xx^\star(\theta p^\star)_{\tilde\ee}
	= \int_0^1 d\theta \sum_{e\in \vec E, v\in \partial e}
	\frac{\partial}{\partial p_{e\rightarrow v}} \xx^\star(\theta p^\star)_{\tilde\ee}
	\, p^\star_{e\rightarrow v}
	= \WW_{\tilde\vv\tilde\ww}
	\sum_{\vv\in\VV^{t-1}} 
	(\bar\LL^{-1}_{\tilde\vv\vv} - \bar\LL^{-1}_{\tilde\ww\vv})
	\sum_{e\in \vec E, v\in \partial e} (\yy_{e\rightarrow v})_\vv
	\, p^\star_{e\rightarrow v}.
$
By Theorem \ref{thm: opt choice p}
$
	p^\star_{e\rightarrow v} = - A_{we} \nu^\star_w,
$
$w=\partial e\setminus v$, and
$
	\sum_{\substack{e\in \vec E\\v\in \partial e}}
	(\yy_{e\rightarrow v})_\vv
	\, p^\star_{e\rightarrow v}
	= \sum_{\ww\in\VV^t} 
	\WW_{\vv\ww} \nu^\star_{\sigma(\ww)}.
$
\end{proof}

We now use the error characterization provided by Theorem \ref{thm:errorcharacterization} and the general connection between restricted Laplacians and random walks developed in Appendix \ref{sec:Laplacians and random walks} to prove the results in Section \ref{sec:results}. Henceforth, for each $\vv\in \VV$, let $\mathbf{P}_{\vv}$ be the law of a time homogeneous Markov chain $X_{0},X_{1},X_2,\ldots$ on $\VV$ with transition matrix $\PP:=\DD^{-1}\WW$ and initial condition $X_0=\vv$. Analogously, denote by $\mathbf{E}_{\vv}$ the expectation with respect to this law. The hitting time to a set $\mathbb{Z}\subseteq \VV$ is defined as
$
	T_\mathbb{Z} := \inf\{k\ge0 : X_k \in \mathbb{Z}\}.
$
We also extend the notation $\mathcal{N}(\vv) := \{\ww\in\VV: \{\vv,\ww\}\in \EE\}$ to denote the vertex-neighborhood of vertex $\vv$ in $\TT$.

The next lemma, which builds on the results developed in Appendix \ref{sec:Laplacians and random walks}, gives an expression for the restricted Laplacian in terms of hitting probabilities of random walks. This result represents the computational device that we will use repeatedly in what follows.

\begin{lemma}\label{lem:computationdevice}
Consider the setting in Section \ref{sec:comp tree}. Fix $\tilde e=(\tilde v,\tilde w)\in \vec E$, $t\ge 1$, and let $\vec{\TT}$ be the computation tree of depth $t$ rooted at $\tilde e$, with root edge $\tilde \ee=(\tilde \vv,\tilde \ww)\in\vec\EE$. For each $\vv\in\bar\VV$, $\ww\in\VV^{t-1}$, with $\zz=\mathcal{N}(\ww) \setminus \VV^t$, we have
\begin{align*}
	\bar \LL^{-1}_{\vv\ww} 
	=
	\frac{\mathbf{P}_{\vv}(T_{\ww}<T_{\VV^{t}})}
	{[1-\mathbf{P}_{\zz}(T_{\ww}<T_{\VV^{t}}) \,\mathbf{P}_\ww(X_1=\zz)]
	\,\dd_\ww}.
\end{align*}
\end{lemma}

\begin{proof}
Fix $\vv\in\bar\VV$, $\ww\in\VV^{t-1}$, and let $\zz=\mathcal{N}(\ww) \setminus \VV^t$. Define the random variable
$
	A := \sum_{k = 0}^{T_{\VV^{t}}} \mathbf{1}_{X_k=\ww}.
$
From Proposition \ref{prop:reducedlapandgreenfunction} we have
$
	\bar \LL^{-1}_{\vv\ww} 
	= \frac{1}{\dd_\ww} \mathbf{E}_{\ww}
	[A]
	\,\mathbf{P}_{\vv}(T_{\ww}<T_{\VV^{t}}).
$
By the Markov property, performing a first step analysis for the random walk, we get
$
	\mathbf{E}_{\ww} [A]
	= 1 + \sum_{\tilde\zz\in\mathcal{N}(\ww)} \!\! \mathbf{E}_{\tilde\zz} [A] \mathbf{P}_\ww(X_1=\tilde\zz)
	= 1 + \mathbf{E}_{\zz} [A] \mathbf{P}_\ww(X_1=\zz),
$
where we used that $\mathbf{E}_{\tilde\zz} [A]=0$ for any $\tilde\zz\in\VV^t$.
Conditioning on the event $\{T_{\ww}<T_{\VV^t}\}$ and its complementary, using that $\mathbf{E}_{\zz}[A| T_{\ww}> T_{\VV^t}] =0$, we also get
$
	\mathbf{E}_{\zz}[A]
	= \mathbf{E}_{\zz}[A| T_{\ww}<T_{\VV^t}] \, \mathbf{P}_{\zz}(T_{\ww}<T_{\VV^t}).
$
By the strong Markov property $\mathbf{E}_{\zz}[A| T_{\ww}<T_{\VV^t}]=\mathbf{E}_{\ww}[A]$, so that
$
	\mathbf{E}_{\ww} [A]
	= 1 + \mathbf{E}_{\ww} [A] \, \mathbf{P}_{\zz}(T_{\ww}<T_{\VV^t}) 
	\,\mathbf{P}_\ww(X_1=\zz)
$
and
$
	\mathbf{E}_{\ww} [A]
	= 1/(1-\mathbf{P}_{\zz}(T_{\ww}<T_{\VV^t}) 
	\,\mathbf{P}_\ww(X_1=\zz)).
$
\end{proof}

\subsubsection{Proofs of Results for $d$-Regular Graphs (Flow Problem)}

We can now present the proofs of the results concerning the flow problem contained in Lemma \ref{lem:ring} and Lemma \ref{lem:regulargraph} on the characterization of the error committed by the Min-Sum algorithm as a function of time and as a function of the voltage solution $\nu^\star=L^+ b$.\\

\begin{proof}[Proof of Lemma \ref{lem:ring}, flow problem]
Let $\vec G=(V=\{0,\ldots,n-1\},\vec E,W)$ be a weighted directed cycle.
Consider the setting in Section \ref{sec:comp tree}. Fix $v\in V$ and assume that $\tilde e=(\tilde v=\rho(v),\tilde w=\rho(v+1))\in \vec E$ (the case $e=(\tilde w=\rho(v+1),\tilde v=\rho(v))$ yields to the same result with a minus sign in front). Fix $t\ge 2$, and let $\vec{\TT}$ be the computation tree of depth $t$ rooted at $\tilde e$, with root edge $\tilde \ee=(\tilde \vv,\tilde \ww)\in\vec\EE$. Label the tree vertices as $\VV=\{0,1,\ldots,2t+1\}$, as in Figure \ref{fig:ring}. Clearly, $\tilde\vv=t,\tilde\ww=t+1$, $\VV^{t-1}=\{1,2t\}$, and $\VV^t=\{0,2t+1\}$.

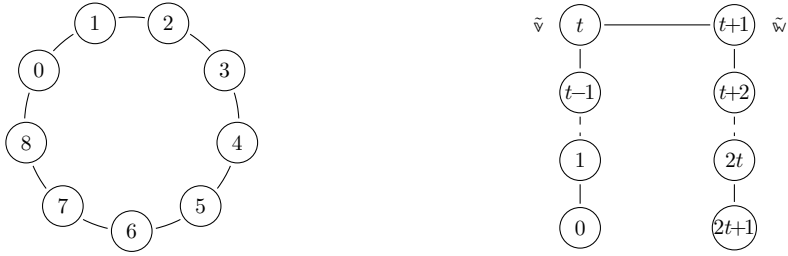
\begin{figure}[h!]
\begin{subfigure}{.45\textwidth}
\centering
\begin{tikzpicture}[scale=0.75,every node/.style={draw,circle,scale=0.75,solid,minimum size=0.72cm,inner sep=0pt}]
\def \n {9}
\def \radius {1.9cm}
\def \margin {13}

\foreach \s in {0,...,8}
{
  \node at ({360/\n * (\n-\s + 5 - 1)-10}:\radius) {$\s$};
  \draw[-, >=latex] ({360/\n * (\s - 1)+\margin-10}:\radius) 
    arc ({360/\n * (\s - 1)+\margin-10}:{360/\n * (\s)-\margin-10}:\radius);
}
\end{tikzpicture}
\end{subfigure}
\begin{subfigure}{.45\textwidth}
\centering
\begin{tikzpicture}
[scale=0.75,every node/.style={draw,circle,scale=0.75,solid,minimum size=0.72cm,inner sep=0pt},
,>=latex,shorten >=1.3pt,shorten <=1.3pt,
level distance=1.2cm, 
level 1/.style={sibling distance=2cm},
level 2/.style={sibling distance=1cm}]

\node (Root1) [label=left:$\tilde\vv$] {$t$}
child {
	node {$t\!\!-\!\!1$} edge from parent [-]
	child {
		node {$1$} edge from parent [-,dashed]
			child { node {$0$} edge from parent [-,solid]}
		}
	};
\node[right=1.5cm of Root1] [label=right:$\ \tilde\ww$] (Root2) {$t\!\!+\!\!1$}
	child {
		node {$t\!\!+\!\!2$} edge from parent [-]
			child { node {$2t$} edge from parent [-,dashed]
				child { node {$2t\!\!+\!\!1$} edge from parent [-,solid]}
			}
	};
\draw[-] (Root1) -- (Root2);

\end{tikzpicture}
\end{subfigure}
\caption{Cycle $G=(V,E)$ (left) and computation tree $\TT=(\VV,\EE)$ of depth $t$ with root edge $\{\tilde\vv= t,\tilde\ww= t+1\}\in \EE$ (right).}
\label{fig:ring}
\end{figure}

\noindent We will apply Lemma \ref{lem:computationdevice} to compute the quantities $\bar \LL^{-1}_{\tilde\vv\vv} - \bar \LL^{-1}_{\tilde\ww\vv}$, for each $\vv\in\VV^{t-1}$. We write the computations for $\vv=1$, as the case $\vv=2t$ follows immediately by symmetry. For convenience of notation, we write $\WW_{s}$ to indicate $\WW_{s,s+1}$. For each $s\in\VV$, define $f_s:=\mathbf{P}_{s}(T_{1}<T_{\VV^t})$, $p_s:=\mathbf{P}_s(X_1=s+1) = \WW_{s}/\dd_s$, $\dd_s:= \WW_{s-1}+\WW_{s}$, and $q_s:=\mathbf{P}_s(X_1=s-1)=1-p_s$. By a first step analysis of the random walk and the Markov property, we have
\begin{align}
	f_s = f_{s-1} q_s + f_{s+1} p_s,
	\quad s\in\{2,\ldots,2t\},
	\label{system gambler cycle}
\end{align}
with boundary conditions $f_1=1,f_{2t+1}=0$. The above system of equations corresponds to the classical Gambler Ruin's problem, with final states $1$ and $2t+1$, see Figure \ref{fig:gambler}.
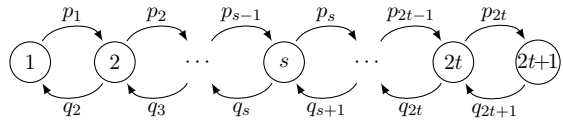
\begin{figure}[h!]
\centering
\begin{tikzpicture}
[scale=0.75,auto,
state/.style={draw,circle,scale=0.75,solid,minimum size=0.72cm,inner sep=0pt},
textt/.style={draw=none,scale=0.75,solid,inner sep=3pt},
>=latex,shorten >=1.3pt,shorten <=1.3pt,->,node distance=1.5cm]
\node[state]    		(A)		           	{$1$};
\node[state]    		(B)[right of=A]   	{$2$};
\node[state,draw=none]    					(C)[right of=B]   {$\cdots$};
\node[state]    		(D)[right of=C]   {$s$};
\node[state,draw=none]    					(E)[right of=D]   {$\cdots$};
\node[state]    		(F)[right of=E]   {$2t$};
\node[state]    	(G)[right of=F]   {$2t\!\!+\!\!1$};
\path
(A) edge[bend left=60,below]   	node[textt,above] {$p_1$}      	(B)
(B) edge[bend left=60,below]   	node[textt,below] {$q_2$}      	(A)
(B) edge[bend left=60,below]   	node[textt,above] {$p_2$}      	(C)
(C) edge[bend left=60,below]   	node[textt,below] {$q_3$}      	(B)
(C) edge[bend left=60,below]   	node[textt,above] {$p_{s-1}$}	(D)
(D) edge[bend left=60,below]   	node[textt,below] {$q_s$}      	(C)
(D) edge[bend left=60,below]   	node[textt,above] {$p_{s}$}	(E)
(E) edge[bend left=60,below]   	node[textt,below] {$q_{s+1}$}  	(D)
(E) edge[bend left=60,below]   	node[textt,above] {$p_{2t-1}$}	(F)
(F) edge[bend left=60,below]    	node[textt,below] {$q_{2t}$}  	(E)
(F) edge[bend left=60,below]    	node[textt,above] {$p_{2t}$}     	(G)
(G) edge[bend left=60,below]   	node[textt,below] {$q_{2t+1}$}	(F);
\end{tikzpicture}
\caption{Gambler's Ruin.}
\label{fig:gambler}
\end{figure}

\noindent To solve these equations, using that $1=p_s+q_s$, we first rewrite them as
$
	f_{s+1}-f_{s} = \frac{q_s}{p_s} (f_s-f_{s-1}),
$
for
$
	s\in\{2,\ldots,2t\}.
$
Using that $q_s/p_s=\WW_{s-1}/\WW_{s}$ we get
$
	f_{s+1}-f_{s} = \frac{\WW_1}{\WW_s} (f_2-f_{1}).
$
Summing these equations side by side, using that $f_1=1$, we get
$
	f_s = 1+(f_2-1) \WW_{1} (1/\WW_{1}+1/\WW_{2}+\cdots+1/\WW_{s-1}).
$
Using that $f_{2t+1}=0$ we can solve for $f_2$ and finally obtain
$
	f_s = \frac{1/\WW_{s}+1/\WW_{s+1}+\cdots+1/\WW_{2t}}
	{1/\WW_{1}+1/\WW_{2}+\cdots+1/\WW_{2t}}
$
for any $s\in\{1,\ldots,2t\}$.
From Lemma \ref{lem:computationdevice}, noticing that $\zz=2$ in this case, we have
$\bar \LL^{-1}_{\tilde\vv 1} - \bar \LL^{-1}_{\tilde\ww 1} = (f_t - f_{t+1})/((1-f_2p_1)\dd_1)$, which equals
$
	\bar \LL^{-1}_{\tilde\vv 1} - \bar \LL^{-1}_{\tilde\ww 1} 
	= \frac{1/\WW_{t}}{\WW_{0}(1/\WW_{0} +\cdots+1/\WW_{2t})}.
$
By symmetry, we have
$
	\bar \LL^{-1}_{\tilde\vv, 2t} - \bar \LL^{-1}_{\tilde\ww, 2t} 
	= \frac{-1/\WW_{t}}{\WW_{2t}(1/\WW_{0}+\cdots+1/\WW_{2t})}.
$
By Theorem \ref{thm:errorcharacterization}, as $\WW_{s}=W_{\sigma(s),\sigma(s+1)}$, we finally get
$
	x^\star_{\tilde e} - \hat x^t_{\tilde e}
	=
	\frac{\nu^\star_{\sigma(0)}	- \nu^\star_{\sigma(2t+1)}}
	{1/W_{\sigma(0)\sigma(1)} +\cdots+1/W_{\sigma(2t)\sigma(2t+1)}}.
$
The proof follows as the map $\sigma:\VV \rightarrow V$ reads $\sigma(s)=\rho(s-t+v)$.
\end{proof}

\begin{proof}[Proof of Lemma \ref{lem:regulargraph}, flow problem]
Let $\vec G=(V,\vec E,W)$ be a connected $d$-regular graph, with $d\ge 3$, where each edge has the same weight $\omega$.
Consider the setting in Section \ref{sec:comp tree}. Fix $\tilde e=(\tilde v,\tilde w)\in \vec E$, $t\ge 3$, and let $\vec{\TT}$ be the computation tree of depth $t$ rooted at $\tilde e$, with root edge $\tilde \ee=(\tilde \vv,\tilde \ww)\in\vec\EE$. Fix $\vv\in\VV^{t-1}$, and assume that $\vv$ is contained in the subtree with root vertex $\tilde \vv$. For what follows, it is convenient to label some of the vertices in $\VV$. Let $\vv=0, 1=\mathcal{N}(\vv)\setminus \VV^t, 2=\mathcal{N}(1)\setminus \VV^{t-1},\ldots,k+1=\mathcal{N}(k)\setminus \VV^{t-k}$, until we reach $\tilde\vv=t-1$. Label $\tilde\ww=t$. See Figure \ref{fig:regulargraph}. 
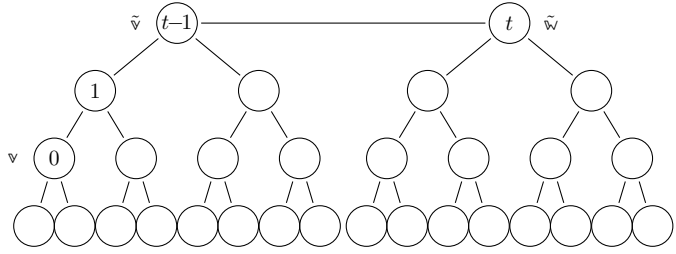
\begin{figure}[h!]

\centering

\begin{tikzpicture}
[scale=0.75,every node/.style={draw,circle,scale=0.75,solid,minimum size=0.72cm,inner sep=0pt},
,>=latex,shorten >=1.3pt,shorten <=1.3pt,
level distance=1.2cm, 
level 1/.style={sibling distance=2.9cm},
level 2/.style={sibling distance=1.45cm},
level 3/.style={sibling distance=0.72cm}]

\node[label=left:$\tilde\vv$] (Root1) {$t\!\!-\!\!1$}
child {
	node {$1$} edge from parent [-]
	child {
		node[label=left:$\vv$] {$0$} edge from parent [-]
			child { node{\phantom{0}} edge from parent [-]}
			child { node{\phantom{0}} edge from parent [-]}
		}
	child {
		node {\phantom{0}} edge from parent [-]
			child { node{\phantom{0}} edge from parent [-]}
			child { node{\phantom{0}} edge from parent [-]}
		}
	}
child {
	node {\phantom{0}} edge from parent [-]
	child {
		node {\phantom{0}} edge from parent [-]
			child { node{\phantom{0}} edge from parent [-]}
			child { node{\phantom{0}} edge from parent [-]}
		}
	child {
		node {\phantom{0}} edge from parent [-]
			child { node{\phantom{0}} edge from parent [-]}
			child { node{\phantom{0}} edge from parent [-]}
		}
	};
\node[right=3.87cm of Root1,label=right:$\tilde\ww$] (Root2) {$t$}
child {
	node {\phantom{0}} edge from parent [-]
	child {
		node {\phantom{0}} edge from parent [-]
			child { node{\phantom{0}} edge from parent [-]}
			child { node{\phantom{0}} edge from parent [-]}
		}
	child {
		node {\phantom{0}} edge from parent [-]
			child { node{\phantom{0}} edge from parent [-]}
			child { node{\phantom{0}} edge from parent [-]}
		}
	}
child {
	node {\phantom{0}} edge from parent [-]
	child {
		node {\phantom{0}} edge from parent [-]
			child { node{\phantom{0}} edge from parent [-]}
			child { node{\phantom{0}} edge from parent [-]}
		}
	child {
		node {\phantom{0}} edge from parent [-]
			child { node{\phantom{0}} edge from parent [-]}
			child { node{\phantom{0}} edge from parent [-]}
		}
	};
\draw[-] (Root1) -- (Root2);
\end{tikzpicture}

\caption{Computation tree $\TT=(\VV,\EE)$ of depth $t=3$ associated to any $3$-regular graph.}
\label{fig:regulargraph}
\end{figure}
We will apply Lemma \ref{lem:computationdevice} to compute the quantity $\bar \LL^{-1}_{\tilde\vv\vv} - \bar \LL^{-1}_{\tilde\ww\vv}$. To this end, we need to compute the hitting probabilities $f_\ww:=\mathbf{P}_{\ww}(T_{\vv}<T_{\VV^{t}})$ for $\ww\in\{1,t-1,t\}$. To do this, it is convenient to first reduce the network topology by invoking the equivalence between random walks and electrical circuits. See \cite{10.4169/j.ctt5hh804}. The electrical circuit associated to the random walk $X$ on $\TT$ is constructed by assigning to each edge $\ee=\{\vv,\ww\}\in\TT$ a conductor $\omega$, which corresponds to a resistor $1/\omega$. Then, the quantities $f_\ww$, $\ww\in\VV$, correspond to the voltages induced on $\VV\setminus\{\vv\cup \VV^t\}$ when a voltage of intensity $1$ is applied to $\vv$ ($f_\vv=1$), and a voltage of intensity $0$ is applied to the sites in $\VV^t$ ($f_\ww=0$, $\ww\in\VV^t$). By repeatedly applying the standard network reduction operations for conductors in parallel and in series, it is easy to verify that, as long as the quantities $f_1,f_2,\ldots,f_t$ are concerned, the original network is equivalent to the one represented in Figure \ref{fig:reducednetwork}.
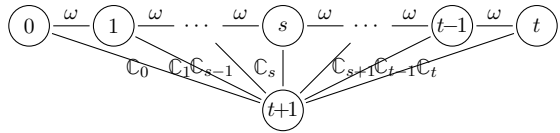
\begin{figure}[h!]
\centering
\begin{tikzpicture}
[scale=0.75,auto,
state/.style={draw,circle,scale=0.75,solid,minimum size=0.72cm,inner sep=0pt},
textt/.style={draw=none,scale=0.75,solid,inner sep=3pt},
>=latex,shorten >=1.3pt,shorten <=1.3pt,-,node distance=1.5cm]
\node[state]    			(A)                     {$0$};
\node[state]    			(B)[right of=A]   {$1$};
\node[state,draw=none]    	(C)[right of=B]   {$\cdots$};
\node[state]    			(D)[right of=C]   {$s$};
\node[state,draw=none]    	(E)[right of=D]   {$\cdots$};
\node[state]		    	(F)[right of=E]   {$t\!\!-\!\!1$};
\node[state]    			(G)[right of=F]   {$t$};
\node[state]    			(H)[below of=D]   {$t\!\!+\!\!1$};
\path
(A) edge   	node[textt,above] {$\omega$}      	(B)
(B) edge   	node[textt,above] {$\omega$}      	(C)
(C) edge   node[textt,above] {$\omega$}      	(D)
(D) edge   node[textt,above] {$\omega$}      	(E)
(E) edge   node[textt,above] {$\omega$}      	(F)
(F) edge   node[textt,above] {$\omega$}      	(G)

(A) edge   node[textt,left] {$\CC_{0}$}      	(H)
(B) edge   node[textt,left] {$\CC_{1}$}      	(H)
(C) edge   node[textt,left] {$\CC_{s-1}$}      	(H)
(D) edge   node[textt,left] {$\CC_{s}$}      	(H)
(E) edge   node[textt,right] {$\CC_{s+1}$}      	(H)
(F) edge   node[textt,right] {$\CC_{t-1}$}      	(H)
(G) edge  	node[textt,right] {$\CC_{t}$}      	(H);
\end{tikzpicture}
\caption{Reduced network.}
\label{fig:reducednetwork}
\end{figure}

In the reduced network, the conductance between the pairs of nodes $s$ and $s+1$, for $s\in\{0,\ldots,t-1\}$, remain $\omega$ as in the unreduced network, i.e., the computation tree.
The effect of all the nodes different from $0,\ldots,t$ in the original graph is summarized in the reduced network by node $t+1$. For any $s\in\{0,\ldots,t\}$, between node $s$ and node $t+1$ there is a new conductor $\CC_s$, which is obtained by reducing the subtree having root vertex $s$ in the original graph.
It is easy to verify that
\begin{align*}
	\CC_0
	= \omega (d-1),\qquad
	\CC_s 
	= \omega \frac{(d-2)^2}{d-1}(1+h_{s+1}), s\in\{1,\ldots,t-1\},\qquad
	\CC_t 
	= \omega (d-2)(1+h_{t}),
\end{align*}
with $h_s:=\frac{1}{(d-1)^{s}-1}$. 
Let $p_0:=\frac{\omega}{\omega+\CC_0},q_0:=1-p_0$, $p_s=q_s:=\frac{\omega}{2\omega+\CC_s}$ for $s\in\{1,\ldots,t-1\}$, $p_t:=\frac{\CC_t}{\omega+\CC_t},q_t:=1-p_t$.
By the Markov property we have
\begin{align}
	f_s = f_{s-1} q_s + f_{s+1} p_s,
	\quad s\in\{1,\ldots,t\},
	\label{eqn:sys}
\end{align}
with $f_0=1,f_{t+1}=0$. This system of equations corresponds to a modification of the classical Gambler Ruin's problem, where states $0$ and $t+1$ are the final states, and each intermediate state $s$ has a probability $r_s:=1-p_s-q_s$ to reach state $t+1$. See Figure \ref{fig:gambler killed}.
\begin{figure}[h!]
\centering
\begin{tikzpicture}
[scale=0.75,auto,
state/.style={draw,circle,scale=0.75,solid,minimum size=0.72cm,inner sep=0pt},
textt/.style={draw=none,scale=0.75,solid,inner sep=3pt},
>=latex,shorten >=1.3pt,shorten <=1.3pt,->,node distance=1.5cm]
\node[state]    			(A)                     {$0$};
\node[state]    			(B)[right of=A]   {$1$};
\node[state,draw=none]    	(C)[right of=B]   {$\cdots$};
\node[state]    			(D)[right of=C]   {$s$};
\node[state,draw=none]    	(E)[right of=D]   {$\cdots$};
\node[state]		    	(F)[right of=E]   {$t\!\!-\!\!1$};
\node[state]    			(G)[right of=F]   {$t$};
\node[state]    			(H)[right of=G]   {$t\!\!+\!\!1$};
\path
(A) edge[bend left=60,below]   	node[textt,above] {$p_0$}      	(B)
(B) edge[bend left=60,below]   	node[textt,below] {$q_1$}      	(A)
(B) edge[bend left=60,below]   	node[textt,above] {$p_1$}      	(C)
(C) edge[bend left=60,below]   	node[textt,below] {$q_2$}      	(B)
(C) edge[bend left=60,below]   	node[textt,above] {$p_{s-1}$}	(D)
(D) edge[bend left=60,below]   	node[textt,below] {$q_s$}      	(C)
(D) edge[bend left=60,below]   	node[textt,above] {$p_{s}$}	(E)
(E) edge[bend left=60,below]   	node[textt,below] {$q_{s+1}$}  	(D)
(E) edge[bend left=60,below]   	node[textt,above] {$p_{t-2}$}	(F)
(F) edge[bend left=60,below]   	node[textt,below] {$q_{t-1}$}  	(E)
(F) edge[bend left=60,below]   	node[textt,above] {$p_{t-1}$}	(G)
(G) edge[bend left=60,below]    node[textt,below] {$q_{t}$}  	(F)
(G) edge[bend left=60,below]    node[textt,above] {$p_{t}$}     	(H)

(A) edge[out=90,in=90,below]   	node[textt,above] {$q_0$}		(H)
(B) edge[out=90,in=90,below]   	node[textt,above] {$r_{1}$}	(H)
(C) edge[out=90,in=90,below]   	node[textt,above] {$r_{s-1}$}	(H)
(D) edge[out=90,in=90,below]   	node[textt,above] {$r_{s}$}	(H)
(E) edge[out=90,in=90,below]   	node[textt,above] {$r_{s+1}$}	(H)
(F) edge[out=90,in=90,below]   	node[textt,above] {$r_{t-1}$}	(H);
\end{tikzpicture}
\caption{Gambler's Ruin on the reduced network.}
\label{fig:gambler killed}
\end{figure}
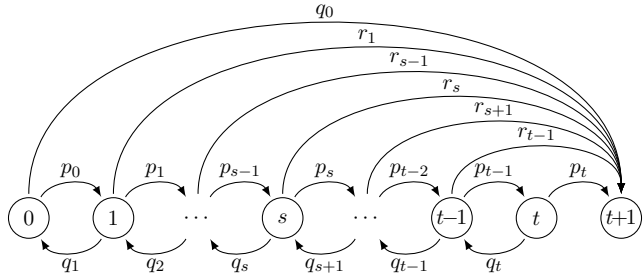

Without loss of generality, we set $r_{t}=0$.
We can easily write the solution of this system of equations in a compact form. Define $\alpha_1 = \alpha_2 := 1, \alpha_s := \alpha_{s-1} - p_{s-2}q_{s-1}\alpha_{s-2}$ for $s\ge 3$. Define $\beta_1 = \beta_2 := 1$, $\beta_s := \beta_{s-1} - p_{s-1}q_{s}\beta_{s-2}$ for $s\ge 3$. Then, as a function of $f_1$, we can write the solution of \eqref{eqn:sys} as
$
	f_s = \frac{1}{p_1\cdots p_{s-1}} (\alpha_sf_1 - \beta_{s-1}q_1),
$
for $s\in\{2,\ldots,t+1\}$.
In fact, it is easy to verify that this expression satisfies system \eqref{eqn:sys}. Using $f_{t+1}=0$ we find $f_1=q_1\beta_t /\alpha_{t+1}$ and 
$
	f_s = \frac{p_0q_1}{p_0\cdots p_{s-1}} 
	(\frac{\alpha_s \beta_t}{\alpha_{t+1}} - \beta_{s-1})
$
for $s\in\{1,\ldots,t+1\}$, with $\beta_0:=0$. Lemma \ref{lem:computationdevice} yields the following, using that $\dd_\vv=\omega d$,
\begin{align*}
	\bar \LL^{-1}_{\tilde\vv \vv} &= \frac{f_{t-1}}{(1-f_1p_0)\dd_\vv}
	= \frac{1}{\omega d} \, \frac{q_1}
	{p_1\cdots p_{t-2} (\alpha_{t+1}-p_0q_1\beta_t)} \, 
	(\alpha_{t-1}\beta_t-\alpha_{t+1}\beta_{t-2}),\\
	\bar \LL^{-1}_{\tilde\ww \vv} &= \frac{f_{t}}{(1-f_1p_0)\dd_\vv}
	= \frac{1}{\omega d} \, \frac{q_1}
	{p_1\cdots p_{t-2} (\alpha_{t+1}-p_0q_1\beta_t)} \, 
	\frac{1}{p_{t-1}} \, (\alpha_{t}\beta_t-\alpha_{t+1}\beta_{t-1}).
\end{align*}
Using the recursive formulas for $\alpha_{t+1}$ and $\beta_t$, and the fact that $p_t+q_t=1$, we get
$
	\bar \LL^{-1}_{\tilde\vv \vv} - \bar \LL^{-1}_{\tilde\ww \vv} 
	= \frac{1}{\omega d} \, \frac{q_1}
	{p_1\cdots p_{t-2} (\alpha_{t+1}-p_0q_1\beta_t)} \, 
	p_t \, (\alpha_{t-1}\beta_{t-1}-\alpha_{t}\beta_{t-2}),
$
which, as
$
	\alpha_{s}\beta_s-\alpha_{s+1}\beta_{s-1}
	= p_{s-1}q_s (\alpha_{s-1}\beta_{s-1}-\alpha_{s}\beta_{s-2}),
$
becomes
\begin{align}
	\bar \LL^{-1}_{\tilde\vv \vv} - \bar \LL^{-1}_{\tilde\ww \vv}
	= \frac{p_t}{q_t} \, \bar \LL^{-1}_{\tilde\ww \vv}
	= \frac{1}{\omega d (d-1)^{t-1}} \frac{1}{\xi_{t-1}},
	\label{regular case general formula}
\end{align}
where $\xi_{t-1}:=\frac{1}{(d-1)^{t-1}}\frac{\alpha_{t+1}-p_0q_1\beta_{t}}{q_1\cdots q_{t-1}p_{t}}$ 
is defined by
$
	\xi_s = 
	\frac{p_{s}}{q_{s}p_{s+1}(d-1)} \xi_{s-1}
	- \frac{p_{s-1}p_{s}q_{s+1}}{q_{s-1}q_{s}p_{s+1}(d-1)^2} \xi_{s-2}
$
for $s\in \{2,\ldots,t-1\}$, with $\xi_0 := \frac{p_0q_1}{p_1}$ and $\xi_1 = \frac{1}{d-1}\frac{1-p_0q_1-p_1q_2}{q_1p_2}$. Plugging in the definitions of the transition probabilities, recalling the quantities defined in Proposition \ref{prop:constant d t}, it is immediate to verify that $\delta_s\equiv\xi_s$ for $s\in\{0,\ldots,t-2\}$ and $c_{d,t}\equiv\xi_{t-1}$.
By symmetry, it is clear that
$
	\bar \LL^{-1}_{\tilde\vv \vv} - \bar \LL^{-1}_{\tilde\ww \vv}
	= \frac{1}{\omega  d(d-1)^{t-1}} \frac{1}{c_{d,t}}
$
holds for \emph{any} $\vv\in\VV^{t-1}$ that is contained in the subtree of $\TT$ with root vertex $\tilde \vv$.
On the other hand, again by symmetry, for any $\vv\in\VV^{t-1}$ contained in the subtree of $\TT$ with root vertex $\tilde \ww$ we have
$
	\bar \LL^{-1}_{\tilde\vv \vv} - \bar \LL^{-1}_{\tilde\ww \vv}
	= - \frac{1}{\omega  d(d-1)^{t-1}} \frac{1}{c_{d,t}}.
$
Theorem \ref{thm:errorcharacterization} yields
\begin{align*}
	x^\star_{\tilde e} - \hat x^t_{\tilde e}
	= \omega^2\!
	\sum_{\vv\in\VV^{t-1}} 
	(\bar\LL^{-1}_{\tilde\vv\vv} - \bar\LL^{-1}_{\tilde\ww\vv})
	\sum_{\ww\in\VV^t} 
	\nu^\star_{\sigma(\ww)}
	= 
	\frac{\omega}{d(d-1)^{t-1}} \frac{1}{c_{d,t}}
	\sum_{v\in V}
	 \nu^\star_{v} (\#^{t}_{\tilde v \rightarrow v,\tilde w} - \#^{t}_{\tilde w \rightarrow v,\tilde v}),
\end{align*}
where $\#^{t}_{\tilde v \rightarrow v,\tilde w}$ is the number of non-backtracking paths in the original undirected graph $G=(V,E)$ that connect node $\tilde v$ to node $v$ in $t$ time steps, where the first step is different from $\tilde w$. As the total number of non-backtracking paths after $t$ steps is $d(d-1)^{t-1}$, we have $\vec{\mathbf{P}}_{\tilde v}(Y_1\neq \tilde w,Y_t=v)=\#^{t}_{\tilde v \rightarrow v,\tilde w}/(d(d-1)^{t-1})$. The proof follows as $\vec{\mathbf{P}}_{\tilde v}(Y_1\neq \tilde w,Y_t=v) = \vec{\mathbf{P}}_{\tilde v}(Y_t=v | Y_1\neq \tilde w) \vec{\mathbf{P}}_{\tilde v}(Y_1\neq \tilde w)$, and $\vec{\mathbf{P}}_{\tilde v}(Y_1\neq \tilde w) = (d-1)/d$.
\end{proof}

\subsection{Voltage Problem, Algorithm \ref{alg:Min-Sum, quadratic messages, no leaves}}

We now give an overview of the results in the voltage case. The main framework for these results is the same as the one used in the previous section to investigate the behavior of the Min-Sum algorithm applied to the flow problem. For this reason, as outlined at the beginning of Section \ref{sec:proofs}, here we only discuss in details the parts of the proofs that present major differences compared to the flow problem. 

Henceforth we consider Algorithm \ref{alg:Min-Sum voltages, quadratic messages} when the initial messages are parametrized by $\{W^{0}_{e\rightarrow v}=W_{vw}, e=\{v,w\}\}, \{h^{0}_{e\rightarrow v}=p_{e\rightarrow v}\}$, for a certain set of real numbers $p=\{p_{e\rightarrow v}\}$.
Also in this case, as far as the analysis is concerned, it is convenient to consider the form of Algorithm \ref{alg:Min-Sum voltages}. For a given choice of perturbation parameters $p=\{p_{e\rightarrow v}\}$, define the initial messages as
$
	\mu^{0}_{e\rightarrow v}(\,\cdot\,,p)
	: z\in\R \longrightarrow
	\mu^{0}_{e\rightarrow v}(z,p)
	= \frac{1}{2} 
	W_{vw}z^2 + p_{e\rightarrow v}z,
$
where $e=\{v,w\}$, and denote by $\{\mu^{s}_{e\rightarrow v}(\,\cdot\,,p)\}$, $s\ge 1$, the corresponding sequence of messages generated by the Min-Sum algorithm. Denote the estimates at time $t$ as $\hat \nu_v^t = \arg\min_{z\in\R} \mu^{t}_{v}(z,p)$, where $\mu^{t}_{v}(\,\cdot\,,p)$ is the corresponding belief function.

\subsubsection{Computation Tree (Voltage Problem)}\label{sec:comp tree voltage}
The computation tree for the voltage problem differs from the one in the flow problem due to the fact that the Min-Sum algorithms for the voltage problem propagates messages that are supported on vertices, not on edges. For this reason, the computation tree that we will now define is rooted at a vertex, not at an edge as the one defined in Section \ref{sec:comp tree}.

Given a vertex $\tilde v \in V$, the computation tree rooted at $\tilde v$ of depth $t$ supports the optimization problem that is obtained by unfolding the computations involved in the Min-Sum estimate $\hat \nu^{t}(p)_{\tilde v}$.
Formally, the computation tree is an undirected tree $\TT=(\VV,\EE)$ where each vertex in $\VV$ is mapped to a vertex in $V$ through a map $\sigma:\VV \rightarrow V$ that preserves the edge structure, namely, if $\ee=\{\vv,\ww\}\in \EE$ then $\sigma(\ee):=\{\sigma(\vv),\sigma(\ww)\}\in E$. The tree $\TT$ is defined iteratively. Initially, at level $-1$, the tree is made by a single root vertex $\tilde \vv\in\VV$ corresponding to $\tilde v$, i.e., $\sigma(\tilde \vv)=\tilde v$. Consecutively, for any $z\in\mathcal{N}(\tilde v)$, a vertex $\zz$ with $\sigma(\zz)=z$ and an edge $\{\tilde\vv,\zz\}$ are added to the level $0$ of the tree. For the remaining $t-1$ levels, the following procedure is repeated. The leaves in the tree are examined. Given a leaf $\vv$ with $\sigma(\vv)=v$ that is connected to a vertex $\ww$ with $\sigma(\ww)=w$, for any $z\in\mathcal{N}(v)\setminus w$, a vertex $\zz$ with $\sigma(\zz)=z$ and an edge $\{\vv,\zz\}$ are added to the next level of the tree. Figure \ref{fig:computationtree voltage} gives an example.
We denote the set of vertices and edges in the $k$-th level of the tree respectively by $\VV^{k}\subset\VV$ and $\EE^{k}\subset\EE$.
In what follows we also extend the usual neighborhood notation to vertices and edges in the graph $\TT$, namely, $\partial\vv$ is the set of edges in $\TT$ that are connected to node $\vv$, and $\partial\ee$ is the set of vertices in $\TT$ that are connected to edge $\ee$.

\begin{figure}[h!]
\begin{subfigure}{.45\textwidth}
\centering

\begin{tikzpicture}
[scale=0.75,every node/.style={draw,circle,scale=0.75,solid,minimum size=0.72cm,inner sep=0pt},
-,>=latex,shorten >=1pt,shorten <=1pt]

\def \x {2}
\def \y {-2}

\node (1) at (0,0) {1};
\node (2) at (\x,0) {2};
\node (3) at (\x,\y) {3};
\node (4) at (0,\y) {4};

\draw (1) -- (2);
\draw (2) -- (3);
\draw (4) -- (3);
\draw (4) -- (1);
\draw (3) -- (1);

\end{tikzpicture}
\end{subfigure}
\begin{subfigure}{.45\textwidth}
\centering

\begin{tikzpicture}
[scale=0.75,every node/.style={draw,circle,scale=0.75,solid,minimum size=0.72cm,inner sep=0pt},
,>=latex,shorten >=1.3pt,shorten <=1.3pt,
level distance=1.2cm, 
level 1/.style={sibling distance=2cm},
level 2/.style={sibling distance=1cm}]

\node (Root1) {$1$}
child {
	node {$2$} edge from parent [-]
	child {
		node {$3$} edge from parent [-]
			child { node{$1$} edge from parent [-]}
			child { node{$4$} edge from parent [-]}
		}
	}
child {
	node {$3$} edge from parent [-]
	child { 
		node {$2$}  edge from parent [-]
		child { node{$1$} edge from parent [-] }
	}
	child {
		node {$4$}  edge from parent [-]
		child { node{$1$} edge from parent [-] }
	}
	}
child {
	node {$4$} edge from parent [-]
	child {
		node {$3$} edge from parent [-]
			child { node{$1$} edge from parent [-]}
			child { node{$2$} edge from parent [-]}
		}
	};
\end{tikzpicture}
\end{subfigure}
\caption{Graph $G=(V=\{1, 2, 3, 4\}, E)$ (left) and computation tree $\TT=(\VV,\EE)$ of depth $t=2$ with root edge corresponding to $\text{1}\in V$ (right). Vertices in $\TT$ are labeled by the corresponding vertices in $G$.}
\label{fig:computationtree voltage}
\end{figure}
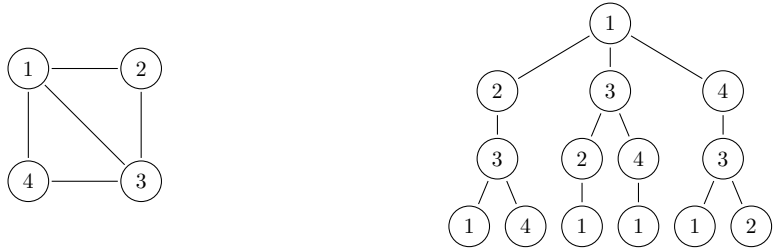

The $t$-th (last) level of the tree is of particular relevance as it supports the initial conditions of the Min-Sum algorithm, as we are now about to describe. Let $\WW\in\R^{\VV\times \VV}$ be the diagonal matrix defined by $\WW_{\vv\ww}:=W_{\sigma(\vv)\sigma(\ww)}$ if $\{\vv,\ww\}\in\EE$ and $\WW_{\vv\ww}:=0$ otherwise. Let $\bar \VV := \VV\setminus \VV^{t}$.
Define the weighted degree of a vertex $\vv$ by $\dd_\vv:=\sum_{\ww\in \VV} \WW_{\vv\ww}$, and let $\DD\in\R^{\VV\times \VV}$ be the diagonal matrix defined by $\DD_{\vv\vv}:=\dd_\vv$. The Laplacian $\LL$ of the graph $\TT$ is the matrix defined as $\LL:=\DD-\WW$. Let $\bar\LL\in\R^{\bar\VV\times\bar\VV}$ denote the submatrix of $\LL$ that corresponds to the rows and columns associated to $\bar\VV$.
For a given choice of the perturbation parameters $p=\{p_{e\rightarrow v}\}$, define $\bar\pp(p)\in\R^{\bar\VV}$ by
\begin{align}
	\bar\pp(p)_\vv :=
	\begin{cases}
		\sum_{\ff \in\partial \vv\setminus \ee} p_{\sigma(\ff)\rightarrow \sigma(\vv)} &\text{if } \vv\in\VV^{t-1}, \ee=\partial \vv \setminus \EE^{t},\\
		0 &\text{otherwise},
	\end{cases}
	\label{def:p bar voltages}
\end{align}
and let $\bar\bb\in\R^{\bar\VV}$ defined by
$
	\bar\bb_\vv := b_{\sigma(\vv)}
$
for each $\vv\in\bar\VV$.

Consider the following problem supported on the computation tree $\TT$, over $\bar\bbnu\in\R^{\bar \VV}$:
\begin{align}
\begin{aligned}
	\text{minimize }\quad    & \frac{1}{2} \bar\bbnu^T \bar\LL \bar\bbnu + (\bar\pp(p) - \bar\bb)^T \bar\bbnu.
\end{aligned}
\label{computation tree voltage}
\end{align}

While the Laplacian $\LL$ is positive semi-definite, the restricted Laplacian $\bar\LL$ is positive definite \citep{spielman10}. Hence, for any choice of the perturbation $p$ the objective function in problem \eqref{computation tree voltage} is strictly convex and it admits a unique solution $\bar\bbnu^\star(p)$. The relationship between the Min-Sum algorithm and the tree is made explicit by the following lemma, which can be established by inductively examining the operations of Algorithm \ref{alg:Min-Sum voltages}.

\begin{lemma}\label{lem:computation tree voltage}
Given the setting of this section, $\hat \nu^t(p)_{\tilde v}=\bar\bbnu^\star(p)_{\tilde \vv}$.
\end{lemma}

\subsubsection{Fix Point (Voltage Problem)}
The choice $p=0$ for the linear perturbation parameters yields the original (unperturbed) algorithm. The next theorem below shows that for a particular choice of $p$ the Min-Sum algorithm yields the optimal solution at any time $t\ge 1$. The optimal choice of $p$ is a function of the optimal solution $\nu^\star$ for problem \eqref{dual}. Recall that, by definition, $\nu^\star$ is the unique vector that satisfy $L\nu^\star = b$ and $1^T\nu^\star =0$. Hence, component-wise, for each $v\in V$,
\begin{align}
	d_v\nu^\star_v - \sum_{w\in\partial v} W_{vw} \nu^\star_w - b_v = 0,
	\label{optimality voltage}
\end{align}
where recall that $d_v:=\sum_{w\in\partial v} W_{vw}$.
Given $\nu^\star$, we are now ready to define the set of parameters $p^\star=\{p^\star_{e\rightarrow v}\}$ so that the Min-Sum algorithm yields the optimal solution at any time step.

\begin{theorem}\label{thm: opt choice p voltage}
For each $e\in E$, $v\in\partial e$, let $w=\partial e\setminus v$ and define
$
	p^\star_{e\rightarrow v}
	:= 
	- W_{vw} \nu^\star_w.
$
Then, for any $\tilde v\in V$, $t\ge 1$, we have $\hat \nu^t(p^\star)_{\tilde v} = \nu^\star_{\tilde v}$.
\end{theorem}

\begin{proof}
Fix $\tilde v\in V$, $t\ge 1$. Consider the setting in Section \ref{sec:comp tree voltage}, and let $\TT$ be the computation tree of depth $t$ rooted at $\tilde v$.
From the first order optimality conditions for problem \eqref{computation tree voltage} corresponding to the choice $p=p^\star$, we know that $\bar\bbnu^\star \equiv \bar\bbnu^\star(p^\star)$ is the unique vector so that
\begin{align*}
	\begin{cases}
	\dd_\vv\bar\bbnu^\star_\vv - \sum_{\ww\in\partial \vv} \WW_{\vv\ww} \bar\bbnu^\star_\ww - \bar \bb_\vv = 0,
	\vv\in\bar\VV \setminus \VV^{t-1},\\
	\dd_\vv\bar\bbnu^\star_\vv - \WW_{\vv\ww} \bar\bbnu^\star_\ww + \bar \pp(p^\star)_\vv- \bar \bb_\vv = 0,
	\vv\in\VV^{t-1}, \{\vv,\ww\}=\partial\vv \setminus \EE^{t}.
	\end{cases}
\end{align*}
Using \eqref{optimality voltage} it is easy to check that the choice $\bar\bbnu^\star_{\vv}=\nu^\star_{\sigma(\vv)}$ satisfies the above system of equations. The proof is concluded by Lemma \ref{lem:computation tree voltage}.
\end{proof}

\subsubsection{Error Characterization (Voltage Problem)}
The next lemma characterizes the sensitivity of the $\tilde \vv$-th component (root vertex of $\TT$) of the optimal solution of the optimization problem \eqref{computation tree voltage} with respect to perturbations of the parameters $p_{e\rightarrow v}$'s, which are supported on $\VV^{t-1}$.

\begin{lemma}\label{lem:corr tree voltage}
Consider the setting in Section \ref{sec:comp tree voltage}. Fix $\tilde v\in V$, $t\ge 1$, and let $\TT$ be the computation tree of depth $t$ rooted at $\tilde v$, with root edge $\tilde \vv$. For any $p=\{p_{e\rightarrow v}\}$, $e\in E$ and $v\in \partial e$, we have
$$
	\frac{\partial}{\partial p_{e\rightarrow v}} \bar\bbnu^\star(p)_{\tilde \vv} 
	= - \sum_{\vv\in\VV^{t-1}} \bar\LL^{-1}_{\tilde \vv\vv} (\yy_{e\rightarrow v})_\vv,
$$
where $\yy_{e\rightarrow v}\in\R^{\VV^{t-1}}$ is defined, for each $\vv\in\VV^{t-1}$, as
$
	(\yy_{e\rightarrow v})_\vv := 
	\mathbf{1}_{\sigma(\vv)=v}\mathbf{1}_{\sigma(\partial \vv\cap \vec \EE^{t})=e}.
$
\end{lemma}

\begin{proof}
From the first order optimality conditions, we know that the unique solution of problem \eqref{computation tree voltage} reads
$
	\bar\bbnu^\star(p) 
	= \bar\LL^{-1} (\bar\bb - \bar\pp(p) )
$
so that
$
	\frac{\partial}{\partial p_{e\rightarrow v}} \bar\bbnu^\star(p) 
	= - \bar\LL^{-1} \frac{\partial}{\partial p_{e\rightarrow v}} \bar\pp(p).
$
Using \eqref{def:p bar voltages} it is easy to check that $(\yy_{e\rightarrow v})_\vv=(\frac{\partial}{\partial p_{e\rightarrow v}} \bar\pp(p))_\vv$.
\end{proof}

As a consequence of Lemma \ref{lem:corr tree voltage}, we immediately have the following characterization of the error committed by Algorithm \ref{alg:Min-Sum voltages, quadratic messages} with $\{W^{0}_{e\rightarrow v}=W_{vw}, e=\{v,w\}\}, \{h^{0}_{e\rightarrow v}=0\}$.

\begin{theorem}\label{thm:errorcharacterization voltage}
Consider the setting in Section \ref{sec:comp tree voltage}. Fix $\tilde v\in V$, $t\ge 1$, and let $\TT$ be the computation tree of depth $t$ rooted at $\tilde v$, with root vertex $\tilde\vv\in\VV$. We have
\begin{align*}
	\nu^\star_{\tilde v} - \hat \nu^t_{\tilde v}
	&=
	\sum_{\vv\in\VV^{t-1}} 
	\bar\LL^{-1}_{\tilde\vv\vv}
	\sum_{\ww\in\VV^{t}} 
	\WW_{\vv\ww} \nu^\star_{\sigma(\ww)}.
\end{align*}
\end{theorem}

\begin{proof}
From Theorem \ref{thm: opt choice p voltage} and Lemma \ref{lem:corr tree voltage},
$
	\nu^\star_{\tilde v} - \hat \nu^t_{\tilde v}
	= \hat \nu^t(p^\star)_{\tilde v} - \hat \nu^t(0)_{\tilde v}
	= \int_0^1 d\theta\, \frac{d}{d \theta} \bar\bbnu^\star(\theta p^\star)_{\tilde\vv}
	= \int_0^1 d\theta \sum_{e\in E, v\in \partial e}
	\frac{\partial}{\partial p_{e\rightarrow v}} \bar\bbnu^\star(\theta p^\star)_{\tilde\vv}
	\, p^\star_{e\rightarrow v}
	=
	- \sum_{\vv\in\VV^{t-1}} 
	\bar\LL^{-1}_{\tilde\vv\vv}
	\sum_{e\in E, v\in \partial e} (\yy_{e\rightarrow v})_\vv
	\, p^\star_{e\rightarrow v},
$
by the fundamental theorem of calculus and the chain rule of differentiation.
From Theorem \ref{thm: opt choice p voltage} we get
$
	p^\star_{e\rightarrow v}
	:= 
	- W_{vw} \nu^\star_w
$
for $e=\{v,w\}$,
and
$
	\sum_{e\in E, v\in \partial e} 
	(\yy_{e\rightarrow v})_\vv
	\, p^\star_{e\rightarrow v}
	= - \sum_{\ww\in\VV^{t}} 
	\WW_{\vv\ww} \nu^\star_{\sigma(\ww)}.
$
\end{proof}

Lemma \ref{lem:computationdevice} for the flow case holds also in the present case, upon referring the notation $\mathbf{P}_{\vv}$, $\mathbf{E}_{\vv}$, etc. to the Markov chain defined on the computation tree for the voltage case.

\subsubsection{Proofs of Results for $d$-Regular Graphs (Voltage Problem)}
We can now present the outline of the proofs for the results concerting the voltage problem in Lemma \ref{lem:ring} and Lemma \ref{lem:regulargraph} in Section \ref{sec:results}. Here we assume that the reader is already familiar with the proofs for the flow case.\\

\begin{proof}[Proof of Lemma \ref{lem:ring}, voltage problem]
Let $G=(V=\{0,\ldots,n-1\},E,W)$ be a weighted cycle.
Consider the setting in Section \ref{sec:comp tree voltage}. Fix $v\in V$, $t\ge 2$, and let $\TT$ be the computation tree of depth $t$ rooted at $\tilde v$, with root vertex $\tilde\vv$, i.e., $\sigma(\tilde\vv)=\tilde v$. Label the vertices of the computation tree as $\VV=\{0,1,\ldots,2t+2\}$, as in Figure \ref{fig:ring voltage}. Clearly, $\tilde\vv=t+1$, $\VV^{t-1}=\{1,2t+1\}$, and $\VV^t=\{0,2t+2\}$.

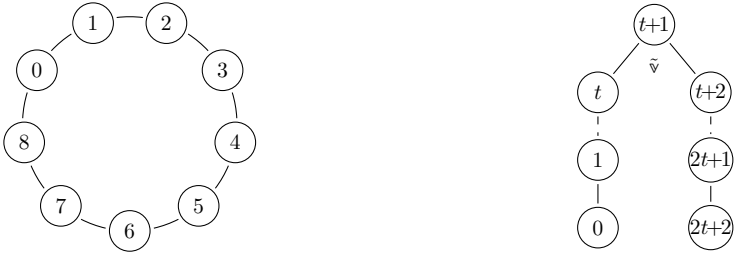
\begin{figure}[h!]
\begin{subfigure}{.45\textwidth}
\centering
\begin{tikzpicture}[scale=0.75,every node/.style={draw,circle,scale=0.75,solid,minimum size=0.72cm,inner sep=0pt}]
\def \n {9}
\def \radius {1.9cm}
\def \margin {13}

\foreach \s in {0,...,8}
{
  \node at ({360/\n * (\n-\s + 5 - 1)-10}:\radius) {$\s$};
  \draw[-, >=latex] ({360/\n * (\s - 1)+\margin-10}:\radius) 
    arc ({360/\n * (\s - 1)+\margin-10}:{360/\n * (\s)-\margin-10}:\radius);
}
\end{tikzpicture}
\end{subfigure}
\begin{subfigure}{.45\textwidth}
\centering
\begin{tikzpicture}
[scale=0.75,every node/.style={draw,circle,scale=0.75,solid,minimum size=0.72cm,inner sep=0pt},
,>=latex,shorten >=1.3pt,shorten <=1.3pt,
level distance=1.2cm, 
level 1/.style={sibling distance=2cm},
level 2/.style={sibling distance=1cm}]

\node (Root1) [label=below:$\tilde\vv$] {$t\!\!+\!\!1$}
child {
	node {$t$} edge from parent [-]
	child {
		node {$1$} edge from parent [-,dashed]
			child { node {$0$} edge from parent [-,solid]}
		}
	}
child {
	node {$t\!\!+\!\!2$} edge from parent [-]
	child {
		node {$2t\!\!+\!\!1$} edge from parent [-,dashed]
			child { node {$2t\!\!+\!\!2$} edge from parent [-,solid]}
		}
	};

\end{tikzpicture}
\end{subfigure}
\caption{Ring $G=(V,E)$ (left) and computation tree $\TT=(\VV,\EE)$ of depth $t$ with root vertex $\tilde\vv= t+1\in\VV$ (right).}
\label{fig:ring voltage}
\end{figure}

\noindent Given this set up for the computation tree, the remaining of the proof follows exactly the steps of the proof in the flow case. Using the analogue of Lemma \ref{lem:computationdevice} for the voltage case, we can compute the quantities $\bar \LL^{-1}_{\tilde\vv\vv}$, for each $\vv\in\VV^{t-1}$, by solving a system of equations like \eqref{system gambler cycle}. If we write $\WW_{s}$ to indicate $\WW_{s,s+1}$,
we find that
$
	\bar \LL^{-1}_{\tilde\vv 1}
	= \frac{1/\WW_{t+1} + \cdots + 1/\WW_{2t+1}}
	{\WW_{0}(1/\WW_{0} +\cdots+1/\WW_{2t+1})},
$
and, by symmetry,
$
	\bar \LL^{-1}_{\tilde\vv, 2t+1}
	= \frac{1/\WW_{0} + \cdots + 1/\WW_{t}}
	{\WW_{2t+1}(1/\WW_{0}+\cdots+1/\WW_{2t+1})}.
$
By Theorem \ref{thm:errorcharacterization voltage} we get
$
	\nu^\star_{\tilde v} - \hat \nu^t_{\tilde v}
	=
	\sum_{\vv\in\VV^{t-1}} 
	\bar\LL^{-1}_{\tilde\vv\vv}
	\sum_{\ww\in\VV^{t}} 
	\WW_{\vv\ww} \nu^\star_{\sigma(\ww)}=
	\frac{1/\WW_{t+1} + \cdots + 1/\WW_{2t+1}}
	{1/\WW_{0} +\cdots+1/\WW_{2t+1}}
	\nu^\star_{\sigma(0)}
	+
	\frac{1/\WW_{0} + \cdots + 1/\WW_{t}}
	{1/\WW_{0}+\cdots+1/\WW_{2t+1}}
	\nu^\star_{\sigma(2t+2)},
$
and the proof follows using $\WW_{s}=W_{\sigma(s),\sigma(s+1)}$ and that $\sigma:\VV \rightarrow V$ reads $\sigma(v)=\rho(s-t-1+\tilde v)$.
\end{proof}

\begin{proof}[Proof of Lemma \ref{lem:regulargraph}, voltage problem]
Let $G=(V,E,W)$ be a connected $d$-regular graph, with $d\ge 3$, where each edge has the same weight $\omega$. Consider the setting in Section \ref{sec:comp tree voltage}. Fix $\tilde w\in V$, $t\ge 3$, and let $\TT$ be the computation tree of depth $t$ rooted at $\tilde w$, with root vertex $\tilde \ww\in \VV$. Fix $\vv \in \VV^{t-1}$. For what follows, it is convenient to label some of the vertices in $\VV$. Let $\vv=0, 1=\mathcal{N}(\vv)\setminus \VV^t, 2=\mathcal{N}(1)\setminus \VV^{t-1},\ldots,k+1=\mathcal{N}(k)\setminus \VV^{t-k}$, until we reach $\tilde\ww=t$. See Figure \ref{fig:regulargraph voltage}. 
\begin{figure}[h!]

\centering

\begin{tikzpicture}
[scale=0.75,every node/.style={draw,circle,scale=0.75,solid,minimum size=0.72cm,inner sep=0pt},
,>=latex,shorten >=1.3pt,shorten <=1.3pt,
level distance=1.2cm, 
level 1/.style={sibling distance=2.9cm},
level 2/.style={sibling distance=1.45cm},
level 3/.style={sibling distance=0.72cm}]

\node[label=left:$\tilde\ww$] (Root1) {$t$}
child {
	node {$1$} edge from parent [-]
	child {
		node[label=left:$\vv$] {$0$} edge from parent [-]
			child { node{\phantom{0}} edge from parent [-]}
			child { node{\phantom{0}} edge from parent [-]}
		}
	child {
		node {\phantom{0}} edge from parent [-]
			child { node{\phantom{0}} edge from parent [-]}
			child { node{\phantom{0}} edge from parent [-]}
		}
	}
child {
	node {\phantom{0}} edge from parent [-]
	child {
		node {\phantom{0}} edge from parent [-]
			child { node{\phantom{0}} edge from parent [-]}
			child { node{\phantom{0}} edge from parent [-]}
		}
	child {
		node {\phantom{0}} edge from parent [-]
			child { node{\phantom{0}} edge from parent [-]}
			child { node{\phantom{0}} edge from parent [-]}
		}
	}
child {
	node {\phantom{0}} edge from parent [-]
	child {
		node {\phantom{0}} edge from parent [-]
			child { node{\phantom{0}} edge from parent [-]}
			child { node{\phantom{0}} edge from parent [-]}
		}
	child {
		node {\phantom{0}} edge from parent [-]
			child { node{\phantom{0}} edge from parent [-]}
			child { node{\phantom{0}} edge from parent [-]}
		}
	};
\end{tikzpicture}

\caption{Computation tree $\TT=(\VV,\EE)$ of depth $t=2$ associated to any given $d$-regular graph, with $d=3$.}
\label{fig:regulargraph voltage}
\end{figure}
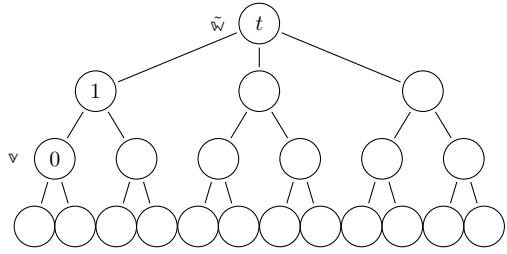
Given this set up for the computation tree, the remaining of the proof follows exactly the steps of the proof in the flow case.
Using the analogue of Lemma \ref{lem:computationdevice} for the voltage case, we can compute $\bar \LL^{-1}_{\tilde\ww\vv}$ by solving a system of equations of the form \eqref{eqn:sys} supported on the a reduced network that is exactly like the one in Figure \ref{fig:reducednetwork}. The only difference is that now $\CC_t = \omega (d-2)(1+h_{t+1})$, while all other effective conductances $\CC_s$, for $s\in\{0,\ldots,t-1\}$, read exactly as for the flow problem.
Rearranging \eqref{regular case general formula} we get
$
	\bar \LL^{-1}_{\tilde\ww \vv}
	= \frac{1}{\omega d (d-1)^{t-1}} \, \frac{q_t}{p_t} \, \frac{1}{\xi_{t-1}},
$
where $\xi_{t-1}:=\frac{1}{(d-1)^{t-1}}\frac{\alpha_{t+1}-p_0q_1\beta_{t}}{q_1\cdots q_{t-1}p_{t}}$ 
is defined as in the flow case by
$
	\xi_s = 
	\frac{p_{s}}{q_{s}p_{s+1}(d-1)} \xi_{s-1}
	- \frac{p_{s-1}p_{s}q_{s+1}}{q_{s-1}q_{s}p_{s+1}(d-1)^2} \xi_{s-2}
$
for $s\in \{2,\ldots,t-1\}$, with $\xi_0 := \frac{p_0q_1}{p_1}$ and $\xi_1 = \frac{1}{d-1}\frac{1-p_0q_1-p_1q_2}{q_1p_2}$. Plugging in the definitions of the transition probabilities, noticing that $p_t$ and $q_t$ differ from the analogous quantities defined in the flow case due to the new definition of $\CC_t$, recalling Proposition \ref{prop:constant d t}, it is immediate to verify that $\delta_s\equiv\xi_s$ for $s\in\{0,\ldots,t-2\}$ and $\frac{d-1}{(d-2)(1+h_{t+1})} b_{d,t}\equiv\xi_{t-1}$, so that
$
	\bar \LL^{-1}_{\tilde\ww \vv}
	= \frac{1}{\omega d (d-1)^{t}} \frac{1}{b_{d,t}}.
$
Theorem \ref{thm:errorcharacterization voltage} yields
\begin{align*}
	\nu^\star_{\tilde w} - \hat \nu^t_{\tilde w}
	&=
	\sum_{\vv\in\VV^{t-1}} 
	\bar\LL^{-1}_{\tilde\ww\vv}
	\sum_{\ww\in\VV^{t}} 
	\WW_{\vv\ww} \nu^\star_{\sigma(\ww)}
	=
	\frac{1}{d (d-1)^{t}} \frac{1}{b_{d,t}}
	\sum_{v\in V} \nu^\star_{v} \,\#^{t+1}_{\tilde w \rightarrow v}.
\end{align*}
where $\#^{t+1}_{\tilde w \rightarrow v}$ is the number of non-backtracking paths in the original undirected graph $G=(V,E)$ that connect node $\tilde w$ to node $v$ in $t+1$ time steps. The proof is easily concluded noticing that the total number of non-backtracking paths after $t+1$ steps is $d(d-1)^{t}$, so that $\vec{\mathbf{P}}_{\tilde w}(Y_{t+1}=v)=\#^{t+1}_{\tilde w \rightarrow v}/(d(d-1)^{t})$.
\end{proof}

\section*{Acknowledgments}
We would like to acknowledge support for this project from the National Science Foundation (NSF grant ECCS-1609484).

\appendix

\section{Inverse of Restricted Laplacians and Hitting Times}\label{sec:Laplacians and random walks}
In this appendix we establish a standalone general connection between the inverse of restricted Laplacian matrices---which are obtained by removing certain rows and columns from Laplacians matrices---and hitting times of random walks on graphs. This connection (Proposition \ref{prop:reducedlapandgreenfunction} below) is used repeatedly in Section \ref{sec:proofs} to prove the error characterization of the min-sum algorithm.

Let $G=(V,E,W)$ be a simple (i.e., no self-loops, and no multiple edges), connected, undirected, weighted graph, where to each edge $\{v,w\}\in E$ is associated a positive weight $W_{vw}=W_{wv}>0$, and $W_{vw}=0$ if $\{v,w\}\not\in E$. Let $D$ be a diagonal matrix with entries $d_v=D_{vv}=\sum_{w\in V} W_{vw}$ for each $v\in V$.
Let $L:=D-W$ be the Laplacian matrix for $G$.
Henceforth, for each $v\in V$, let $\mathbf{P}_{v}$ be the law of a time homogeneous Markov chain $X_{0},X_{1},X_2,\ldots$ on $V$ with transition matrix $P=D^{-1}W$ and initial condition $X_0=v$. Analogously, denote by $\mathbf{E}_{v}$ the expectation with respect to this law. The hitting time to a set $Z\subseteq V$ is defined as
$
	T_Z := \inf\{k\ge0 : X_k \in Z\}.
$
Let $Z\subseteq V$ be fixed, and define $\bar W$ and $\bar D$ as the matrix obtained by removing the rows and columns associated to $Z$ from $W$ and $D$, respectively. Let $\bar V:=V\setminus Z$, $\bar E:=E\setminus \{\{u,v\}\in E: u\in Z \text{ or } v\in Z\}$ and consider the graph $\bar G:=(\bar V,\bar E)$.
The quantity $\bar L:=\bar D-\bar W$ represents a restricted Laplacian---sometimes also called grounded Laplacian. Let $\bar P:=\bar D^{-1}\bar W$ be the transition matrix of the transient part of the killed random walk that is obtained from $X$ by adding cemeteries at the sites in $Z$. Creating cemeteries at the sites in $Z$ means modifying the walk $X$ so that each site in $Z$ becomes a recurrent state, i.e., once the walk is in state $\bar z\in Z$ it will go back to $\bar z$ with probably $1$. This is clearly done by replacing the $\bar z$-th row of $P$ by a row with zeros everywhere but in the $\bar z$-th coordinate, where the entry is equal to $1$.

The relation between the transition matrix $\bar P$ of the killed random walk and the law of the random walk $X$ itself is made explicit in the next proposition.

\begin{proposition}\label{prop:killedrv}
For any $v,w\in \bar V$, $k\ge 0$, we have
$
	\bar P^k_{vw}
	=
	\mathbf{P}_v(X_k=w,T_{Z}>k).
$
\end{proposition}

\begin{proof}
Proof is by induction. Clearly, for any $v,w\in \bar V$, we have
$
	\mathbf{P}_v(X_0=w,T_{Z}>0) = \mathbf{P}_v(X_0=w)
	= \mathbf{1}_{v=w} = \bar P^0_{vw},
$
which proves the statement for $k=0$ ($\mathbf{1}_S$ is the indicator function defined as $\mathbf{1}_S = 1$ if statement $S$ is true, $\mathbf{1}_S = 0$ otherwise).
Assume that the statement holds for any time $i\ge 0$ up to $k>0$. By the properties of conditional expectation, as $\{T_{Z} >k+1\}=\{X_0\not\in Z,\ldots,X_{k+1}\not\in Z\}$, we have
$
	\mathbf{P}_v(X_{k+1}=w,T_{Z}>k+1)=
	\mathbf{E}_v [\mathbf{P}_v(X_{k+1}=w,T_{Z}>k+1 | X_0,\ldots,X_k)]=
	\mathbf{E}_v [\mathbf{1}_{\{X_0\not\in Z,\ldots,X_k\not\in Z\}}\mathbf{P}_v(X_{k+1}=w,X_{k+1}\not\in Z | X_0,\ldots,X_k)]
$
for any $v,w\in \bar V$. By the Markov property, on the event $\{X_k\not\in Z\}$, we have
$
	\mathbf{P}_v(X_{k+1}=w,X_{k+1}\not\in Z | X_0,\ldots,X_k)
	= \mathbf{P}_{X_k}(X_{1}=w,X_{1}\not\in Z)
	= \bar P_{X_kw},
$
so that by the induction hypothesis we have
$
	\mathbf{P}_v(X_{k+1}=w,T_{Z}>k+1)
	= \mathbf{E}_v [\mathbf{1}_{\{T_{Z} >k\}} \bar P_{X_kw}]
	= \sum_{u\in V\setminus Z} 
	\mathbf{P}_v(X_k=u,T_{Z}>k) \bar P_{uw}
	= \bar P^{k+1}_{vw},
$
which proves the statement for $k+1$.
\end{proof}

We can now relate the inverse of the reduced Laplacian $\bar L$ with the hitting times of the original random walk $X$. The following result represents the main computational tool used in Section \ref{sec:proofs} to establish decay of correlation in the computation tree that supports the operations of the min-sum algorithm over time.

\begin{proposition}\label{prop:reducedlapandgreenfunction}
For $v,w\in \bar V$, 
$
	\bar L^{-1}_{vw} 
	\!=\! \bar L^{-1}_{ww}\mathbf{P}_v(T_w<T_{Z}),$
and
$
	\bar L^{-1}_{ww}
	\!=\! \frac{1}{d_w} \mathbf{E}_w\bigg[\sum_{k=0}^{T_{Z}} \mathbf{1}_{X_k=w} \bigg].
$
\end{proposition}

\begin{proof}
Let us first assume that $\bar G$ is connected. The matrix $\bar P$ is clearly sub-stochastic. Then $\bar P$ is irreducible (in the sense of Markov chains, i.e., for each $v,w\in \bar V$ there exists $t$ to that $\bar P^k_{vw}\neq 0$) and the spectral radius of $\bar P$ is strictly less than $1$ (see Corollary 6.2.28 in \cite{Horn:1985:MA:5509}, for instance), so that the Neumann series $\sum_{k=0}^\infty \bar P^k$ converges.
The Neumann series expansion for $\bar L^{-1}$ yields
$
	\bar L^{-1}
	= \sum_{k=0}^\infty (I-\bar D^{-1} \bar L)^k \bar D^{-1}
	= \sum_{k=0}^\infty \bar P^k \bar D^{-1},
$
or, $\bar L^{-1}_{vw} = \frac{1}{d_w} \sum_{k=0}^\infty \bar P^k_{vw}$. As $\bar P^k_{vw}=\mathbf{P}_v(X_k=w,T_{Z}>k)$ by Proposition \ref{prop:killedrv}, by the Monotone convergence theorem we can take the summation inside the expectation and get
$
	\sum_{k=0}^\infty \bar P^k_{vw} 
	= \sum_{k=0}^\infty \mathbf{E}_v[\mathbf{1}_{X_k=w}\mathbf{1}_{T_{Z}>k}]
	= \mathbf{E}_v
	[\sum_{k=0}^{T_{Z}-1} \mathbf{1}_{X_k=w}],
$
which equals $\mathbf{E}_v [\sum_{k=0}^{T_{Z}} \mathbf{1}_{X_k=w}]$ as $X_{T_{Z}}\in Z$, and $w\not\in Z$. Recall that if $S$ is a stopping time for the Markov chain $X:=X_0,X_1,X_2,\ldots$, then by the strong Markov property we know that, conditionally on $\{S<\infty\}$ and $\{X_S=w\}$, the chain $X_S,X_{S+1},X_{S+2},\ldots$ has the same law as the Markov chain $Y:=Y_{0},Y_1,Y_2,\ldots$ with transition matrix $P$ and initial condition $Y_0=w$, and $Y$ is independent of $X_0,\ldots,X_S$. The hitting times $T_w$ and $T_{Z}$ are two stopping times for $X$, and so is their minimum $S:=T_w\wedge T_{Z}$. As either $X_S=w$ or $X_S\in Z$, we have
$
	\mathbf{E}_v[\sum_{k=0}^{T_{Z}} \mathbf{1}_{X_k=w}]
	= \mathbf{E}_v[\sum_{k=0}^{T_{Z}} \mathbf{1}_{X_k=w} \vert X_S=w]
	\mathbf{P}_v(X_S=w),
$
where we used that, conditionally on $\{X_S\in Z\}=\{T_w >k_{Z}\}$, clearly $\sum_{k=0}^{T_{Z}} \mathbf{1}_{X_k=w} = 0$. Conditionally on $\{X_S=w\}=\{T_w < T_{Z}\}=\{S=T_w\}$, we have $T_{Z}=S + \inf\{k\ge 0:X_{S+k}\in Z\}$, and the strong Markov property yields (note that the event $\{S<\infty\}$ has probability one from any starting point, as the graph $G$ is connected by assumption so that the Markov chain will almost surely eventually hit either $w$ or a site in $Z$)
$
	\mathbf{E}_v[\sum_{k=0}^{T_{Z}} \mathbf{1}_{X_k=w} \vert X_S=w]
	=
	\mathbf{E}_v[\sum_{k=0}^{\inf\{k\ge 0:X_{S+k}\in Z\}} 
	\mathbf{1}_{X_{S+k}=w} \vert X_S=w]
	= 
	\mathbf{E}_w[\sum_{k=0}^{T_{Z}} \mathbf{1}_{X_k=w}].
$
Putting everything together we have
$
	\bar L^{-1}_{vw}
	= \frac{1}{d_w} \mathbf{E}_w[\sum_{k=0}^{T_{Z}} \mathbf{1}_{X_k=w} ]
	\mathbf{P}_v(T_w<T_{Z}).
$
As $\mathbf{P}_w(T_w<T_{Z})=1$, clearly
$
	\bar L^{-1}_{ww}
	= \frac{1}{d_w} \mathbf{E}_w[\sum_{k=0}^{T_{Z}} \mathbf{1}_{X_k=w} ]
$
so that
$
	\bar L^{-1}_{vw}
	= \bar L^{-1}_{ww}
	\mathbf{P}_v(T_w<T_{Z}).
$
The argument just presented extends easily to the case when $\bar G$ is not connected. In this case the matrix $\bar P$ has a block structure, where each block corresponds to a connected component and to a sub-stochastic submatrix, so the argument above can be applied to each block.
\end{proof}

\section{Technical Results}\label{app:technical}

This appendix contains technical results that are used in the paper.
The following result on quadratic optimization has been used repeatedly in the main text (see Section \ref{sec:Voltage and flow problem} and the proofs of Proposition \ref{prop:Min-Sum updates} and Lemma \ref{lem:corr tree}).

\begin{proposition}\label{prop:quadratic}
Given $R\in \R^{m\times m}$ symmetric positive definite, $r\in \R^m$, $A\in\R^{n\times m}$, $b\in\R^{n}$ in the range of $A$, the unique solution of the optimization problem
\begin{align*}
\begin{aligned}
	\text{minimize }\quad   & \frac{1}{2} x^T R x + r^T x\\
	\text{subject to }\quad & Ax = b
\end{aligned}
\end{align*}
over $x\in\R^m$ reads
$
	x^\star = R^{-1} A^T L^+ b + (R^{-1}A^TL^+ A-I) R^{-1} r
$
with $L:= A R^{-1} A^T$.
\end{proposition}

\begin{proof}
The KKT optimality equations for this problem read
$$
	\left( \begin{array}{cc}
	R & -A^T \\
	A & 0
	\end{array} \right)
	\left( \begin{array}{c}
	x\\
	\nu
	\end{array} \right)
	=
	\left( \begin{array}{c}
	-r\\
	b
	\end{array} \right).
$$
Solving the first equation for $x$ we have $x=R^{-1}(A^T\nu - r)$, which, upon substitution into the second equation $Ax=b$ yields a solution $\nu=L^+(AR^{-1}r+b)$, and the proof follows.
\end{proof}

We now present the proof of Proposition \ref{prop:constant d t} in Section \ref{sec:results}.\\

\begin{proof}[Proof of Proposition \ref{prop:constant d t}]
Fix $d\ge 3$, $t\ge 3$. As $s\rightarrow h_s$ monotonically converges to $0$ exponentially fast, we expect that the function $s\rightarrow\delta_s$, hence the values of both $b_{d,t}$ and $c_{d,t}$, can be controlled by the system that is obtained by replacing $h_s$ with $0$ for any $s\ge 0$. Thus, we define $\hat\delta_0 := \delta_0$, $\hat\delta_1 := \delta_1$, and 
$
	\hat\delta_s := 
	\frac{1}{d-1}(2+\frac{(d-2)^2}{d-1}) \hat\delta_{s-1} - \frac{1}{(d-1)^2}\hat\delta_{s-2},
	s\ge 2.
$
The characteristic equation of the dynamical system $\hat\delta_s$, $s\ge 0$, reads $(d-1)^2\lambda^{2}-(d^2-2d+2)\lambda + 1 = 0$. This equation has two distinct real solutions: $\lambda_+=1$ and $\lambda_-=1/(d-1)^2 < 1$. Imposing the initial conditions for $\hat\delta_0$ and $\hat\delta_1$, we find
$
	\hat\delta_s = c_+ + c_- \lambda_-^s \ge 0
$
with $c_+=\frac{d(d-1)((d-1)^2+1)-1}{d((d-1)^2-1)}>0$ and $c_-=1/d-c_+<0$. 

We begin by analyzing $c_{d,t}$. To this end, define
$
	\hat c_{d,t} := 
	\frac{1}{d-1}(1+\frac{1}{(d-2)(1+h_{t})}) \hat \delta_{t-2}
	- \frac{1}{(d-1)^2} \hat \delta_{t-3}.
$
First, we prove that $c_{d,t}\ge 1$. 
By induction it is easy to show that $(d-1)(\delta_s-\hat\delta_s)\ge\delta_{s-1}-\hat\delta_{s-1}\ge 0$ for any $s\ge 1$, so we have $c_{d,t}\ge \hat c_{d,t}$.
As $\hat\delta_s \ge c_+ + c_{-} \lambda_{-}$ for any $s\ge 1$ and $\hat\delta_s \le c_+$ for any $s\ge 0$, we get $\hat c_{d,t}\ge \frac{1}{d-1}(1+\frac{1}{(d-2)(1+h_{1})} ) (c_+ + c_- \lambda_-) - \frac{1}{(d-1)^2}c_+ \ge 1$.
Second, we prove that $c_{d,t}\le 1+\varepsilon_d$.
Given $\alpha := 1+\frac{1}{(d-1)^2}$, we first show that $(d-1)(\alpha\hat\delta_s-\delta_s) \ge \alpha\hat\delta_{s-1}-\delta_{s-1}\ge 0$ for any $s\ge 1$, so that $c_{d,t}\le \alpha \hat c_{d,t}$. Define $\xi_s:=\alpha\hat\delta_{s}-\delta_{s}$.
We proceed by induction. Note that $(d-1) \xi_1 \ge \xi_0 \ge 0$.
For a given $s\ge 3$, assume that $(d-1) \xi_{r} \ge \xi_{r-1}$ for all $r\le s-1$. 
We will prove that $(d-1)\xi_{s} \ge \xi_{s-1}$. 
We have
$
	(d-1)\xi_s - \xi_{s-1} = 
	\frac{(d-1)\xi_{s-1} - \xi_{s-2}}{d-1}
	+ \frac{(d-2)^2}{d-1} ((1+h_{s+2})\xi_{s-1} - \alpha h_{s+2} \hat\delta_{s-1}).
$
As $\hat\delta_{s-1}\le c_+$ and $(d-1)^{s-1} h_{s+2} \le h_3$, with $c_+ h_3 < 1$, by the induction hypothesis we have $\xi_{s-1}\ge \frac{\xi_1}{(d-1)^{s-2}}= \frac{(\alpha-1)\delta_1}{(d-1)^{s-2}}$ and
$
	(d-1)\xi_s - \xi_{s-1}
	\ge
	\frac{(d-2)^2}{(d-1)^{s}} ((d-1)(\alpha-1)\delta_1 - \alpha) = 0,
$
which proves the induction step.
As $\frac{1}{d} = c_+ + c_{-} \le \hat\delta_s \le c_+$, we get
$
	c_{d,t} \le \alpha\hat c_{d,t} 
	\le \bar\alpha := \alpha (\frac{c_+}{d-1}(1+\frac{1}{d-2}) 
	- \frac{1}{d(d-1)^2})
	).
$
The function $d\ge 3\rightarrow \varepsilon_d := \bar\alpha - 1$ is decreasing, with $\varepsilon_d\rightarrow 0$ as $d\rightarrow \infty$, and $\varepsilon_3<3$.

By induction it is immediate to show that $(d-1)\delta_s\ge\delta_{s-1}\ge 0$ for any $s\ge 1$, so that a direct comparison yields $b_{d,t} \le c_{d,t}$. Finally, we prove that $b_{d,t}\ge \frac{d-2}{d-1}$. Define
$
	\hat b_{d,t} := 
	\frac{1}{(d-1)^2}(1+(d-2)(1+h_{t+1})) \hat\delta_{t-2}
	- \frac{(d-2)(1+h_{t+1})}{(d-1)^3} \hat\delta_{t-3}.
$
As $(d-1)(\delta_s-\hat\delta_s)\ge\delta_{s-1}-\hat\delta_{s-1}\ge 0$ for $s\ge 1$, we have $b_{d,t}\ge \hat b_{d,t}$.
As $\hat\delta_s \ge \hat\delta_3 = c_+ + c_{-} \lambda_{-}^3$ for $s\ge 3$ and $\hat\delta_s \le c_+$ for $s\ge 0$, by the monotonicity of $h_s$ we get 
$
	\hat b_{d,t} 
	\ge 
	\frac{1}{(d-1)^2}(1+(d-2)
	) 
	(c_+ + c_- \lambda_-^3)
	- \frac{(d-2)
	(1+h_{1})
	}{(d-1)^3} c_+
	\ge \frac{d-2}{d-1}.
$
\end{proof}

The next proposition yields a recursive formula to compute, in the case of $d$-regular graphs, the probability distributions $P^{(t)}_v$ and $P^{(t,w)}_v$ (interpreted as row vectors) given by the $v$-th row of the matrices $P^{(t)}$ and $P^{(t,w)}$ defined in \eqref{prob distrib voltage flow}, respectively. For any $v \in V$, let $\mathbf{1}_v\in\R^V$ be defined as $(\mathbf{1}_v)_z := 1$ if $v=z$, $(\mathbf{1}_v)_z := 0$ otherwise.

\begin{proposition}\label{proof:Non-backtracking random walks}
Let $(V, E)$ be a connected $d$-regular graph with $d\ge 2$, and let $B\in\mathbb{R}^{V\times V}$ be its adjacency matrix, that is, the symmetric matrix defined by $B_{vw}:=1$ if $\{v,w\}\in E$, $B_{vw}:=0$ otherwise. For any $v\in V$ we have
\begin{align*}
	P^{(1)}_{v} &= \frac{\mathbf{1}_v^TB}{d},\quad
	P^{(2)}_{v} = \frac{P^{(1)}_{v}B - \mathbf{1}_v^T}{d-1},\quad
	P^{(t)}_v = \frac{P^{(t-1)}_vB - P^{(t-2)}_v}{d-1}\ t\ge 3,
\end{align*}
Given $w\in V$ such that $\{v,w\}\in E$, we have that $P^{(t,w)}_{v}$ satisfies the same recursion as above with initial condition $P^{(1,w)}_{v} = (\mathbf{1}_v^TB - \mathbf{1}_w^T)/(d-1)$.
\end{proposition}

\begin{proof}
We first prove the recursion for $P^{(t,w)}_{v}$. Fix $\{v,w\}\in E$, and let $\widetilde\#^{t}_{z}$ be the number of non-backtracking paths that connect node $v$ to node $z$ in $t$ steps, where the first step is different from $w$.
By definition, $B_{vz}$ equals the number of paths that connect node $v$ to node $z$ in a single step. Clearly, $\widetilde\#^{1}_{z} = B_{vz} - (\mathbf{1}_w)_z$, i.e., in row-vector form, $\widetilde\#^{1} = \mathbf{1}_v^T B - \mathbf{1}_w^T$. It is immediate to verify that $\widetilde\#^{2}= \widetilde\#^{1}B - (d-1) \mathbf{1}_v^T$, as there are $d-1$ paths that can backtrack to $v$ at time $t=2$. For $t\ge 3$ we have $\widetilde\#^{t}= \widetilde\#^{t-1}B - (d-1) \widetilde\#^{t-2}$. In fact, $(\widetilde\#^{t-1}B)_z$ equals the number of paths that go from $v$ to $z$ in $t$ steps such that the first move is different from $w$ and the first $t-1$ steps are non-backtracking. To count only the paths that are non-backtracking for all $t$ steps, we need to subtract from $(\widetilde\#^{t-1}B)_z$ the number of paths that backtrack at the $t$-th step, which are the ones that are at $z$ both at step $t-2$ ant $t$. There are $(d-1) \widetilde\#^{t-2}_z$ of these paths, as vertex $z$ has $d$ neighbors but only $d-1$ moves are allowed at the $(t-1)$-th step due to the non-backtracking nature of the paths up to step $t-1$. The proof follows noticing that the total number of non-backtracking after $t$ steps is $d(d-1)^{t-1}$, so that $\vec{\mathbf{P}}_v(Y_1\neq w,Y_t=z)=\widetilde\#^{t}_z/(d(d-1)^{t-1})$ and $P^{(t,w)}_{vz} = \vec{\mathbf{P}}_v(Y_t=z|Y_1\neq w) = \vec{\mathbf{P}}_v(Y_1\neq w,Y_t=z) / \vec{\mathbf{P}}_v(Y_1\neq w) = \widetilde\#^{t}_z / (d-1)^t$, $\vec{\mathbf{P}}_v(Y_1\neq w)=(d-1)/d$.
The recursion for $P^{(t)}_{v}$ is proved analogously. Let $\#^{t}_{z}$ be the number of non-backtracking paths that connect node $v$ to node $z$ in $t$ steps. Clearly, $\#^{1} = \mathbf{1}_v^TB$, $\#^{2}= \#^{1}B - d \mathbf{1}_v^T$, and $\#^{t}= \#^{t-1}B - (d-1) \#^{t-2}$. The proof follows as $P^{(t)}_{vz} = \vec{\mathbf{P}}_v(Y_t=z) = \#^{t}_z / (d(d-1)^{t-1})$.
\end{proof}

\bibliography{bib}

\end{document}